\documentclass[reqno]{amsart}
\input{common.tex}
\usepackage{float}
\usepackage{multirow}
\numberwithin{theorem}{section}
\numberwithin{equation}{section}
\usepackage{caption}
\usepackage{makecell}
\begin{document}
\title{$L^4$-norms of automorphic forms in the depth aspect}
\author{Marius Fischer}
\email{marius.fischer@math.au.dk}
\address{Department of Mathematics, Aarhus University, 1530-432, DK-8000 Aarhus C, Denmark}
\thanks{The author is supported by grant VIL54509 from Villum Fonden.}

\begin{abstract}
Let $p$ be an odd prime, and suppose $f$ is an $L^2$-normalised newform for $\Gamma_0(p^n)$ with bounded spectral parameters and trivial central character. We prove the optimal $L^4$-norm bound $\lVert f \rVert_4 \ll_{p,\varepsilon}(p^{n})^\varepsilon $ for all $\varepsilon >0$ as $n \rightarrow \infty$.
\end{abstract} 

\maketitle

\section{Introduction}

A major problem in harmonic analysis is to understand how Laplace eigenfunctions distribute as their parameters tend to infinity. A common approach to this problem is to bound their $L^p$-norms for $2<p \leq \infty$ while keeping the $L^2$-norm fixed. On arithmetic hyperbolic surfaces, much progress has been made in this direction, and a recent breakthrough of Ki \cite{Ki2023L4norms} established the optimal $L^4$-norm bound $\lVert f \rVert_4 \ll_\varepsilon t_f^{\varepsilon}$ for all $\varepsilon>0$ when $f$ is a Hecke-Maa{\ss} cusp form on $\mathrm{SL}_2(\mathbb{Z})$ of weight $0$ with Laplace eigenvalue $\frac{1}{4} + t_f^2$ tending to infinity.

There are strong analogies between the spectral limit $t_f \rightarrow \infty$ and the \emph{depth aspect}, where one considers newforms for $\Gamma_0(p^n)$ with $p$ a fixed prime and $n$ tending to infinity (see for instance \cite{Nelson2018,Hu_Saha_2020}). In this paper we establish a depth aspect analogue of Ki's result. Let $f$ be a newform for $\Gamma_0(p^n)$ with bounded spectral parameters (i.e. weight and Laplace eigenvalue) and trivial central character. We assume that $f$ is $L^2$-normalised with respect to the standard probability measure on $\Gamma_0(p^n)\backslash \mathbb{H}$ and prove the following theorem:

\begin{theorem}\label{main_theorem}
Let $p$ be an odd prime, and suppose $f$ is an $L^2$-normalised newform for $\Gamma_0(p^n)$ with bounded spectral parameters and trivial central character. Then, for any $\varepsilon >0$, the bound $\lVert f \rVert_4 \ll_{p,\varepsilon}(p^n)^\varepsilon$ holds.
\end{theorem}

\noindent
Interpolating using the $L^2$-normalisation $\lVert f \rVert_2 =1$ and the sup-norm bound $\lVert f \rVert _\infty \ll_\varepsilon(p^n)^{\frac{1}{4} +\varepsilon}$ due to Saha \cite[Theorem 3.2]{saha}, we deduce the $L^q$-norm bound
\begin{equation*}
\lVert f \rVert_q \ll_{p,\varepsilon}(p^n)^{\delta(q) + \varepsilon} \quad \textrm{where} \quad \delta(q):=
\begin{cases}
  0 & \textrm{if }2 \leq q \leq 4\\
  \frac{1}{4}-\frac{1}{q} & \textrm{if }4 \leq q \leq \infty
\end{cases}.
\end{equation*}
By recent work of Assing and Toma \cite{assing_toma_2024}, we have the stronger bound $\lVert f \rVert _\infty \ll_\varepsilon(p^n)^{\frac{5}{24}+\varepsilon}$ if $n$ is divisible by $4$, which gives $\delta(q)=5/24-5/6q$, when $q \geq 4$, and $n$ traverses the positive multiples of $4$.

We will see in the proof of Theorem \ref{main_theorem} that the bound $\lVert f \rVert_4 \ll_{p,\varepsilon}(p^n)^\varepsilon$ can be explicated as $\lVert f \rVert_4 \leq C_\varepsilon p^a(p^n)^\varepsilon$, where $C_\varepsilon$ only depends on $\varepsilon$ and the spectral parameters of $f$, and $a$ is some absolute constant (one can take $a  =10$). Hence it is also possible to deduce non-trivial horizontal estimates for $\lVert f \rVert_4$ when $p \rightarrow \infty$, and $n$ is some large, fixed integer.\\ 

\noindent
Prior to Ki's work, the $L^4$-norm problem was typically addressed \cite{humphries-khan,luo,blomer,buttcane_khan,Buttcane_Khan_2017} using Watson's formula \cite{Watson2008Rankin}, which relates $\lVert f \rVert_4$ to triple product $L$-functions. However, this approach has led to the optimal bound only in the case of dihedral Maa{\ss} forms \cite{luo} and truncated Eisenstein series \cite{spinu2003,humphries2018}.

Ki begins more directly by integrating over a Siegel domain and applying Parseval's identity and fourth moment bounds for Hecke eigenvalues (see \cite[Proposition 2.1]{Ki2023L4norms}). In this way, the problem reduces to estimating integrals of four-fold products of archimedean Bessel functions. Analysing these integrals is the main step in the argument and requires uniform asymptotic formulas for the Bessel functions involved. For an exposition of Ki's proof see \cite{Nelson-notes}.

In our proof of Theorem \ref{main_theorem}, we start from the adelic Fourier expansion of the $\Gamma_0(p^n)$-newform $f$, and by analogous arguments we reduce the problem to understanding integrals of the form
\begin{equation}\label{key_integrals}
\int_{t \in \mathcal{O}^\times} W(m_1 t) W(m_2 t) \overline{W(m_3 t) W(m_4 t) }\, d^\times t,
\end{equation}
where $W$ is a translate of the $p$-adic Whittaker newvector, and $m_1$, $m_2$, $m_3$ and $m_4$ are positive integers satisfying $m_1 + m_2 = m_3 + m_4$. While this reduction is similar to Ki's argument, estimating \eqref{key_integrals} poses two new challenges:\\

\begin{enumerate}[(i)]
\item Ki's argument uses classical asymptotic expansions for archimedean Bessel functions (see for example \cite[Lemma 3.1]{Ghosh_2013}), but in our case, an analogous description of the newvector is not available. Starting from the PhD-thesis of Edgar Assing \cite{assing_thesis, assing_on_the_size}, we derive a description of the $p$-adic Whittaker newvector of $f$ that is valid uniformly in the local representation at $p$ generated by $f$. This derivation requires an intricate analysis of the $p$-adic Airy function and more general oscillatory integrals with cubic phase. The formulas for the Whittaker function can be found in Theorem \ref{whittaker_formulas}, and an overview is given in Table \ref{heuristic} below (and Table \ref{heuristic_balanced} for the balanced variant). As we explain below, our formulas reveal a strong analogy with the archimedean case, and may be of independent interest.\\
  
\item To estimate the archimedean analogue of \eqref{key_integrals}, Ki detects cancellation using a criterion on the sign of the second derivative of the phase function for which we found no clear $p$-adic analogue. Instead, a careful calculation reveals that the phase function can be expressed in terms of a power series that linearises after a suitable substitution. Using this substitution, we can determine under which conditions \eqref{key_integrals} vanishes in terms of the integers $m_1$, $m_2$, $m_3$ and $m_4$.\\
\end{enumerate}

To illustrate the analogy alluded to in (i), we first recall the archimedean setting. Let $W_t : \mathrm{GL}_2(\mathbb{R}) \rightarrow \mathbb{R}$ denote the archimedean Whittaker function of a weight $0$ Maa{\ss} form with Laplace eigenvalue $\frac{1}{4} + t^2$. This function is right-invariant under $\mathrm{SO}(2)$ and transforms via $e^{2 \pi i x}$ under left-translation by $\left(
\begin{smallmatrix}
1&x\\
0&1 \\
\end{smallmatrix}
\right)$. By the Iwasawa decomposition of $\mathrm{GL}_2(\mathbb{R})$, $W_t(y)$ is completely determined by its values on the diagonal matrices $a(y) := \left(
\begin{smallmatrix}
y&0\\
0&1 \\
\end{smallmatrix}
\right)$ for $y>0$. Writing $W_t(y)$ for $W_t(a(y))$, it is well-known that $W_t(y) = \sqrt{y} e^{\frac{\pi}{2}t}K_{it}(y)$, where $K_{it}(y)$ is the $K$-Bessel function. Defining the transition parameter $u:= y-t$ and the phase function $H(\xi)$ as in \cite[Equation (19)]{Ghosh_2013}, we have the following description of $W_t(y)$ and its basic upper bounds that fall into three ranges.\\

\noindent
\textbf{Oscillatory range:} If $u<0$ and $u \gg   t^{\frac{1}{3}}$, then
\begin{equation*}
W_t(y) = \frac{\sqrt{2 \pi y}}{(t^2 -y^2 )^{\frac{1}{4}}} \mathrm{sin}\left( \frac{\pi}{4}+t H \left( \frac{y}{t} \right) \right) \left( 1 + \O\left(\frac{1}{t H \left( \tfrac{y}{t} \right)}\right) \right) \ll
\begin{cases}
  t^{-\frac{1}{2}}y^{\frac{1}{2}} & u \asymp t\\
  t^{\frac{1}{4}}u^{-\frac{1}{4}} & u = o(t)
\end{cases}.
\end{equation*}
\textbf{Transition range:} If $u \ll t^{\frac{1}{3}}$ and $\mathrm{Ai}$ denotes the Airy function, then
\begin{equation*}
W_t(y) = \pi \sqrt{y} \left( \frac{2}{t} \right)^{\frac{1}{3}}\mathrm{Ai} \left((y-t) \left( \frac{y}{2} \right)^{-\frac{1}{3}}  \right)+\O(t^{-\frac{2}{3}}) \ll t^{1/6}.
\end{equation*}
\textbf{Exponential decay:} If $u>0$ and $u \gg t^{\frac{1}{3}}$, then
\begin{equation*}
W_t(y) = \frac{1}{2} \frac{\sqrt{2 \pi y}}{(y^2 -t^2)^{\frac{1}{4}}} e^{-tH \left( \frac{y}{t} \right)}\left( 1 + \O\left(\frac{1}{t H \left( \tfrac{y}{t} \right)}\right) \right) \ll t^{\frac{1}{4}} u^{-\frac{1}{4}} e^{-t H ( \frac{y}{t})}.
\end{equation*}
The formulas are \emph{uniform} meaning that all implied constants are absolute, which makes them an important tool in analytic questions related to automorphic forms. In addition to Ki's $L^4$-norm bound, they are used in the study of sup-norm problems \cite{BlomerHolowinsky2010}, nodal domains \cite{Ghosh_2013}, distribution of zeros \cite{LesterMatomakiRadziwill2018} and restriction problems \cite{BlomerCorbett2021}.\\

Now, let $F$ be a non-archimedean local field of characteristic $0$ and odd residue characteristic $p$. Let $\mathcal{O}$ denote the ring of integers, $\varpi$ a uniformiser and $q$ the cardinality of the residue field. Suppose $\pi$ is an infinite dimensional, irreducible, admissible unitary representation of $\mathrm{GL}_2(F)$ having trivial central character. Let $W_\pi: \mathrm{GL}_2(F) \rightarrow \mathbb{C}$ be the Whittaker newvector of $\pi$ normalised to take the value $1$ at the identity. In analogy with the archimedean case, $W_\pi$ is right-invariant under the compact part of a torus, namely the group of diagonal units $\left(
\begin{smallmatrix}
u_1&0\\
0&u_2 \\
\end{smallmatrix}
\right)$ where $u_1 , u_2 \in \mathcal{O}^\times$. In Proposition \ref{representatives}, we show how $W_\pi$ is determined by its value on the matrices $a(y) \kappa$, where $y \in F^\times$ and $\kappa$ traverses a certain set of representatives (see the table below). The analogy with archimedean case is most clearly seen when $\kappa = \left(
\begin{smallmatrix}
1&0\\
\varpi^{n/2}&1 \\
\end{smallmatrix}
\right)$, where $n$ is the conductor exponent of $\pi$. Here the analogue of the spectral parameter $t$ is $q^{\frac{n}{2}}$, and we define $\Delta= \Delta(y) := 1+ 4 b^2 y^{-2}$ for $y \in \mathcal{O}$, where $b \in \mathcal{O}^\times$ is a constant associated to $\pi$ (see Section \ref{formulas_for_the_whittaker_function}). In the archimedean setting, $\lvert y \rvert$ and $\lvert \Delta \rvert$ correspond to the ratios $y/t$ and $ u/t$ respectively. As shown in the table below, $W_\pi(a(y)\kappa)$ mimics the behaviour of $W_t(y)$: It exhibits an oscillatory range and a transition (Airy) range, and the exponential decay is reflected by vanishing.

In the oscillatory ranges, the Whittaker function is given by a combination of additive and multiplicative characters. Typically, this means that $W_\pi(a(y) \kappa)$ is a linear combination of at most two expressions of the type
\begin{equation*}
\xi^{-1}(x+b)\psi((x+y) \varpi^{-n/2})
\end{equation*}
where $\xi$ is a character attached to the unit group of a quadratic space over $F$ (see Section \ref{quadratic_spaces}), $b$ is some unit depending only on $\pi$, and $x$ satisfies the quadratic equation $x^2 + yx - b^2 =0$. In the Airy range, $W_\pi(a(y) \kappa)$ is expressed as an oscillatory integral with cubic phase.

Moreover, the sizes of these ranges and the corresponding upper bounds on $W_\pi(a(y) \kappa)$ match the archimedean case exactly. For precise formulas for $W_\pi$, we refer to Theorem \ref{whittaker_formulas}.


\begin{table}[H]
\centering 
\renewcommand{\arraystretch}{2}
\caption{The behaviour of the newvector on the matrices $a(y) \kappa$.}\label{heuristic}
\begin{tabular}{|c|c|c|c|c|c|}
\hline
$\kappa$                        & $\gamma$                       & \multicolumn{2}{|c|}{Support}                                       & Type         & Max          \\
\hline
$\left(
		\begin{smallmatrix}
                  0&1\\
                  
			1&\varpi^\gamma \\
		\end{smallmatrix}
		\right)$                   & $1\leq \gamma \leq \infty$             & \multicolumn{2}{|c|}{$\lvert y \rvert=q^n$}                                       & Const.     & $\O\left(1\right)$       \\
\hline
\multirow{4}{*}{$\left(
			\begin{smallmatrix}
				1&0\\
				\varpi^\gamma&1 \\
			\end{smallmatrix}\right)$}  & $0\leq \gamma < \frac{n}{2}$             &\multicolumn{2}{|c|}{$\lvert y \rvert=q^{n-2 \gamma}$}                                     & Osc.  & $\O\left(1\right)$        \\ \cline{2-6}
		
                         & \multirow{2}{*}{$\gamma= \frac{n}{2}$} &\multicolumn{2}{|c|}{$\pi$ principal series, $0<|y|<1$}                        & Osc.  & $\O(|y|^{\frac{1}{2}})$     \\ \cline{3-6}
		
                           &                          &\multirow{2}{*}{\makecell{$\lvert y\rvert=1$, and $\Delta \in (\mathcal{O})^2$\\ if $q^{-\frac{n}{4}} < \lvert \Delta \rvert \leq 1$ }}                 & $q^{-\frac{n}{3}} \ll \lvert \Delta \rvert \leq 1$       & Osc.  & $\O (\lvert \Delta \rvert^{-\frac{1}{4}})$        \\ \cline{4-6}
  
                           &                          &                                                 & $\lvert \Delta \rvert \ll q^{-\frac{n}{3}}$   & Airy          & $\O\left(q^{\frac{n}{12}}\right)$      \\ \cline{2-6}
  
                           & $\frac{n}{2}<\gamma < n$               & \multicolumn{2}{|c|}{$\lvert y \rvert=1$}                                        & Osc.    & $\O\left(1\right)$      \\ \cline{2-6}
  
                           & $n\leq \gamma \leq \infty$              & \multicolumn{2}{|c|}{$|y|=1$}                                        & Const.       & $\O\left(1\right)$       \\  
\hline
\end{tabular}
\end{table}


As explained above, our proof of Theorem \ref{main_theorem} further strengthens the analogy between Maa{\ss} forms of large eigenvalue and newforms of large depth. One of the first results that manifested this analogy is the sup-norm bound $\lVert f \rVert _\infty \ll_\varepsilon(p^n)^{\frac{5}{24}+\varepsilon}$ due to Hu and Saha \cite{Hu_Saha_2020} where $f$ is a newform of level $p^n$ for an indefinite quaternion division algebra over $\mathbb{Q}$. This bound is the depth aspect analogue of the bound $\lVert f \rVert _\infty \ll_\varepsilon t_f^{\frac{5}{24}+\varepsilon}$ in the eigenvalue aspect due to Iwaniec and Sarnak \cite{Iwaniec_Sarnak_1995}, and there are many similarities between the proofs. In our work, we not only establish the optimal $L^4$-norm bound $\lVert f \rVert_4 \ll_\varepsilon(p^n)^\varepsilon$, but also provide a detailed description of the $p$-adic Whittaker newvector that can potentially be used to study related problems in the future.

\begin{remark}
In Theorem \ref{main_theorem}, we assume that $p$ is odd because otherwise \cite{assing_thesis,assing_on_the_size} do not provide a description of the Whittaker newvector. However, \cite{assing_thesis,assing_on_the_size} gives integral representations for $W_\pi$ that are valid even when $p=2$, as long as $\pi$ is not a non-dihedral supercuspidal representation. In that case, if we also impose that $\pi$ has trivial central character, there are only $16$ possibilities for $\pi$, and the conductor exponent of $\pi$ is at most $7$ \cite[Proposition 3.9]{saha_2024}. By only considering $n \geq 8$, we believe that stationary phase arguments similar to \cite{assing_thesis,assing_on_the_size} can be used to find formulas for $W_\pi$ when $p=2$, and the condition $p\neq 2$ can be removed from Theorem \ref{main_theorem}.
\end{remark}

\begin{remark}
In order to highlight the novelties of our methods, we do not prove a hybrid bound for $\lVert f \rVert_4$ in terms of the weight and Laplace eigenvalue of $f$. Instead, we have used a pointwise bound on the archimedean Whittaker function that does not make the dependence on these parameters explicit. In the case of Maa{\ss} forms, a hybrid bound can likely be obtained by combining our work with Ki's proof.
\end{remark}

\begin{remark}
Our argument cannot be generalised to compact quotients because it relies on having a Fourier expansion and an associated fourth moment bound available. For alternative approaches, we refer to the work of Humphries-Khan \cite{humphries-khan} and Marshall \cite{marshall}.
\end{remark}

\subsection{Structure of the paper}
We now outline how the paper is structured. In Section \ref{classical_automorphic_forms} - \ref{geometric_considerations} we phrase the problem adelically and in terms of the $p$-adic newvector $W_p$ and its balanced variant $W_b$ (see Section \ref{balancing}). This reformulation will eventually reduce the problem to estimating integrals of the form
\begin{equation*}
 \int_{x \in \mathbb{A} / \mathbb{Q} } \int_{y \geq y_0} \int_{t \in \mathcal{O}^\times} \Bigg| \sum_{m\neq 0} \psi(p^{-\lceil n/2\rceil}m x) W(m t) \frac{\lambda(\lvert m \rvert')}{\sqrt{\lvert m \rvert'}} W _\infty(p^{-\lceil n/2 \rceil}m y)\Bigg|^4 d^\times t\frac{d y}{y^2}d x 
\end{equation*}
which we denote $\mathcal{N}(W)$. Here $\psi : \mathbb{A} / \mathbb{Q} \rightarrow \mathbb{C}^\times$ is an additive character, $\lambda(m)$ the Hecke eigenvalues of $f$, $m'$ the prime-to-$p$ part of $m$, and $W _\infty$ the Whittaker function at the infinite place. The function $W: \mathbb{Q}_p^\times \rightarrow \mathbb{C} $ will range over components of a dyadic decomposition of the balanced newvector $W_b$. The dyadic decomposition is explained in Section \ref{dyadic_decomposition}. To estimate $\mathcal{N}(W)$ we unfold as in Ki's paper, which introduces a sum over tuples $\mathbf{m} =(m_1,m_2 , m_3 , m_4) $ with $m_1 + m_2 = m_3 + m_4$ involving $\lambda(m_1 )^4$, the integral
\begin{equation}\label{key_integral}
I_p(\mathbf{m}) := \int_{t \in \mathcal{O}^\times} W(m_1 t) W(m_2 t) \overline{W(m_3 t) W(m_4t)} d^\times t,
\end{equation}
and some archimedean factors. Understanding the integrals $I_p(\mathbf{m})$ is the key to the proof and requires detailed formulas for the $p$-adic Whittaker function.

Before describing the Whittaker function, we collect some properties of oscillatory integrals in Section \ref{oscillatory_integrals}. In particular, we study the $p$-adic Airy function which, as in the archimedean case, shows up in a transition region of the Whittaker function. We prove explicit upper bounds which, in some ranges, improves the best bounds we found in literature.

In Section \ref{formulas_for_the_whittaker_function}, we give formulas for the $p$-adic Whittaker function in the coordinates on $\mathrm{GL}_2(\mathbb{Q}_p)$ described in Section \ref{geometric_considerations}. Explicit formulas are given in Theorem \ref{whittaker_formulas}, and a useful overview can be found in Table \ref{heuristic} (and Table \ref{heuristic_balanced} for the balanced newvector). In comparison to the archimedean Whittaker function used in Ki's argument, our description of the Whittaker function becomes slightly more involved because, as $n$ increases, the invariance of the newvector shrinks so there will be more cases to consider.

In Section \ref{dyadic_decomposition}, we dyadically decompose the balanced newvector, and in Section \ref{first_estimates_section} we explain how to estimate $\mathcal{N}(W)$ when $W$ is a component of this dyadic decomposition. This estimate gives an upper bound for $\mathcal{N}(W)$ in terms of the support of $W$, the size of $W$, and for how many tuples $\mathbf{m}$ the integral $I_p(\mathbf{m})$ vanishes. In the proof we estimate the Hecke eigenvalues using a fourth moment bound similar to the one in Ki's paper, see Equation (\ref{fourth_moment}).

In Section \ref{estimating_via_size_and_support}, we use the estimate from Section \ref{first_estimates_section} only taking into account the size and support of $W$. In Section \ref{the_strategy_for_proving_proposition_11.2} and \ref{the_proof_of_proposition_11.2} we handle the cases where size and support properties of $W$ are not sufficient. We show that in these case, $I_p(\mathbf{m})$ reduces to a linear combination of oscillatory integrals of a very specific form which allows us to prove vanishing of $I_p(\mathbf{m})$ for sufficiently many $\mathbf{m}$.

\subsection{Asymptotic notation}
Throughout the rest of the paper, we use the following assymptotic notation: The number $n$ denotes a positive integer tending to infinity, and everything is allowed to depend on $n$ unless it has been declared to be fixed. We write $A \ll B$ if $\lvert A \rvert \leq C \lvert B \rvert$ for some fixed $C$. We write $A \ll_{v_1 , ..., v_k} B$ if we want to emphasize that the implied constant depends on some parameters $v_1 ,..., v_k$. The notation $A \asymp B$ means that both $A \ll B$ and $B \ll A$ hold, and finally we write $A \prec B$ if $A \ll_{\varepsilon}(p^n)^\varepsilon B$ for all $\varepsilon >0$. Hence the aim is to prove the bound $\lVert f \rVert_4 \prec 1$.

\section{Classical automorphic forms}\label{classical_automorphic_forms}
Before adelising the problem, we collect some properties of $f$ as a classical automorphic form. We start with a word on normalisation. In the introduction, we assumed that $f$ is $L^2$-normalised, but it will be more convenient to work with the \emph{arithmetic normalisation} i.e. we assume that the first Fourier coefficient is $1$. These two normalisations only differ by the adjoint $L$-value $L(\mathrm{ad}f,1)$ which satisfies $(p^n)^{-\varepsilon} \ll L(\mathrm{ad} f,1) \ll(p^n)^\varepsilon$ \cite{iwaniec90,hoffstein-lockhart} so the aim is still to prove $\lVert f \rVert_4 \prec 1$.\\

\noindent
Let $\mathbb{H} := \left\{ z=x+iy \in \mathbb{C}\,:\, y>0 \right\}$ denote the upper half plane. If $\phi: \mathbb{H} \rightarrow \mathbb{C}$ is a function, $k \in \mathbb{Z}$, and $g \in \mathrm{GL}_2(\mathbb{R})^+ $, we define the weight $k$ action of $g$ on $\phi$ by
\begin{equation*}
\phi |_k [g](z):= \left( \frac{c \overline{z}+d}{\lvert cz +d \rvert} \right)^{-k} \phi \left( \frac{a z +b}{c z +d} \right) \quad \mathrm{if} \quad g =
\begin{pmatrix}
a&b\\
c&d \\
\end{pmatrix}.
\end{equation*}
Thus $f |_k [\gamma]=f$ for all $\gamma \in \Gamma_0(p^n)$. The weight $k$ hyperbolic Laplace operator on $\mathbb{H}$ is denoted $\Delta_k$, and we have $\Delta_k f = \nu(1- \nu) f$ for some $\nu \in \mathbb{C}$. The Fourier expansion of $f$ is given by
\begin{equation*}
f(z) =\sum_{m \neq 0} \frac{\lambda( m )}{\sqrt{\lvert m \rvert}} \mathrm{sgn}(m)^a W_{\frac{\mathrm{sgn}(m)k}{2},\nu -\frac{1}{2}}(4 \pi \lvert m \rvert y) e(m x)
\end{equation*}
where, for $m >0$, $\lambda(m)$ denotes the $m$\textsuperscript{th} Hecke eigenvalue of $f$. If $f$ is a holomorphic modular form, $\lambda(m)=0$ for $m<0$, and otherwise $\lambda(m)= \lambda(\lvert m \rvert)$. The number $a$ is $0$ or $1$ if $f$ is even or odd respectively, and we set $e(x):= e^{2 \pi i x}$.  For $\alpha \in \mathbb{R}$ and $\nu \in \mathbb{C}$, $W_{\alpha , \nu}(y)$ is the solution to Whittaker's differential equation having a specified asymptotics as $y \rightarrow \infty$, see \cite[Section 3.6]{goldfeld_hundley}. We will not write down explicit formulas for $W_{\alpha , \nu}$; In our case we use the following pointwise estimate: For any $N \geq 1$, we have
\begin{equation}\label{pointwise_whittaker_bound}
W_{\alpha, \nu}(y) \ll_{N, \alpha, \nu}
\begin{cases}
  \lvert y  \rvert^\beta & \lvert y \rvert \leq 1\\
  \lvert y \rvert^{-N} & \lvert y \rvert \geq 1
\end{cases}
\end{equation}
where $1/4<\beta <1/2$ is some absolute constant, see \cite[Proposition 3.2.3]{PMIHES_2010__111__171_0}. The fact that we can take $\beta>1/4$ follows from bounds towards the Ramanujan--Petersson conjecture, see for example \cite{kim03}.\\

\noindent
Our result is unconditional, and in particular it does not depend on the Ramanujan-Petersson conjecture for $\lambda(m)$. Instead, we use the following bound
\begin{equation}\label{fourth_moment}
\sum_{\substack{
1 \leq m \leq X\\ p \nmid m 
}} \lambda(m)^4 \prec X
\end{equation}
valid for all $X \geq 1$. This is proven in the same way as \cite[Lemma 3.6]{Ghosh_2013}. Moreover, this bound is uniform in $p$. On one occasion, we would like to bound individual eigenvalues, but then we only need $\lambda(m)^4 \ll m$ which follows from the best known bounds towards the Ramanujan-Petersson conjecture \cite{kim03}.\\

\section{Balancing}\label{balancing}

We introduce balanced variants of $\Gamma_0(p^n)$ and $f$. This is necessary because $\Gamma_0(p^n)$ is unbalanced in the sense that its cusp have different widths. For example $\infty$ has width $1$ whereas $0$ has width $p^n$. The reason why this is an issue becomes clear by working adelically. Then one finds that the Whittaker function is constant or oscillates linearly in a region in $\mathrm{GL}_2(\mathbb{Q}_p)$ of too large volume, and there is no cancellation in the integral $I_p(\mathbf{m})$ defined in \eqref{key_integral}. When working with a balanced variant, we shrink this region so that estimating trivially is sufficient. And indeed working in the unbalanced setting, our methods give the bound $\lVert f \rVert^4_4 \prec p^{\frac{n}{2}}$ which is \emph{worse} than the trivial bound obtained by interpolating using the best known bound for $\lVert f \rVert _\infty$ and the $L^2$-normalisation $\lVert f \rVert_2 \asymp 1$.\\

\noindent
Write $n = n_1 + n_2$ where $(n_1, n_2)=(n/2,n/2)$ if $n$ is even, and $(n_1, n_2)=((n-1)/2,(n+1)/2)$ if $n$ is odd. The balanced variant of $\Gamma_0(p^n)$ is the congruence subgroup
\begin{equation*}
\Gamma_n := a(p^{n_2}) \Gamma_0(p^n)a(p^{-n_2}) =
\begin{pmatrix}
\mathbb{Z}&p^{n_2} \mathbb{Z}\\
p^{n_1} \mathbb{Z}& \mathbb{Z} \\
\end{pmatrix}\cap \mathrm{SL}_2(\mathbb{Z}),
\end{equation*}
and if we set $f_b:= f \mid_k \left[ a(p^{-n_2}) \right]$, $f_b$ is an automorphic form for $\Gamma_n$ of weight $k$ which we call the \emph{balanced variant of $f$}. By a change of variables, $\lVert f \rVert_4 = \lVert f_b\rVert_4$ where now
\begin{equation*}
\lVert f_b \rVert_4 := \left[ \frac{1}{\mathrm{vol}(\Gamma_n \backslash \mathbb{H})} \int_{\Gamma_n \backslash \mathbb{H}} \lvert f_b(x + i y) \rvert^4 \frac{d x d y}{y^2}\right]^{\frac{1}{4}}.
\end{equation*}

\section{Adelisation}
In this section, we phrase the problem adelically. Our motivation is that the number of cusps of $\Gamma_n$ is of order $p^{n/2}$ and working classically, we would have to consider the Fourier expansion of $f_b$ at each cusp. Instead, we reduce the problem to estimating the integrals defined in Equation \eqref{key_integral}.

\subsection{Adelisation of classical automorphic forms}
Let $\mathbb{A}$ denote the ring of adeles over $\mathbb{Q}$, $\mathbb{A}_f$ the ring of finite adeles, and $\widehat{\mathbb{Z}} \leq \mathbb{A}_f $ the subring of profinite integers. Let $\psi$ denote the additive character of $\mathbb{Q} / \mathbb{A} $ whose infinite part is $\psi _\infty (x)=e(x):=e^{2 \pi i x}$. We have a factorisation $\psi = \psi _\infty \prod_{v \nmid \infty} \psi_v$ where $\psi_v $ is an unramified additive character of $\mathbb{Q}_v$ for each finite place $v$.\\

\noindent
Suppose now that $K$ is an open compact subgroup of $\mathrm{GL}_2(\mathbb{A}_f)$ such that the image of the determinant map $K \rightarrow \mathbb{A}_f^\times$ contains $\widehat{\mathbb{Z}}^\times$. Since $\mathbb{Q}$ has class number one, strong approximation \cite[Theorem 3.3.1]{Bump_1997} tells us that
\begin{equation}\label{strong_approximation}
\mathrm{GL}_2(\mathbb{A})=\mathrm{GL}_2(\mathbb{Q})\mathrm{GL}_2(\mathbb{R})^+ K.
\end{equation}
Here $\mathrm{GL}_2(\mathbb{Q})$ is diagonally embedded into $\mathrm{GL}_2(\mathbb{A})$, and $\mathrm{GL}_2(\mathbb{R})^+$ denotes the matrices in $\mathrm{GL}_2(\mathbb{R})$ with positive determinant which we embed into the infinite part of $\mathrm{GL}_2(\mathbb{A} )$. Given $K$ as above, let $\Gamma_K $ denote the pre-image of $K$ under the diagonal embedding $ \mathrm{SL}_2(\mathbb{Z})\hookrightarrow \mathrm{GL}_2(\mathbb{A}_f)$. Since $K$ is open and compact in $\mathrm{GL}_2(\mathbb{A}_f)$, this is a finite index subgroup of $\mathrm{SL}_2(\mathbb{Z})$. If $\phi$ is an automorphic form of weight $k$ for $\Gamma_K $ and trivial central character, we define its adelisation as the function $\Phi : \mathrm{GL}_2(\mathbb{A} ) \rightarrow \mathbb{C} $ given by
\begin{equation*}
\Phi(g) := \phi \mid_k \left[ g _\infty \right](i)
\end{equation*}
where we have used Equation \eqref{strong_approximation} to write $g = \gamma g _\infty k_0$ for $\gamma \in \mathrm{GL}_2(\mathbb{Q})$, $g _\infty \in \mathrm{GL}_2(\mathbb{R})^+$ and $k_0 \in K$. This is an adelic automorphic form on $\mathrm{GL}_2(\mathbb{A})$ with trivial central character.

Suppose now that $\phi$ and hence $\Phi$ is cuspidal. To compute the $L^4$-norm of $\phi$ adelically, we use that the natural map
\begin{equation*}
\Gamma_K \backslash \mathrm{SL}_2(\mathbb{R}) / \mathrm{SO}(2) \rightarrow Z(\mathbb{A} ) \mathrm{GL}_2(\mathbb{Q}) \backslash \mathrm{GL}_2(\mathbb{A} ) / (\mathrm{SO}(2) \times K)
\end{equation*}
is a homeomorphism, and, if we normalise measures correctly, that
\begin{equation*}
\frac{1}{\mathrm{vol}(\Gamma_K \backslash \mathbb{H})} \int_{\Gamma_K \backslash \mathbb{H} } \lvert \phi(x + i y) \rvert^4 \frac{d x d y}{y^2}= \int_{Z(\mathbb{A}) \mathrm{GL}_2(\mathbb{Q}) \backslash \mathrm{GL}_2(\mathbb{A} )}\lvert \Phi(g) \rvert^4 d g.
\end{equation*}
Returning to our newform $f$ and its balanced variant $f_b$, we have $\Gamma_0(p^n) = \Gamma_{K_0(p^n )}$, and $\Gamma_n = \Gamma_{K_n}$ where
\begin{equation*}
K_0(p^n) =
\begin{pmatrix}
\widehat{\mathbb{Z}}&\widehat{\mathbb{Z}}\\
p^n \widehat{\mathbb{Z} }& \widehat{\mathbb{Z}} \\
\end{pmatrix} \cap \mathrm{GL}_2(\mathbb{A}_f)\quad \text{and} \quad K_n =
\begin{pmatrix}
\widehat{\mathbb{Z} }& p^{n_2}\widehat{\mathbb{Z}}\\
p^{n_1} \widehat{\mathbb{Z}}&\widehat{\mathbb{Z}} \\
\end{pmatrix} \cap \mathrm{GL}_2(\mathbb{A}_f).
\end{equation*}
\noindent
Let $\Phi_f$ and $\Phi_{f_b}$ denote the adelisations of $f$ and $f_b$ respectively. Since $K_n = a^{-1} K_0(p^n) a$ where
\begin{equation*}
a =(I,a(p^{n_2}), a(p^{n_2}),...) \in \mathrm{GL}_2(\mathbb{A}),
\end{equation*}
we find that $\Phi_{f_b}(g)=\Phi_f(g a)$, and since $\mathrm{GL}_2(\mathbb{A})$ is unimodular, we have
\begin{equation}\label{adelic_norm}
\lVert f \rVert_4^4 = \int_{\mathrm{PGL}_2(\mathbb{Q}) \backslash \mathrm{PGL}_2(\mathbb{A} )} \lvert \Phi_{f_b}(g) \rvert^4 d g
\end{equation}
where we have identified $Z(\mathbb{A}) \mathrm{GL}_2(\mathbb{Q})\backslash \mathrm{GL}_2(\mathbb{A})$ with $\mathrm{PGL}_2(\mathbb{Q})\backslash \mathrm{PGL}_2(\mathbb{A})$.

\subsection{Adelic Siegel domain}
To estimate the right hand side of Equation \eqref{adelic_norm}, we construct a Siegel domain for the quotient $\mathrm{PGL}_2(\mathbb{Q}) \backslash \mathrm{PGL}_2(\mathbb{A})$. If $G$ is a group with a left action on a set $X$, we say that $\Omega\subset X$ is a Siegel domain for this action if $\Omega \rightarrow G \backslash X$ is surjective, and $g \Omega \cap \Omega \neq \emptyset$ for only finitely many $g \in G$. If $\Omega $ is a Siegel domain for $\mathrm{PGL}_2(\mathbb{Q} ) \backslash \mathrm{PGL}_2(\mathbb{A})$, then we would have $\lVert f \rVert_4^4 \asymp \int_\Omega \lvert \Phi_{f_b} \rvert^4 d g$. Let
\begin{equation*}
K = \mathrm{PSO}(2) \times \prod_{v \nmid \infty} \mathrm{PGL}_2(\mathbb{Z}_v) \leq \mathrm{PGL}_2(\mathbb{A} ).
\end{equation*}
\begin{lemma}
A Siegel domain for $\mathrm{PGL}_2(\mathbb{Q}) \backslash \mathrm{PGL}_2(\mathbb{A})$ is given by
  \begin{equation*}
    \Omega = \left\{ n(x) a(y) k\,:\,x \in \mathbb{A} / \mathbb{Q},\, y \geq y_0,\, k \in K \right\}
  \end{equation*}
for any $y_0 \in (0,\sqrt{3}/2]$, where $n(x):= \left(
\begin{smallmatrix}
1&x\\
0&1 \\
\end{smallmatrix}
\right)$, and for $y \in \mathbb{R}_{>0}$, $a(y)$ denotes the embedding of $\left(
\begin{smallmatrix}
y&0\\
0&1 \\
\end{smallmatrix}
\right)$ into the infinite part of $\mathrm{PGL}_2(\mathbb{A} )$. Moreover, for any measurable function $h: \Omega \rightarrow \mathbb{C}$, we have
\begin{equation*}
\int_{\Omega }h(g) d g =C \int_{x \in \mathbb{A} / \mathbb{Q} } \int_{y \geq y_0} \int_{k \in K } h(n(x)a(y)k)d k\, \frac{d y}{y^2}\, d x
\end{equation*}
where $d k$ and $d x$ denote the probability Haar measures on $K$ and $\mathbb{A} / \mathbb{Q} $ respectively, $d y$ denotes the Lebesgue measure on $\mathbb{R} $, and $C>0$ is some absolute constant.
\end{lemma}

\begin{proof}
  For the first part of the lemma, we make use of strong approximation and the basic fact that $\Omega_0 = \left\{ z \in \mathbb{H}\,:\,\lvert \operatorname{Re} z \rvert \leq 1/2,\, \operatorname{Im} z \geq y_0 \right\}$ is a Siegel domain for $\mathrm{SL}_2(\mathbb{Z}) \backslash \mathbb{H}$.  Suppose that $g \in \mathrm{PGL}_2(\mathbb{A})$. By strong approximation, we can write $g = \gamma g _\infty k_0$ where $\gamma \in \mathrm{GL}_2(\mathbb{Q})$, $g _\infty \in \mathrm{GL}_2(\mathbb{R})^+$, and $k_0$ lies in the finite part of $K$. We have $g _\infty i = \delta (x + i y)$ for some $\delta  \in \mathrm{SL}_2(\mathbb{Z})$ and $x , y  \in \mathbb{R}$ with $\lvert x \rvert \leq 1/2$, and $y \geq y_0$. Since $x + i y = n(x ) a(y) i$, we have $g _\infty = \delta  n(x)a(y) k _\infty$ for some $k _\infty \in \mathrm{PSO}(2)$. Hence $g = \gamma \delta n(x) a(y) k _\infty k_0$ where the four middle factors are embedded into the infinite part of $\mathrm{PGL}_2(\mathbb{A})$. If we rewrite this expression as $(\gamma \delta \delta_f) n(x) a(y) (k _\infty \delta_f^{-1} k_0) $ where $\delta_f$ is the diagonal embedding of $\delta$ into $\mathrm{PGL}_2(\mathbb{A}_f)$ then $\gamma \delta \delta_f \in \mathrm{GL}_2(\mathbb{Q})$, $n(x) \in N(\mathbb{A} / \mathbb{Q})$, $y \geq y_0$, and $k _\infty \delta_f k_0 \in K$. Hence $\Omega \rightarrow \mathrm{PGL}_2(\mathbb{Q}) \backslash \mathrm{PGL}_2(\mathbb{A})$ is surjective.

  If $g \Omega \cap \Omega \neq \emptyset$ for some $g \in \mathrm{PGL}_2(\mathbb{Q})$, we can find $g_1 , g_2 \in \Omega$ such that $g g_1 = g_2$ in $\mathrm{PGL}_2(\mathbb{A})$ which lifts to an equality $g g_1 = z g_2$ in $\mathrm{GL}_2(\mathbb{A})$ for some $z \in Z(\mathbb{A})$. Comparing determinants at the infinite place shows that $\det g >0$. If $z=(z_v )_v$, comparing the finite places on each sides gives that $g z_v^{-1} \in \mathrm{GL}_2(\mathbb{Z}_v)$ for all finite $v$, and since $z_v \in \mathbb{Z}_v^\times$ for almost all $v$, there is a positive rational number $m$ such that $m g \in \mathrm{SL}_2(\mathbb{Z})$. Replacing $z$ by $mz$, we can assume that $g \in \mathrm{SL}_2(\mathbb{Z})$. Now, by letting the infinite parts on both sides act on $i \in \mathbb{H}$, we see that $g \Omega_0 \cap \Omega_0 \neq \emptyset$ so there are indeed only finitely many possibilities for $g$.

  For the last part, we use that $\mathrm{PGL}_2(\mathbb{A}) = B K$ where $B$ denotes the subgroup of upper triangular matrices. Since $\mathrm{PGL}_2(\mathbb{A})$ is unimodular, \cite[Proposition 2.1.5]{Bump_1997} implies that for any measurable $h : \mathrm{PGL}_2(\mathbb{A}) \rightarrow \mathbb{C}$, we have
  \begin{equation*}
    \int_{g \in \mathrm{PGL}_2(\mathbb{A})} h(g) d g = C\int_{b \in B} \int_{k \in K } h(b k)\, d k \, d b
  \end{equation*}
where $d b$ is the left Haar measure on $B$, $d k$ is the right (and left) Haar measure on $K$ and $C$ is some normalisation constant. When restricted to $\Omega$, it is easy to verify that $d b = d x \frac{d y}{y^2}$ so the proof is complete. 
\end{proof}

From the above lemma we deduce the following estimate:

\begin{equation}\label{adelic_estimate}
\lVert f \rVert_4^4 \asymp \int_{x \in \mathbb{A} / \mathbb{Q}} \int_{y \geq y_0} \int_{k \in K} \lvert \Phi_{f_b}(n(x) a(y) k) \rvert^4 d k\frac{d y}{y^2} d x.
\end{equation}

\subsection{Whittaker expansion}\label{sec:whittaker-expansion}

The next step is to expand $\Phi_{f_b}$ into its Fourier-Whittaker expansion which we derive from that of $\Phi_f$. Let $\pi$ denote the cuspidal automorphic representation generated by $f$. Then $\Phi_f$ is the unique (up to scaling) non-zero vector in $\pi$ which is invariant under $K_0(p^n)$. We can uniquely realise $\pi$ in the $\psi$-Whittaker space $\mathcal{W}(\psi)$ consisting of complex valued functions on $\mathrm{GL}_2(\mathbb{A})$ transforming via $\psi$ under left translation by elements of $N(\mathbb{A})$ and satisfying certain regularity conditions. More precisely, we embed $\pi$ into $\mathcal{W}(\psi)$ using the assignment $\Phi \mapsto W_\Phi \in \mathcal{W}(\psi)$ which sends $\Phi \in \pi$ to
\begin{equation*}
W_\Phi (g) := \int_{x \in \mathbb{A} / \mathbb{Q} } \Phi(n(x)g) \psi^{-1}(x) d x.
\end{equation*}
If $\Phi \in \pi$, it has a Fourier-Whittaker expansion
\begin{equation*}
\Phi(g)  = \sum_{m \in \mathbb{Q}^\times} W_\Phi(a(m)g).
\end{equation*}
We can factor $\pi$ as a restricted tensor product $\otimes_v ' \pi_v$, and likewise each $\pi_v$ has a unique Whittaker model in $\mathcal{W}(\psi_v)$, functions on $\mathrm{GL}_2(\mathbb{Q}_v)$ transforming via $\psi_v$ under left translation by $N(\mathbb{Q}_v)$ subject to additional regularity conditions if $v= \infty$. If $W_f$ is the Whittaker function of $\Phi_f$, it has a factorisation $W_f(g )= \prod_v W_v(g_v)$ where, for $v$ finite, $W_v \in \mathcal{W}(\psi_v)$ is the newvector in the Whittaker model of $\pi_v$, meaning that, up to scaling, $W_v$ is the unique non-zero vector in $\pi_v$ that is right invariant under
\begin{equation}\label{compact_subgroup}
\begin{pmatrix}
\mathbb{Z}_v& \mathbb{Z}_v\\
p^n \mathbb{Z}_v& \mathbb{Z}_v \\
\end{pmatrix} \cap \mathrm{GL}_2(\mathbb{Z}_v).
\end{equation}
The arithmetic normalisation of $f$ explained in Section \ref{classical_automorphic_forms} implies that $W_v$ is normalised such that $W_v(I)=1$. When $v \neq p, \infty$, $\pi_v$ is unramified and $W_v$ can be expressed in terms of the Satake parameters of $\pi_v$. If $v=p$, there are several possibilites for $\pi_p$, and $W_p$ is harder to describe. In Section \ref{whittaker_formulas}, we give formulas for $W_p$ when $p$ is odd, and $n$ is sufficiently large.\\

With this setup, we return to the integral in Equation \eqref{adelic_estimate}. Recalling that $\Phi_{f_b}$ is the right translate of $\Phi_{f_b}$ by $a=(I,a(p^{n_2}), a(p^{n_2}), ...)$, we have
\begin{equation*}
\Phi_{f_b}(n(x)a(y)k)=\sum_{m \in \mathbb{Q}^\times}\psi(mx) W _\infty( a(my) k _\infty) \prod_{v \nmid \infty} W_v(a(m)k_v a(p^{n_2})) 
\end{equation*}
for $x \in \mathbb{A} / \mathbb{Q} $, $y \geq y_0$ and $k =(k _\infty , k_2 , k_3 ,...) \in K$. Since our newform $f$ is of weight $k$, $W _\infty(a(my)k _\infty)= e^{ik \theta} W _\infty(a(m y))$ if $k _\infty = \left(
\begin{smallmatrix}
\cos \theta& -\sin \theta\\
\sin \theta& \cos \theta \\
\end{smallmatrix}
\right) \in \mathrm{SO}(2)$. For a finite place $v$, we shall also make use of the following basic observation

\begin{lemma}
Suppose $m \in \mathbb{Q}^\times$, and $\lvert m \rvert_v> \lvert p^{- n_2} \rvert_v$. Then $W_v(a(m) k_v a(p^{n_2})) = 0$ for all $k_v \in \mathrm{GL}_2(\mathbb{Z}_v)$.  
\end{lemma}

\begin{proof}
Assume that $\lvert m \rvert_v > \lvert p^{- n_2} \rvert_v$. Since $\psi_v$ is unramified, we can choose an element $x \in \mathbb{Q}_v \setminus \mathbb{Z}_p $ such that $\psi_v(x)\neq 1$, and $\lvert m^{-1} x \rvert_v \leq \lvert p^{n_2} \rvert_v$. Since $W_v$ transforms by $\psi_v$ under $N(\mathbb{Q}_v)$, we have $\psi_v (x) W_v(a(m) k_v a(p^{n_2}))=W_v(n(x) a(m) k_v a(p^{n_2}))$. We now write
\begin{equation*}
n(x) a(m) k_v a(p^{n_2}) = (a(m) k_v a(p^{n_2}))(a(p^{n_2})^{-1} k_v^{-1} n(m^{-1} x) k_v a(p^{n_2})), 
\end{equation*}
and claim that the element $a(p^{n_2})^{-1} k_v^{-1} n(m^{-1} x) k_v a(p^{n_2})$ lies in the subgroup \eqref{compact_subgroup}. Indeed, since $\lvert m^{-1} x \rvert_v \leq \lvert p^{n_2} \rvert_v$, we have $n(m^{-1} x) \equiv I \bmod{p^{n_2}\mathbb{Z}_v}$, and so $k_v^{-1} n(m^{-1} x) k_v \equiv I \bmod{p^{n_2} \mathbb{Z}_v}$. Therefore, we can write $k_v^{-1} n(m^{-1} x) k_v = \left(
\begin{smallmatrix}
a& p^{n_2}b\\
p^{n_2}c&d \\
\end{smallmatrix}
\right)$ for some $a,b,c,d \in \mathbb{Z}_v$. We now see that
\begin{equation*}
a(p^{n_2})^{-1} k_v^{-1} n(m^{-1} x) k_v a(p^{n_2}) =
\begin{pmatrix}
a&b\\
p^{2 n_2}c&d \\
\end{pmatrix}
\end{equation*}
which is indeed an element of the subgroup \eqref{compact_subgroup} since it also has unit determinant. Hence $W_v$ is right-invariant under this matrix so
\begin{equation*}
W_v( a(m) k_v a(p^{n_2})) = \psi_v(x)W_v (a(m) k_v a(p^{n_2})), 
\end{equation*}
and since $\psi_v(x)\neq 1$, it follows that $W_v(a(m) k_v a(p^{n_2}))= 0$.
\end{proof}

When $v \neq p, \infty$, $W_v$ is right-invariant under $k_va(p^{n_2})$, so, using the above lemma, we reduce the sum to the following expression for $\Phi_{f_b}(n(x) a(y) k)$:
\begin{multline*}
  \sum_{m \in \mathbb{Z} \setminus \left\{ 0 \right\}} \psi(p^{-n_2}m x) e^{i k \theta} W _\infty(a(p^{-n_2}my))\\
  \cdot \left[ \prod_{v \neq \infty, p} W_v(a(p^{-n_2}m)) \right] W_p(a(p^{- n_2}m) k_p a(p^{n _2})).
\end{multline*}
For $g \in \mathrm{GL}_2(\mathbb{Q}_p)$, we define the balanced (or conjugate) newvector by $W_b(g):= W_p(a(p^{-n_2})g a(p^{n_2}))$ so that $W_p(a(p^{- n_2}m)k_p a(p^{n_2}))= W_b(a(m) k_p)$. $W_b$ is no longer an element of $\mathcal{W}(\psi_p)$ but lives in $\mathcal{W}(\psi_p ')$ where $\psi_p '(x):= \psi_p (p^{- n_2}x)$. At the finite places $q \neq p$, the unramified Whittaker function is given, on the diagonal, by $W_v(a(y))=\lambda(q^{r})/ \sqrt{q^{r}}$ if $v_q(y)=r$ where $\lambda( q^r)$ is the $q^r$th Hecke eigenvalue of $f$. By multiplicativity of Hecke eigenvalues, we deduce that $\prod_{v \neq \infty, p} W_v(a(p^{-n_2}m))= \lambda(|m| ')/\sqrt{|m|'}$ where $m'$ denotes the prime-to-$p$ part of $m$, i.e. $m ' = m / p^{v_p(m)}$. At the infinite place, we have
\begin{equation*}
W _\infty(a(p^{-n_2}my))=
\begin{cases}
  \mathrm{sgn}(m)^a W_{\frac{\mathrm{sgn}(m)k}{2}, \nu -\frac{1}{2}}(4 \pi \lvert m \rvert y)  & \text{if $f$ is a Maass form}\\
  \mathbf{1}_{m >0} W_{\frac{k}{2}, \nu -\frac{1}{2}}(4 \pi m y) & \text{if $f$ is a modular form}
\end{cases}
\end{equation*}
where all notation has been explain in Section \ref{classical_automorphic_forms}. We conclude that the size of $\lVert f \rVert_4^4$ is comparable to that of
\begin{multline}\label{adelic_reformulation}
  \int_{x \in \mathbb{A} / \mathbb{Q} } \int_{y \geq y_0} \int_{k \in \mathrm{PGL}_2(\mathbb{Z}_p)} \\
  \Bigg\lvert \sum_{m \neq 0} \psi(p^{-n_2}m x) W _\infty(a(p^{-n_2}m y)) \frac{\lambda(|m| ')}{\sqrt{ \lvert m \rvert '}} W_b(a(m)k) \Bigg\rvert^4 d k\, \frac{d y}{y^2}\,d x. 
\end{multline}
Before opening up this integral, we need to decompose the integral over $\mathrm{PGL}_2(\mathbb{Z}_p)$ into integrals over $\mathbb{Z}_p^\times$ which is done in the next two sections. 

\section{Geometric considerations}\label{geometric_considerations}

The purpose of this section is to set up the coordinates on $\mathrm{GL}_2(\mathbb{Q}_p)$ that we will use to describe the $p$-adic Whittaker newvector and its balanced variant.  Moreover, we give an expression for the Haar measure when these coordinates are restricted to $\mathrm{PGL}_2(\mathbb{Z}_p)$ which allows to break up the integral in \eqref{adelic_reformulation} into integrals over $\mathbb{Z}_p^\times$.\\

\noindent
Our descriptions are valid for $\mathrm{GL}_2(F)$ where $F$ is any non-archimedean local field of characteristic $0$ and residue characteristic $p$ (we can even allow $p=2$). Given such an $F$, let $\mathcal{O}$ denote the ring of integers of $F$ and $\mathfrak{p}$ the maximal ideal of $\mathcal{O}$. We write $\varpi$ for a generator of $\mathfrak{p}$ and $v$ for the valuation on $F$ normalised such that $v(\varpi)=1$. If $q := \#\mathcal{O} / \mathfrak{p}$ is the size of the residue field, we set $\lvert x \rvert := q^{- v(x)}$. We adopt the convention that $\varpi^{\infty}$ means $0$. The Haar measures $d x$ and $d^\times x$ on $(F,+)$ and $(F^\times ,\,\cdot\,)$ respectively are normalised such that $\mathcal{O}$ and $\mathcal{O}^\times$ have volume $1$. Then $d^\times x = \zeta_F(1) \frac{d x}{\lvert x \rvert}$ where $\zeta_F(s):=(1-q^{-s})^{-1}$ is the local $\zeta$-function of $F$.\\

\noindent
Let $G= \mathrm{GL}_2(F)$, and $K = \mathrm{GL}_2(\mathcal{O})$. The center of $G$ is denoted $Z$. We normalise the Haar measure on $G$ such that $K$ has volume $1$. In addition to $K$, we define the congruence subgroups
\begin{equation*}
K_1(m):=
\begin{pmatrix}
1 + \mathfrak{p}^m&\mathcal{O}\\
\mathfrak{p}^m&\mathcal{O} \\
\end{pmatrix}\quad \text{for} \quad m=1,2,3,...
\end{equation*}
so that the $p$-adic Whittaker function is right invariant under $K_1(n)$ where $n$ is the conductor exponent of $\pi_p$. The subgroup of diagonal matrices in $G$ is denoted $A$, and for $y \in F^\times$, we set $a(y):= \left(
\begin{smallmatrix}
y&0\\
0&1 \\
\end{smallmatrix}
\right) \in A$. The subgroup of upper triangular unipotent matrices is denoted $N$, and we set $n(x ):= \left(
\begin{smallmatrix}
1&x\\
0&1 \\
\end{smallmatrix}
\right) \in N$ for $x \in F$.   

\subsection{Coordinates on $GL_2$}
To motivate our choice of coordinates, recall that if $W _\infty$ denotes the Whittaker function at the infinite place of a newform for $\mathrm{SL}_2(\mathbb{Z})$, $W _\infty$ transforms on the left by $e^{2 \pi i x}$ under $\left(
\begin{smallmatrix}
1&x\\
0&1 \\
\end{smallmatrix}
\right)$ for $x \in \mathbb{R}$, is right invariant under $\mathrm{SO}(2)$ and invariant under scalar matrices. Because of the Iwasawa decomposition of $\mathrm{GL}_2(\mathbb{R})^+$, it is enough to understand $W _\infty$ on the matrices $\left(
\begin{smallmatrix}
y&0\\
0&1 \\
\end{smallmatrix}
\right)$ for $y >0$. When $W_p$ is the $p$-adic Whittaker function of our $\Gamma_0(p^n)$-newform $f$, it is no longer right invariant under the maximal compact subgroup $K$, but as $n$ varies, it is still right invariant under $\cap_m K_1(m)$ which in particular contains $K_A := A \cap K$, the subgroup of diagonal units. Since $A$ normalises $N$, we have well-defined two-sided action of $A \times K_A$ on $N \backslash G$, and when $\kappa$ traverses a set of orbit representatives, it is enough to describe $W_p$ on the matrices $a(y) \kappa$ for $y \in F^\times$. One can think of $(a(y), \kappa)$ as ``polar coordinates'' on $N \backslash G$.

\begin{proposition}\label{representatives}
A set of representatives for the two-sided action of $A\times K_A$ on $N\backslash G$ is given by the cosets of

\begin{equation*}
\begin{pmatrix}
1&0\\
\varpi^{\gamma}&1 \\
\end{pmatrix}\quad0 \leq \gamma \leq \infty\quad\text{and}\quad
\begin{pmatrix}
0&1\\
1&\varpi^{\gamma} \\
\end{pmatrix}\quad 1 \leq \gamma \leq \infty.
\end{equation*}
\end{proposition}

\begin{proof}
We have a bijection
\begin{equation*}
N \backslash G \longrightarrow (F^2 \setminus \left\{ (0,0) \right\}) \times  F^\times,\quad g=\left(
\begin{smallmatrix}
a&b\\
c&d \\
\end{smallmatrix}
\right)\mapsto ((c,d),\det g). 
\end{equation*}
Under this identification, the left-action of $A$ corresponds to 
\begin{equation*}
\left(
\begin{smallmatrix}
	a_1&0\\
	0&a_2 \\
\end{smallmatrix}
\right)((c,d),\det g)=((a_2 c, a_2 d),a_1 a_2 \det g),
\end{equation*}
and the right-action of $K_A$ corresponds to 
\begin{equation*}
((c,d),\det g)\left(
\begin{smallmatrix}
	u_1&0\\
	0&u_2 \\
\end{smallmatrix}
\right)=((u_1c, u_2d), u_1 u_2 \det g).
\end{equation*}
It is then clear that each orbit contains a representative of the form $((1,\varpi^\gamma),1)$ where $0 \leq \gamma \leq \infty$ or $((\varpi^\gamma,1),-1)$ where $1 \leq \gamma \leq \infty$. These lie in distinct orbits since, for an element $((c,d),\det g)$, the quantity $v(c)-v(d)$ is invariant under the action of $A \times K_A$. The cosets in the statement of the proposition are exactly the pre-images of these elements.
\end{proof}

If $\gamma \geq n$, and $\kappa =
(\begin{smallmatrix}
1&0\\
\varpi^\gamma&1 \\
\end{smallmatrix})$, we note that $\kappa \in K_1(n)$ so $W_p(a(y) \kappa)= W_p(a(y))$ which is understood via the Kirillov model of $\pi_p$. In the remaining cases, we extract our formulas for $W_p$ from Assing's thesis \cite{assing_thesis,assing_on_the_size}. Here another set of coordinates is used. There is a decomposition \cite[p. 5292]{assing_on_the_size}:
\begin{equation}\label{decomposition}
G=\bigsqcup_{t\in \mathbb{Z}}\,\,\bigsqcup_{0\leq l\leq n}\,\,\bigsqcup_{v\in \mathcal{O}^{\times}/(1+\mathfrak{p}^{l_n} )}ZNg_{t,l,v}K_1(n),
\end{equation}
where
\begin{equation*}
g_{t,l,v}=
\begin{pmatrix}
0&\varpi^t\\
-1&-v\varpi^{-l} \\
\end{pmatrix},\quad\text{and}\quad l_n=\min\{l,n-l\}.
\end{equation*}
Thus determining $W_p(g_{t, l, v})$ for all $t$, $l$ and $v$ gives a complete description of $W_p$. To use the results of \cite{assing_thesis,assing_on_the_size}, we decompose the matrices $a(y)\kappa$ found above according to \eqref{decomposition}. Write $y=w\varpi^s$ where $w\in \mathcal{O}^{\times}$ and $s\in \mathbb{Z}$. Then one can verfiy the following identities:
\begin{equation*}
\begin{split}
\begin{pmatrix}y
&0\\
\varpi^{\gamma}&1 \\
\end{pmatrix} & =
\begin{pmatrix}
-\varpi^{\gamma}&0\\
0&-\varpi^{\gamma} \\
\end{pmatrix}
\begin{pmatrix}
 1&\varpi^{-\gamma}(1+\varpi^n)y\\
 0&1 \\
\end{pmatrix}
g_{-2\gamma+s,\gamma,w^{-1}}
\begin{pmatrix}
1-\varpi^n &-\varpi^{n-\gamma}\\
w\varpi^{n+\gamma}&w(1+\varpi^n) \\
\end{pmatrix},\\
\begin{pmatrix}
0&y\\
1&\varpi^{\gamma} \\
\end{pmatrix}&=
\begin{pmatrix}
-1&0\\
0&-1 \\
\end{pmatrix}
\begin{pmatrix}
1&-\varpi^ny\\
0&1 \\
\end{pmatrix}g_{s,0,1}
\begin{pmatrix}
1+w\varpi^n&\varpi^{\gamma}+w(\varpi^{n+\gamma}+1)\\
-w\varpi^n&-w(\varpi^{n+\gamma}+1)\\ 
\end{pmatrix}.
\end{split}
\end{equation*}
The last matrix in the first decomposition lies in $K_1(n)$ since we assume that $\gamma \leq n$. We deduce:

\begin{lemma}\label{transform}
Let $y\in F^{\times}$, and write $y=w\varpi^s$ where $w\in \mathcal{O}^{\times}$ and $s\in \mathbb{Z}$. Then
\begin{alignat*}{3}
&W_p\left(a(y)
\begin{pmatrix}
1&0\\
\varpi^{\gamma}&1 \\
\end{pmatrix}\right)
&&=\psi(\varpi^{-\gamma}(1+\varpi^n)y)W_p(g_{-2\gamma+s,\gamma,w^{-1}})\quad&&\text{for }0\leq\gamma\leq n,\\
&W_p\left( a(y)
\begin{pmatrix}
0&1\\
1&\varpi^{\gamma} \\
\end{pmatrix}\right)&&=\psi(-\varpi^n y)W_p(g_{s,0,1})&&\text{for }1 \leq \gamma \leq \infty.
\end{alignat*}
\end{lemma}

\subsection{Decomposing $PGL_2$}\label{sec:decomposing_pgl2}
Because of Equation \eqref{adelic_reformulation}, we are faced with an integral over $\mathrm{PGL}_2(\mathcal{O})$ for which the integrand contains expressions of the form $W_b(a(m)k)$ where $m \in \mathbb{Z}_{>0}$ and $k \in \mathrm{PGL}_2(\mathbb{Z}_p)$. To use the coordinates from the previous subsection, we must understand how the coordinates look when restricted to $\mathrm{PGL}_2(\mathcal{O})$. As $\kappa$ traverses the matrices from Proposition \ref{representatives}
\begin{equation*}
\begin{pmatrix}
1&0\\
\varpi^\gamma&1 \\
\end{pmatrix}\quad 0 \leq \gamma \leq \infty \quad \text{and} \quad
\begin{pmatrix}
0&1\\
1&\varpi^\gamma \\
\end{pmatrix} \quad 1 \leq \gamma \leq \infty
\end{equation*}
the subsets
\begin{equation*}
K_\kappa := (NA \kappa K_A\cap K)/(Z \cap K)
\end{equation*}
partition $\mathrm{PGL}_2 (\mathcal{O})$. The following lemma gives coordinates on each $K_\kappa$ and describes the probability Haar measure on $\mathrm{PGL}_2(\mathcal{O})$ in these coordinates.

\begin{lemma}\label{decomposing_pgl2}
For $\kappa$ as in Proposition \ref{representatives} there is a homeomorphism
\begin{equation*}
\mathcal{O} \times \mathcal{O}^\times \times \mathcal{O}^\times \rightarrow K_\kappa,\quad (x,y,u)\mapsto n(x)a(y)\kappa a(u).
\end{equation*}
In these coordinates, the restriction of the probability Haar measure on $\mathrm{PGL}_2(\mathcal{O})$ to $K_\kappa$ is $\frac{q-1}{q^\gamma (q+1)}d x\, d^\times y\, d^\times u$. 
\end{lemma}

\begin{proof}
We only need to consider the case $\kappa = \left(
\begin{smallmatrix}
1&0\\
\varpi^\gamma&1 \\
\end{smallmatrix}
\right)$. Indeeed, if $\kappa' = \left(
\begin{smallmatrix}
0&1\\
1&\varpi^\gamma \\
\end{smallmatrix}
\right)$, then $\kappa ' = \kappa \left(
\begin{smallmatrix}
0&1\\
1&0 \\
\end{smallmatrix}
\right)$ so one sees that $K_{\kappa '}=K_{\kappa}\left(
\begin{smallmatrix}
0&1\\
1&0 \\
\end{smallmatrix}
\right)$, and $n(x)a(y)\kappa' a(u)=n(x)a(y)\kappa a(u^{-1})\left(
\begin{smallmatrix}
0&1\\
1&0 \\
\end{smallmatrix}
\right)$, where we have used that $\left(
\begin{smallmatrix}
0&1\\
1&0 \\
\end{smallmatrix}
\right) a(u)= a(u^{-1} ) \left(
\begin{smallmatrix}
0&1\\
1&0 \\
\end{smallmatrix}
\right)$ in $\mathrm{PGL}_2(\mathcal{O})$. Moreover, Haar measure on $\mathrm{PGL}_2(\mathcal{O})$ is invariant under right-translation by $\left(
\begin{smallmatrix}
0&1\\
1&0 \\
\end{smallmatrix}
\right)$, and the Haar measure $d^\times u$ on $\mathcal{O}^\times$ is invariant under $u \mapsto u^{-1}$ so these observations reduce the proof to the case $\kappa = \left(
\begin{smallmatrix}
1&0\\
\varpi^\gamma&1 \\
\end{smallmatrix}
\right)$.

To verify that the maps in the statement are homeomorphisms, we start with an element $g \in K \cap N A \kappa K_A$, say
\begin{equation*}
g =
\begin{pmatrix}
1&x\\
0&1 \\
\end{pmatrix}
\begin{pmatrix}
y_1&0\\
0 & y_2 \\
\end{pmatrix}
\begin{pmatrix}
 1  & 0\\
\varpi^\gamma&1 \\
\end{pmatrix}
\begin{pmatrix}
u_1&0\\
0& u_2 \\
\end{pmatrix}=
\begin{pmatrix}
u_1 (y_1 + x y_2 \varpi^\gamma)&x u_2 y_2\\
u_1 y_2 \varpi^\gamma&u_2 y_2 \\
\end{pmatrix}
\end{equation*}
where $x \in F$, $y_1, y_2 \in F^\times$ and $u_1 , u_2 \in \mathcal{O}^\times$. The lower right entry $u_2 y_2$ is an element of $\mathcal{O}$ so $y_2 \in \mathcal{O}$ as $u_2 \in \mathcal{O}^\times$. If $v(y_2)>0$, the last row is a multiple of $\varpi$ contradicting $\det g \in \mathcal{O}^\times $ so in fact $y_2 \in \mathcal{O}^\times$. Since $\det g = y_1 y_2 u_1 u_2 \in \mathcal{O}^\times$, it also follows that $y_1 \in \mathcal{O}^\times$.  Looking at the upper right entry, we find that $x \in \mathcal{O}$. Scaling by $\left(
\begin{smallmatrix}
y_2 u_2&0\\
0&y_2 u_2 \\
\end{smallmatrix}
\right)$, we see that $g$ defines the same element in $\mathrm{PGL}_2(\mathcal{O})$ as
\begin{equation}\label{coordinates}
n(x)a(y) \kappa a(u)=
\begin{pmatrix}
u(y + x \varpi^\gamma)&x\\
u \varpi^\gamma&1 \\
\end{pmatrix}
\end{equation}
where $y = y_1 y_2^{-1}$ and $u = u_1 u_2^{-1} \in \mathcal{O}^\times$. This shows that the maps are surjective, and injectivity is clear from \eqref{coordinates}. Continuity is also clear, and continuity of the inverse follows from the source being compact and the target Hausdorff. It is also clear from \eqref{coordinates} that when $\kappa = \left(
\begin{smallmatrix}
1&0\\
\varpi^\gamma&1 \\
\end{smallmatrix}
\right)$, we have
\begin{equation}\label{description_of_K_kappa}
K_\kappa=\left\{ \left(
\begin{smallmatrix}
a&b\\
c&d \\
\end{smallmatrix}
\right) \in \mathrm{PGL}_2(\mathcal{O})\,:\, v(c/d)=\gamma \right\}.
\end{equation}

\noindent
Before describing the Haar measure restricted to $K_\kappa$, we compute the volume of $K_\kappa$. If $\gamma>0$, we observe that $\sqcup_{m \geq \gamma} K_\kappa$,  where $\kappa = \left(
\begin{smallmatrix}
1&0\\
\varpi^m&1 \\
\end{smallmatrix}
\right)$, is the stabiliser of $(0:1)$ under the natural right-action of $\mathrm{PGL}_2(\mathcal{O})$ on $\mathbb{P}^1 \left( \mathcal{O} / \mathfrak{p}^\gamma \right)$ so 
\begin{equation*}
\mathrm{vol}(K_\kappa)=\frac{1}{\# \mathbb{P}^1 (\mathcal{O} / \mathfrak{p}^{\gamma})}-\frac{1}{\# \mathbb{P}^1 (\mathcal{O}/ \mathfrak{p}^{\gamma+1})}=\frac{q-1}{q^\gamma (q+1)}.
\end{equation*}
By translation invariance of the Haar measure, this is also the volume of $K_{\kappa '}=K_\kappa \left(
\begin{smallmatrix}
0&1\\
1&0 \\
\end{smallmatrix}
\right)$ where $\kappa '=\left(
\begin{smallmatrix}
0&1\\
1&\varpi^{\gamma} \\
\end{smallmatrix}
\right)$. If $\gamma = 0$, we now compute that the volume of $K_\kappa$ is $(q-1)/(q+1)$, by subtracting the volumes of all $K_{\kappa '}$ for $\kappa ' \neq \kappa$ from $\mathrm{vol}(\mathrm{PGL}_2(\mathcal{O}))=1$.

To show that $\frac{q-1}{q^\gamma (q+1)} d x\, d^\times y\, d^\times u$ is the restriction of the Haar measure to $K_\kappa$, it is enough to verify that it agrees with the Haar measure on a $\pi$-system generating the Borel $\sigma$-algebra on $K_\kappa$. We choose the $\pi$-system consisting of cosets $k K(m)$ (and the empty set) where $k \in K_\kappa$, $m> \gamma$, and  $K(m)$ is the principal congruence subgroup $K(m):=\ker \left( \mathrm{PGL}_2 (\mathcal{O})\twoheadrightarrow \mathrm{PGL}_2 (\mathcal{O} / \mathfrak{p}^m) \right)$. It is clear from the description of $K_\kappa$ in \eqref{description_of_K_kappa} that these are actual subsets of $K_\kappa$.  Under the Haar measure $\mathrm{vol}(kK(m))=\mathrm{vol}(K(m))=1/\#\mathrm{PGL}_2 (\mathcal{O}/\mathfrak{p}^m) =1/q^{3m-2}(q^2-1)$. On the other hand, if $k=\left(
\begin{smallmatrix}
u_0 (y_0+x_0\varpi^\gamma) & x_0\\
u_0 \varpi^\gamma&1 \\
\end{smallmatrix}
\right)$ for $(x_0 , y_0 , u_0) \in \mathcal{O} \times \mathcal{O}^\times \times \mathcal{O}^ \times$, we see that
\begin{multline*}
\begin{pmatrix}
u_0 \left( y_0 + x_0 \varpi^\gamma \right)&x_0\\
u_0 \varpi^\gamma&1 \\
\end{pmatrix} \equiv
\begin{pmatrix}
u \left( y + x \varpi^\gamma \right)&x \\
u \varpi^\gamma&1 \\
\end{pmatrix}\pmod{\mathfrak{p}^m} \\ \Leftrightarrow \quad
\begin{cases}
  x \equiv x_0 &  \pmod{\mathfrak{p}^{m}} \\
  u \equiv u_0 &  \pmod{\mathfrak{p}^{m - \gamma}} \\
  yu \equiv y_0 u_0 & \pmod{\mathfrak{p}^{m}}
\end{cases}
\end{multline*}
from which it follows that $k K(m)$ decomposes as
\begin{equation*}
\bigsqcup_{s \in (1+ \mathfrak{p}^{m - \gamma})/(1+\mathfrak{p}^m)}\left\{ \left(
\begin{smallmatrix}
u(y+x \varpi^\gamma)&x\\
u \varpi^\gamma&1 \\
\end{smallmatrix}
\right)\,:\,x \in x_0 + \mathfrak{p}^m,\, u \in s u_0 + \mathfrak{p}^{m},\, y \in s^{-1} y_0 + \mathfrak{p}^m \right\} 
\end{equation*}
so our claimed Haar measure assigns the volume
\begin{equation*}
\frac{q-1}{q^\gamma(q+1)}q^\gamma\mathrm{vol}(x_0 + \mathfrak{p}^m)\mathrm{vol}^\times
(1+\mathfrak{p}^m)\mathrm{vol}^\times(1+\mathfrak{p}^m)=\frac{1}{q^{3m-2}(q^2 -1)}
\end{equation*}
as desired.
\end{proof}

\section{Unfolding}\label{unfolding_section}
The purpose of this section is to open up the integral on the right hand side of \eqref{adelic_reformulation}. Given measurable functions $W_1: \mathbb{R}_{>0} \rightarrow \mathbb{C}$ and $W_2 : F^\times \rightarrow \mathbb{C}$, we define 
\begin{multline}\label{definition_of_N}
\mathcal{N}(W_1 , W_2):=\int_{x \in \mathbb{A} / \mathbb{Q} } \int_{y \geq y_0} \int_{t \in \mathcal{O}^\times }\\ \Bigg|\sum_{m \neq 0} \psi(p^{- n_2}m x) \mathrm{sgn}(m)^a W _1(|m| y)\frac{\lambda(\lvert m \rvert')}{\sqrt{ \lvert m \rvert'}} W_2(mt)\Bigg|^4 d^\times t\, \frac{d y}{y^2}\, d x.
\end{multline}
The flexibility in choosing $W_1$ and $W_2$ allows us to dyadically decompose the Whittaker functions at $\infty$ and at $p$. If $W_1 = W _\infty(a(p^{-n_2}\bullet))$, and $W: F^\times \rightarrow \mathbb{C}$ is any measureable function, we set $\mathcal{N}(W):= \mathcal{N}(W_1 ,W)$, that is
\begin{multline*}
  \mathcal{N}(W):= \int_{x \in \mathbb{A} /\mathbb{Q}} \int_{y \geq y_0} \int_{t \in \mathcal{O}^\times}\\
  \Bigg| \sum_{m\neq 0} \psi(p^{- n_2}m x)  W _\infty( a(p^{-n_2}m y)) \frac{\lambda(\lvert m \rvert')}{\sqrt{\lvert m \rvert'}} W(m t)\Bigg|^4 d^\times t \, \frac{d y}{y^2}\, d x. 
\end{multline*}
Using Lemma \ref{decomposing_pgl2}, we prove the following lemma that estimates $\lVert f \rVert_4^4$ in terms of $\mathcal{N}(W)$ for certain functions $W$ labelled by the representatives $\kappa$ in Proposition \ref{representatives}.

\begin{lemma}
We have $\lVert f \rVert_4^4 \asymp \sum_{\kappa} p^{- \gamma} \mathcal{N}(\kappa W_b)$ where $\kappa$ traverses the matrices in Proposition \ref{representatives}, and $\kappa W_b(t):= W_b(a(t) \kappa)$.
\end{lemma}

\begin{proof}
By Equation \eqref{adelic_reformulation}, we have
\begin{multline*}
  \lVert f \rVert_4^4 \asymp \sum_{\kappa} \int_{x \in \mathbb{A} / \mathbb{Q}} \int_{y \geq y_0} \int_{k \in K_\kappa}\\
   \Bigg| \sum_{m\neq 0} \psi(p^{-n_2}m x) W _\infty(a(p^{-n_2}m y))\frac{\lambda(\lvert m \rvert')}{\sqrt{\lvert m \rvert'}} W_b(a(m) k)\Bigg|^4\, d k \frac{d y}{y^2} d x
 \end{multline*}
where $\kappa$ traverses the representatives in Proposition \ref{representatives}, and $K_\kappa$ is defined in Section \ref{sec:decomposing_pgl2}. We now fix one such $\kappa$ and consider the integral over $k \in K_\kappa$. By Lemma \ref{decomposing_pgl2}, the integral is given by parametrising $k$ as $n(x') a(t) \kappa a(u)$ for $(x',t,u) \in \mathbb{Z}_p \times \mathbb{Z}_p^\times \times \mathbb{Z}_p^\times $ and replacing the measure $d k$ by $\frac{p-1}{p^\gamma(p+1)} d x' d^\times t d^\times u$. Using that $W_b$ right-invariant under $a(u)$, and the identity $a(m) n(x') a(t)= n(m x ') a(mt)$, we find that
\begin{equation*}
W_b(a(m) n(x') a(t) \kappa a(u)) = W_b(n(mx') a(mt) \kappa)= \psi_p(p^{-n_2}m x')W_b(a(m t) \kappa)
\end{equation*}
where the last equality follows from the fact that $W_b$ lies in the $\psi_p(p^{-n_2}\bullet)$-Whittaker model of $\pi_p$ as noted in Section \ref{sec:whittaker-expansion}. Since $d^\times u$ assigns $\mathcal{O}^\times$ volume one, the $u$-integral goes away. The $x'$-integral can be absorbed into integral over $\mathbb{A} / \mathbb{Q}$ as we now explain. The integral $\int_{x \in \mathbb{A} / \mathbb{Q}}$ factors as $\int_{\mathbb{R} / \mathbb{Z}}  \times \prod_{v \neq \infty} \int_{\mathbb{Z}_v }$ where we assign the probability Haar measure to each factor. Moreover, if $x =(x_v)_v \in \mathbb{A} / \mathbb{Q}$, we have $\psi(p^{-n_2} m x)= \prod_{v} \psi_v(p^{- n_2} m x_v)$. We now have
\begin{multline*}
  \psi(p^{- n_2} m x) W _\infty(a(p^{- n_2}m y)) \frac{\lambda(\lvert m \rvert')}{\sqrt{\lvert m \rvert'}}W_b(a(m) k)\\
  = \psi_p(p^{-n_2}m(x_p + x')) \left(\prod_{v\neq p} \psi_v( p^{- n_2} m x_v)\right)W _\infty(a(p^{- n_2}m y)) \frac{\lambda(\lvert m \rvert')}{\sqrt{\lvert m \rvert'}}W_b(a(mt) \kappa).
\end{multline*}
Next, we note that if $f : \mathbb{Z}_p \rightarrow \mathbb{C}$ is integrable then
\begin{equation*}
\int_{x_p \in \mathbb{Z}_p} \int_{x ' \in \mathbb{Z}_p} f(x_p + x ') \, d x' d x_p = \int_{x_p} f(x_p)\, d x_p 
\end{equation*}
simply by the translation invariance of the Haar measure on $\mathbb{Z}_p$. Hence we can drop the $x'$-integral, and we see that
\begin{multline*}
 \int_{x \in \mathbb{A} /\mathbb{Q}} \int_{y \geq y_0 } \int_{k \in K_\kappa } \Bigg| \sum_{m\neq 0} \psi(p^{-n_2}m x) W _\infty(a(p^{-n_2}m y))\frac{\lambda(\lvert m \rvert')}{\sqrt{\lvert m \rvert'}}W_b(a(m)k) \Bigg|^4\, d k \frac{d y}{y^2} d x \\
  = \frac{(p-1)}{p^{\gamma}(p+1)}\int_{x \in \mathbb{A} / \mathbb{Q}} \int_{y \geq y_0} \int_{t \in \mathbb{Z}_p^\times} \Bigg| \sum_{m\neq 0} \psi(p^{-n_2}m x) W _\infty(a(p^{-n_2}my)) \frac{\lambda(\lvert m \rvert')}{\sqrt{\lvert m \rvert'}} \kappa W_b(mt)\Bigg|^4\, d^\times t \frac{d y}{y^2} d x ,
\end{multline*}
and the lemma now follows.
\end{proof}

We now unfold the expression for $\mathcal{N}(W_1 , W_2)$. 

\begin{lemma}\label{unfolding}
Let $W_1 : \mathbb{R}_{>0} \rightarrow \mathbb{C}$ and $W_2: F^\times \rightarrow \mathbb{C}$ be measurable. Then
\begin{equation*}
\mathcal{N}(W_1 , W_2) \ll \sum_{\substack{
    m_1, m_2, m_3, m_4 \geq 1\\
    m_1 + m_2 = m_3 + m_4}} \frac{\lambda(m_1 ')^4 m_1}{(m_1 ')^2} \lvert I _\infty(\mathbf{m}) \rvert \, \lvert I_p(\mathbf{m}) \rvert
\end{equation*}
where, for $\mathbf{m} =(m_1 , m_2 , m_3 , m_4)$ a tuple of positive integers, we define
\begin{equation*}
  I _\infty(\mathbf{m}):= \int_{y \geq m_1 y_0} W_1(y) W_1(y \tfrac{m_2}{m_1})\overline{ W_1(y \tfrac{m_3}{m_1}) W_1(y \tfrac{m_4}{m_1})}\, \frac{d y}{y^2},
\end{equation*}
\begin{equation*}
  I_p(\mathbf{m}):= \int_{t \in  \mathcal{O}^\times} W_2(m_1 t) W_2(m_2 t) \overline{W_2(m_3 t) W_2(m_4 t) }\, d^\times t.
\end{equation*}  
\end{lemma}

\begin{proof}
We start by observing that $\mathcal{N}(W_1, W_2) \ll \mathcal{N}_{+}(W_1 , W_2)$ where $\mathcal{N}_{+}$ is defined in the same way as $\mathcal{N}$ but with the summation over $m$ restricted to $m>0$, and we use that the $t$-integral and $x$-integral are invariant under scaling by $-1$. The characters $\left\{ \psi(\alpha x) \right\}_{\alpha  \in \mathbb{Q} }$ form a Hilbert space basis for $L^2(\mathbb{A} / \mathbb{Q} )$ so we can first integrate over $x$ and apply Parseval's identity to see that $\mathcal{N}_+(W_1 , W_2)$ equals
\begin{equation*}
\sum_{\substack{
    m_1 , m_2 , m_3 , m_4 \geq 1\\
    m_1 + m_2 = m_3 + m_4 }} \frac{\lambda(m_1 ') \cdots \lambda(m_4')}{(m_1 ' \cdots m_4 ')^{1/2}}
 \Bigg|\int_{y \geq y_0} W_1(m_1 y) W_1(m_2 y)\overline{W_1(m_3y) W_1( m_4 y)} \frac{d y}{y^2}\Bigg|  \lvert I_p(\mathbf{m}) \rvert.
\end{equation*}
Applying the Cauchy-Schwarz inequality twice, we find that $\mathcal{N}_+(W_1 , W_2)$ is bounded by 
\begin{equation*}
\sum_{\substack{
    m_1 , m_2 , m_3 , m_4 \geq 1\\
    m_1 + m_2 = m_3 + m_4
}} \frac{\lambda(m_1 ')^4}{(m_1 ')^2} \Bigg| \int_{y \geq y_0} W_1(m_1 y) W_1(m_2 y) \overline{W_1(m_3 y) W_1(m_4 y)} \frac{d y}{y^2} \Bigg| \lvert I_p(\mathbf{m}) \rvert.
\end{equation*}
Substituting $ y \mapsto m_1^{-1} y$ completes the proof.  
\end{proof}

\noindent
We now estimate at the archimedean place under the assumption that $W_2$ is polynomially bounded in $p^n$. This assumption applies when $W_2$ is the $p$-adic Whittaker newvector because Theorem \ref{whittaker_formulas} shows that $W_2$ is bounded by $p^{\frac{n}{12}}$. 

\begin{lemma}\label{archimedean_lemma}
Let $W: F^\times  \rightarrow \mathbb{C}$ be measureable, and suppose $\lVert W \rVert _\infty \ll(p^n)^A$ for some constant $A>0$. Then 
\begin{equation*}
\mathcal{N}(W) \ll_{\varepsilon, A}(p^n)^\varepsilon \left[ 1 + p^{-n_2} \sum_{\substack{
      m_1 \leq 2 p^{(1+\varepsilon)n_2}\\ m_j \in [ 2^{-1} m_1 , 2 m_1 ]\\
      m_1 + m_2 = m_3 + m_4
}} \frac{\lambda(m_1 ')^4 m_1}{(m_1 ')^2} \lvert I_p(\mathbf{m}) \rvert \right].
\end{equation*}
\end{lemma}

\begin{proof}
We decompose the function $y \mapsto W _\infty(a(p^{-n_2}y))$ as
\begin{equation*}
\sum_{i \in \mathbb{Z}} W^{(i)}(y) \quad \textrm{where} \quad W^{(i)}(y):=\mathbf{1}_{2^{i-1} < y \leq 2^i} W_{\infty}(a( p^{- n_2}y))\textrm{ for } y>0.
\end{equation*}
We observe that $\mathcal{N}(W^{(i)}, W)=0$ if $i < 0$ because $y_0 > \frac{1}{2}$ so $W^{(i)}(\lvert m \rvert y)$ is not supported when $m \in \mathbb{Z} \setminus \left\{ 0 \right\}$, and $y >y_0$.  Moreover, one can verify directly from definitions using Equation \eqref{pointwise_whittaker_bound} and the triangle inequality, that $\mathcal{N}(\sum_{i>M}W^{(i)},W) \ll 1$ when $2^M >p ^{100A n_2}$ (say). It follows that
\begin{equation*}
\mathcal{N}(W) \ll_{\varepsilon, A}(p^{n})^\varepsilon \left[1 +\sum_{0 \leq i \leq M} \mathcal{N}(W^{(i)},W)\right]. 
\end{equation*}
When $i \geq 0$, we observe that the integrand in $I _\infty(\mathbf{m})$ vanishes identically unless $m_1 y_0 < 2^i$ and $m_j \in [ 2^{-1} m_1 , 2 m_1]$ for $j=2,3,4$. By Equation \eqref{pointwise_whittaker_bound}, $W^{(i)}(y) \ll (p^{-n_2}y)^ \beta$ for $y \leq p^{n_2}$, and $W^{(i)}(y) \ll _N(p^{-n_2}y)^{-N}$ for $y> p^{n_2}$. We deduce that the following bounds hold
\begin{equation}\label{archimedean_estimate}
I_\infty (\mathbf{m}) \ll_N  \mathbf{1}_{m_1 \leq 2 \cdot  2^i,\,m_j \in  [ 2^{-1} m_1 , 2 m_1 ]}
\begin{cases}
  p^{-n_2} & \textrm{if }2^i \leq p^{n_2}\\
 2^{-i} p^{4N n_2} 2^{-4Ni} & \textrm{if }2^i  > p^{n_2}
\end{cases}
\end{equation}
when $I _\infty(\mathbf{m})$ has $W^{(i)}$ in the integrand (here we use that $\beta>1/4$). At the finite place, we have $I_p(\mathbf{m}) \ll p^{8A n_2}$. Now fix $\varepsilon >0$. By choosing $N$ large enough (only depending on $\varepsilon$ and $A$, and not on $p$ or $n$) and using Lemma \ref{unfolding}, we find that $\mathcal{N}(W^{(i)},W) \ll 1$ when $2^i  \geq p^{(1+ \varepsilon) n_2}$. It follows that
\begin{equation*}
\mathcal{N}(W) \ll_{\varepsilon,A}(p^n)^{\varepsilon}\left[1+ \sum_{1 \leq  2^i \leq p^{(1+\varepsilon) n_2}} \mathcal{N}(W^{(i)},W)\right]. 
\end{equation*}
Using Lemma \ref{unfolding} and the estimate in Equation \eqref{archimedean_estimate} on $\mathcal{N}(W^{(i)}, W)$ for $1 \leq 2^i \leq  p ^{(1+\varepsilon) n_2}$ completes the proof.
 \end{proof}

\section{Oscillatory integrals}\label{oscillatory_integrals}
In the archimedean case, asymptotic formulas for the Whittaker function are derived from integral representations using the method of stationary phase. Similarly, there are integral representations for the $p$-adic Whittaker new vector which can be analysed using $p$-adic stationary phase. In this non-archimedean setting, one usually obtains explicit formulas instead of asymptotics expansions. These formulas can often be expressed in terms of Gau{\ss} sums, but when one encounters degenerate critical points, special functions appear. In our case, when the representation $\pi_p $ has trivial central character and is sufficiently ramified, we meet the $p$-adic Airy function and more general oscillatory integrals with cubic phase, see Subsection \ref{cubic_phase}. This phenomenon is analogous to the archimedean case where the Airy function appears in the transition range. The purpose of this section is to collect some properties of the integrals appearing in the integral representations of the $p$-adic Whittaker functions. A reference for much of the material is Assing's thesis \cite{assing_thesis,assing_on_the_size}. We remark that \cite{assing_thesis,assing_on_the_size} is not the standard references for the material on Gau{\ss} sums in the next subsection, but it will be convenient to use the same notation as in \cite{assing_thesis,assing_on_the_size} because we derive our formulas for the Whittaker function from these references.\\

\noindent
We retain the setup from Section \ref{geometric_considerations}, but additionally assume that the residue characteristic $p$ is odd. Let $\psi$ denote fixed unramified additive character $F$. Earlier in the paper, $\psi$ denoted an additive character of $\mathbb{A} / \mathbb{Q}$, but the rest of our calculations are local so this should not cause confusion. With this choice of $\psi$ and our normalisation of the Haar measure $d x$ on $F$, the Fourier transform on $F$ is self-dual.\\

\noindent
If $\chi$ is a continuous character of $F^\times$, we write $a(\chi)$ for the log-conductor of $\chi$, i.e. the smallest non-negative integer $m$ such that $\chi$ is trivial on $1+ \mathfrak{p}^m$. The $\varepsilon$-factor associated to $\chi$ is denoted $\varepsilon(s, \chi)$, and though this depends on $\psi$, we omit it from the notation since we consider $\psi$ to be fixed. Often, we would like to translate multiplicative oscillation in $\chi$ into additive oscillation in $\psi$. We then use the following result: Let $\kappa_F := \lceil \frac{e(F / \mathbb{Q}_p)}{p-1} \rceil$ where $e(F/\mathbb{Q}_p)=v(p)$ is the absolute ramification index of $F$. If $a(\chi) \geq \kappa_F$, then there exists $b_\chi \in \mathcal{O}^\times$, unique modulo $\mathfrak{p}^{a(\chi)-\kappa_F}$, such that 
\begin{equation}\label{logarithm}
\chi(1 + z \varpi^{\kappa_F})= \psi \left( \frac{b_\chi}{\varpi^{a(\chi)}} \log(1 + z \varpi^{\kappa_F}) \right)
\end{equation}
for all $z \in \mathcal{O}$ where $\log_F(1+x):=\sum_{i \geq 1}\frac{(-1)^{i-1}}{i} x^i $ denotes the $p$-adic logarithm which converges for $x \in \mathfrak{p}^{\kappa_F}$, see \cite[Lemma 4.1]{assing_on_the_size}.\\

\noindent
Since $F$ has odd residue characteristic, there is a unique quadratic character $\mathcal{O}_F^\times \rightarrow \left\{ \pm 1 \right\}$, and we denote it by $\chi_F$. Moreover, we define
\begin{equation}\label{gammaF}
\gamma_F(A, \rho) =
\begin{cases}
  \chi_F(A) \varepsilon \left( \frac{1}{2} , \chi_F  \right) & \text{if $\rho>0$, and $\rho$ is odd}\\
  1 & \text{otherwise}
\end{cases}
\end{equation}
for $A \in \mathcal{O}^\times$ and $\rho \in \mathbb{Z}$.

\subsection{Gau{\ss} sums}
When applying $p$-adic stationary phase, the result usually involves a multidimensional quadratic Gau{\ss} sum (see \cite[Section 4]{assing_on_the_size}), and evaluating these are essential for describing the $p$-adic Whittaker newvector. Let $m$ be a positive integer. Given $A \in \mathrm{Mat}_{m \times m }(\mathcal{O})$, $\rho \in \mathbb{Z}$ and $B \in F^m$, we define the $m$-dimensional quadratic Gau{\ss} sum as
\begin{equation*}
G(A \varpi^{-\rho},B):= \int_{x \in \mathcal{O}^m} \psi( \varpi^{-\rho}x^t A x +  x^t B) d x
\end{equation*}
where $(\,\cdot \,)^t$ denotes matrix transpose, and all vectors are written as column vectors. The measure $d x$ is the product measure on $F^m$ normalised such that $\mathcal{O}^m$ has volume $1$. We will only be interested in the cases $m=1$ and $m=2$. In the one-dimensional case, we have the following result:

\begin{lemma}\label{one_dimensional_gauss_sum}
Suppose $A \in \mathcal{O}^\times $, $\rho \in \mathbb{Z}_{\geq 0}$ and $B \in F$. Then the one-dimensional Gau{\ss} sum $G(A \varpi^{-\rho},B)$ evaluates as
  \begin{equation*}
G(A \varpi^{-\rho},B)=
\begin{cases}
  q^{-\frac{\rho}{2}} \gamma_F(A, \rho) \psi \left( - \frac{\varpi^{-\rho} B^2}{4A} \right) & \textrm{if }B \in \mathfrak{p}^{- \rho}\\
  0 & \textrm{otherwise}
\end{cases}.
\end{equation*}
\end{lemma}

\begin{proof}
This follows from \cite[Lemma 4.3]{assing_on_the_size}.
\end{proof}

\noindent
We also encounter more general one-dimensional Gau{\ss} sums. If $\chi$ is a character of $F^\times$, let
\begin{equation*}
G(x, \chi):= \int_{\mathcal{O}^\times} \psi(xy) \chi(y) d^\times y
\end{equation*}
for $x \in F^\times$. When $\chi$ is non-trivial, this Gau{\ss} sum can be expressed in terms of the $\varepsilon$-factor $\varepsilon \left( \frac{1}{2}, \chi^{-1}\right)$:

\begin{lemma}\label{general_gauss_sum}
  Let $\chi$ be a non-trivial character of $F^\times$. Then
  \begin{equation*}
G(x , \chi)=
\begin{cases}
  q^{-\frac{a(\chi)}{2}} \zeta_F(1) \varepsilon \left( \frac{1}{2} , \chi^{-1} \right)\chi^{-1}(x) & \textrm{if }v(x)=-a(\chi)\\
  0 & \textrm{otherwise}
\end{cases}
\end{equation*}
\end{lemma}

\begin{proof}
See Equation (1.3) in \cite{assing_on_the_size}.
\end{proof}

\noindent
In the two-dimensional case, we will only need to consider $\rho = 1$, and the result is:

\begin{lemma}\label{two_dimensional_gauss_sum}
Let $A = \left(
\begin{smallmatrix}
a&b\\
b&c \\
\end{smallmatrix}
\right) \in \mathrm{Mat}_{2 \times 2}(\mathcal{O} )$ be a symmetric matrix, and $B= \left(
\begin{smallmatrix}
B_1 \\
B_2 \\
\end{smallmatrix}
\right) \in F^2 $. Write $A_{\mathfrak{p}}$ for the image of $A$ in $\mathrm{Mat}_{2 \times 2}(\mathcal{O} / \mathfrak{p})$, and let $\operatorname{rk}(A_{\mathfrak{p}})$ denote the rank of $A_{\mathfrak{p}}$. Assume that $\operatorname{rk}(A_{\mathfrak{p}})>0$, and define
\begin{equation*}
\gamma(A_{\mathfrak{p}}) =
\begin{cases}
  \chi_F(\frac{a}{2})\varepsilon \left( \frac{1}{2} , \chi_F  \right) & \text{if }\operatorname{rk}(A_{\mathfrak{p}})=1\text{ and }a\notin \mathfrak{p}\\
  \chi_F(\frac{c}{2})\varepsilon \left(\frac{1}{2} , \chi_F \right)& \text{if }\operatorname{rk}(A_{\mathfrak{p}})=1\text{ and }a\in\mathfrak{p},\,c\notin\mathfrak
  {p}\\
  \chi_F(\det(A_{\mathfrak{p}})) \varepsilon \left( \frac{1}{2} , \chi_F  \right)^2 & \text{if }\operatorname{rk}(A_{\mathfrak{p}})=2                                                            
\end{cases}.
\end{equation*}
If $a \notin \mathfrak{p}$, then $G(\frac{\varpi^{-1}}{2}A,B)$ is given by
\begin{equation*}
\begin{cases}
  q^{-\frac{1}{2}}\gamma(A_{\mathfrak{p}})\psi \left( - \frac{\varpi B_1^2}{2a} \right)  & \operatorname{rk}(A_{\mathfrak{p}})=1,B_1 , B_2 \in \mathfrak{p}^{-1},B_2 - \tfrac{b}{a} B_1 \in \mathcal{O} \\
  q^{-1} \gamma(A_{\mathfrak{p}}) \psi \left( \frac{- \varpi}{ 2 \det(A_{\mathfrak{p}})}(a B_1^2 - 2b B_1 B_2 + c B_2^2) \right) & \operatorname{rk}(A_{\mathfrak{p}})=2,B_1 , B_2 \in \mathfrak{p}^{-1}\\
  0 & \textrm{otherwise}
\end{cases}.
\end{equation*}
If $a \in \mathfrak{p}$ and $c\notin  \mathfrak{p} $, the result is the same but with $a$ and $c$ exchanged and with $B_1$ and $B_2$ exchanged. If $a,c \in \mathfrak{p}$, then $b\notin \mathfrak{p}$, and we have
\begin{equation*}
G \left( \frac{\varpi^{-1}}{2}A,B \right)= q^{-1}\gamma(A_{\mathfrak{p}})\psi \left( \frac{- \varpi}{2b}(B_1^2 + B_2^2) \right).
\end{equation*}
\end{lemma}

\begin{proof}
This follows from the proof of Lemma 4.3 in \cite{assing_on_the_size}.
\end{proof}

\subsection{Oscillatory integrals with cubic phase}\label{cubic_phase}
We now consider oscillatory integrals of the form
\begin{equation}\label{integral_with_cubic_phase}
\int_{t \in \mathcal{O} } \psi(a t^3 + b t^2 + c t)d t
\end{equation}
where $a,b,c \in F$. These show up when one encounters degenerate critical points in the integral representations of the Whittaker function. A special case is the \emph{$p$-adic Airy function}
\begin{equation*}
\mathrm{Ai}(a;b) := q^{- \frac{v(a)}{3}} \int_{t \in \mathcal{O}} \psi(a t^3 + b t) d t
\end{equation*}
for $a,b \in F$. Before analysing the Airy function, we consider a simpler example which also appears in our description of the Whittaker newvector. Here, an explicit evaluation is possible.

\begin{lemma}\label{easy_cubic_integral}
Let $a,b \in F$ and assume that the following conditions holds: If $v(b) \geq 0$ then $v(a) \geq 0$, and if $v(b)< 0$ then $v(3a)>v(b)$, and $v(a) \geq 3 \lceil \frac{v(b)}{2}\rceil$. Write $b = \varpi^{v(b)}b_0$ where $b_0 \in \mathcal{O}^\times$. Then we have
  \begin{equation*}
\int_{t \in \mathcal{O}}\psi(a t^3 + b t^2)\, d t= \min \left\{ q^{\frac{v(b)}{2}},1 \right\} \gamma_F(b_0,-v(b)).
\end{equation*} 
\end{lemma}

\begin{proof}
If $v(b) \geq 0$ then $v(a) \geq 0$, and both sides equal $1$ since $\gamma_F=1$ when its second argument is non-positive. We now assume that $v(b)<0$, and, by assumption, this implies that $v(3 a)> v(b)$, and $v(a) \geq 3\lceil \frac{v(b)}{2}\rceil$. Write $-v(b)=2k+l$ where $k \geq 0$ and $l \in \left\{ 0,1 \right\}$. Then
\begin{equation*}
\begin{split}
  \int_{t \in \mathcal{O}} \psi(a t^3 + b t^2)\, d t & = \int_{s \in \mathcal{O}} \int_{t \in \mathcal{O} } \psi(a(t + s \varpi^{k+l})^3 + b(t + s \varpi^{k+l})^2)\, d t\, d s\\
  &= \int_{t \in \mathcal{O}} \psi(a t^3 + b t^2) \int_{s \in \mathcal{O}} \psi((3a t^2 +2b t)\varpi^{k+l}s)\,d s\, d t.
\end{split}
\end{equation*}
Here we have used that $2(k+l) \geq - v(b)>- v(3a)$ to drop the terms $3 a t s^2 \varpi^{2(k+l)}$, $a s^3 \varpi^{3(k+l)}$, and $b s^2 \varpi^{2(k+l)}$. The inner-most integral vanishes unless $v((3 a t^2 + 2 b t)\varpi^{k+l})  \geq 0$ in which case it equals $1$. Since $v(3a)>v(b)$, one sees that this happens if and only if $v(t) \geq k$. When $t \in \mathfrak{p}^k$, we have $a t^3 \in \mathcal{O}$ since
\begin{equation*}
v(a t^3) \geq v(a)+3k = v(a)+ 3 \lfloor -\tfrac{v(b)}{2} \rfloor = v(a)- 3 \lceil \tfrac{v(b)}{2}\rceil \geq 0.
\end{equation*}
Hence we simply arrive at $q^{-k} G(b \varpi^{2k},0)= q^{\frac{v(b)}{2}} \gamma_F(b_0 , -v(b))$ by Lemma \ref{one_dimensional_gauss_sum}.
\end{proof}

We now move towards understanding the $p$-adic Airy function. Rather than giving explicit evaluations, we focus on bounding $\mathrm{Ai}(a;b)$. The $p$-adic Airy function is of course only interesting when $v(a)<0$. Assuming this, we prove

\begin{proposition}\label{airy}
Let $a,b \in F$ with $v(a)<0$. Then we have the following bounds on $\mathrm{Ai}(a;b)$:
\begin{enumerate}
\item If $v(b)< v(a)$, then $\mathrm{Ai}(a;b)=0$.
\item If $v(a) \leq v(b)< \frac{1}{3} v(a)$, then $\lvert \mathrm{Ai}(a;b)\rvert \leq 2q^{2+v(3)}q^{-\frac{1}{12}v(a)+\frac{1}{4}v(b)}$.
\item If $\frac{1}{3} v(a) \leq v(b) $, then $\lvert \mathrm{Ai}(a;b) \rvert \leq q^{1+ v(3)} $.
\end{enumerate}
\end{proposition}

We start with a lemma.

\begin{lemma} \label{airy_lemma}
Let $a,b \in F$ with $v(a) \leq \min \left\{ 0,v(b) \right\}$. Then for all $m<- v(3)$, we have
\begin{equation*}
\int_{t \in \mathfrak{p}^m \setminus \mathfrak{p}^{m+1}} \psi(a t^3 + b t) d t =0.
\end{equation*}  
\end{lemma}

\begin{proof}
We decompose the range of integration into cosets: $\mathfrak{p}^m\setminus \mathfrak{p}^{m+1}=\sqcup(t_0 + \mathfrak{p}^k) $ where $k \geq 0$ is to be determined, and $t_0$ ranges over a set representatives for the cosets of $\mathfrak{p}^k$ in $\mathfrak{p}^{m}$ that do not intersect $\mathfrak{p}^{m+1}$. We have
\begin{equation*}
\int_{t \in t_0 + \mathfrak{p}^k} \psi(a t^3 + b t) d t =  \int_{s \in \mathfrak{p}^k} \psi(a(t_0+s)^3 + b(t_0 +s)) d s.
\end{equation*}
Expanding the phase into a polynomial in $s$ gives:
\begin{equation*}
a(t_0 + s)^3 + b(t_0 +s)=a s^3+3at_0 s^2 + (3a t_0^2 +b)s + (a t_0^3 + b t_0).
\end{equation*}
If we can choose $k$ such that
\begin{equation*}
\textrm{(1) }v(a) +3k \geq 0;\quad \textrm{(2) }v(3 a t_0)+2 k \geq 0; \quad \textrm{(3) }v(3 a t_0^2 +b)+k<0.
\end{equation*}
then the integral vanishes because it reduces to an oscillatory integral with non-constant linear phase. Since $v(3t_0)<0$, (2) implies (1). Moreover, $v(3at_0^2)=v(a)+ v(3)+2m<v(b)$ so $v(3at_0^2 +b)=v(a)+v(3)+2m$. Hence we need
\begin{equation*}
\textrm{(i) }v(a)+v(3)+m+2k \geq 0;\quad \textrm{(ii) }v(a)+ v(3)+2m+k<0.
\end{equation*}
This is achieved by $k=-v(a)-v(3)-m>0$.
\end{proof}

We now prove the proposition.

\begin{proof}[Proof of Proposition \ref{airy}]
\hphantom{}  
\begin{enumerate}
\item Let $m = -v(b) \geq 2$. Then, for any $u \in \mathcal{O}^\times$, we have
\begin{equation*}
\textrm{Ai}(a;b) = q^{-\frac{v(a)}{3}}\int_{t \in \mathcal{O}} \psi(a(t + u \varpi^{m-1})^3 + b(t + u \varpi^{m-1}) )\, d t = \psi(b u \varpi^{m-1}) \textrm{Ai}(a;b)
\end{equation*}
where we have used that $3 a t^2 u\varpi^{m-1} + 3 a t u^2 \varpi^{2(m-1)} + a u^3 \varpi^{3(m-1)} \in \mathcal{O}$ so this lies in the kernel of $\psi$. Since $\psi$ is unramified, we can choose $u$ such that $\psi(b u \varpi^{m-1})\neq 1$, and hence $\textrm{Ai}(a;b)=0$.

\item In the integral defining $\mathrm{Ai}(a;b)$, we make the substitution $t\mapsto \varpi^m t$ where $m \geq 0$ is to be determined. Then
  \begin{equation*}
\int_{t \in \mathcal{O}} \psi(a t^3 + b t) \, d t = q^{-m} \int_{t \in \mathfrak{p}^{-m}} \psi(a \varpi^{3m} t^3 + b \varpi^m t)\, d t .
\end{equation*}
We choose $m$ as large as possible subject to $v(a \varpi^{3m}) \leq 0 $ and $v(a \varpi^{3m}) \leq v(b \varpi^{m})$ so that Lemma \ref{airy_lemma} applies. This amounts to $m \leq -\frac{1}{3} v(a)$ and $m \leq \frac{1}{2} (-v(a)+v(b))$. Since
\begin{equation*}
\frac{1}{2}(-v(a)+v(b)) \leq -\frac{1}{3} v(a) \Leftrightarrow v(b) \leq \frac{1}{3} v(a),
\end{equation*}
the largest $m$, we can choose, is $m= \lfloor -\frac{1}{2} v(a)+\frac{1}{2} v(b) \rfloor$. Letting $\alpha := a \varpi^{3m}$ and $\beta:= b \varpi^m$, and using Lemma \ref{airy_lemma}, the above integral equals
\begin{equation}\label{first_integral}
q^{-m} \int_{t \in \mathfrak{p}^{- \min \left\{ v(3), m \right\}}}\psi(\alpha t^3 + \beta t)\, d t.
\end{equation}
We can write $m = -\frac{1}{2} v(a)+\frac{1}{2} v(b)-\delta$ where $\delta \in \left\{ 0,\frac{1}{2} \right\}$. Then
\begin{equation*}
\begin{split}
  v(\alpha)&= v(a)+3\left(-\frac{1}{2} v(a)+\frac{1}{2} v(b)-\delta\right)\\
           &=-\frac{1}{2} v(a)+\frac{3}{2} v(b) - 3 \delta\\
           &= v(\beta)-2 \delta 
\end{split}
\end{equation*}
so $v(\alpha) \in \left\{ v(\beta), v(\beta)-1 \right\}$. We observe that when $v(b)$ is closer to $v(a)$ rather than $\frac{1}{3} v(a)$, then the integral in \eqref{first_integral} still oscillates a lot, and instead of bounding it trivially, we can use stationary phase to find further cancellation.\\

\noindent
Assume first that $v(b)<\frac{1}{3} v(a)-1$. Then $v(\alpha) \leq -2$ so we can write $v(\alpha)=-2k-l$ where $k \geq 1$ and $l \in \left\{ 0,1 \right\}$. Letting $f(t):=\alpha t^3 + \beta t \in F \left[ t \right]$, the integral in \eqref{first_integral} equals.

\begin{equation*}
q^{-m} \int_{t \in \mathfrak{p}^{- \min \left\{ v(3),m \right\}}}  q^{k+l}\int_{ s \in \mathfrak{p}^{k+l}} \psi(f(t + s))\, d s d t.
\end{equation*}
By Taylor expanding, we have
\begin{equation*}
f(t +s) = f(t) + s f '(t) + 3 \alpha t s^2 + \alpha s^3.
\end{equation*}
We claim that the last two terms in the above expression lie in $\mathcal{O}$. Indeed,
\begin{equation*}
v(3 \alpha t s^2) = v(3) +v(\alpha) + v(t) + 2 v(s) \geq v(3)-2k-l-v(3)+ 2(k+l)=l \geq 0 
\end{equation*}
and
\begin{equation*}
v(\alpha s^3) = v(\alpha) + 3 v(s) \geq -2k-l + 3(k+l) = k+2l \geq 0.
\end{equation*}
Hence we find that the above integral equals
\begin{equation}\label{second_integral}
  q^{-m+k+l} \int_{t \in \mathfrak{p}^{- \min \left\{ v(3),m \right\}}}\psi(f(t))\left(\int_{s \in \mathfrak{p}^{k+l}} \psi(f'(t)s)\, d s\right)  d t.
\end{equation}
The $s$-integral vanishes unless $v(f'(t)) \geq -k-l$ in which case it equals $q^{-k-l}$. The condition $v(f'(t)) \geq -k-l$ is equivalent to
\begin{equation}\label{eq:quadratic-equation-original}
t^2 \in -\frac{\beta}{3 \alpha}+ \mathfrak{p}^{k-v(3)}.
\end{equation}
When $t \in \mathfrak{p}^{-\min \left\{ v(3),m \right\}}$, we can write $t=t'/3$ for some $t' \in \mathcal{O}$, and the above is equivalent to
\begin{equation}\label{eq:quadratic-equation}
(t')^2 \in -\frac{3 \beta}{\alpha}+ \mathfrak{p}^{k+ v(3)}.
\end{equation}
The solution set to this congruence is given by a union of cosets modulo $\mathfrak{p}^{k+ v(3)}$, and we now count the number of such cosets case by case, depending on $k$ and $v(\beta /\alpha) \in \left\{ 0,1 \right\}$.\\

\noindent
\textbf{Case 1:} $v(\beta / \alpha)=0$. Then $v(-3 \beta / \alpha)= v(3)<k+ v(3)$. It follows that if \eqref{eq:quadratic-equation} holds then $v(3)$ must be even, and we must have $v(t')=v(3)/2$. If $v(3)$ is even, $t '$ must be of the form $\varpi^{v(3)/2} t''$ for some $t'' \in \mathcal{O}^\times$, and \eqref{eq:quadratic-equation} reduces to $(t'')^2 \in \frac{-3\beta}{ \varpi^{v(3)}\alpha} + \mathfrak{p}^{k}$. Since $v(3\beta/\varpi^{v(3)}\alpha)=0$, the solution set to this congruence either consists of two cosets modulo $\mathfrak{p}^{k}$ or is empty. It follows that the solution set to the original equation \eqref{eq:quadratic-equation} is either empty or the union of $2 q^{v(3)/2}$ cosets modulo $\mathfrak{p}^{k+ v(3)}$.\\ 

\noindent
\textbf{Case 2:} $v(\beta/\alpha)=1$, and $k=1$. Then the congruence \eqref{eq:quadratic-equation} is equivalent to $(t')^2 \in \mathfrak{p}^{k+v(3)}$ which in turn is equivalent to $t ' \in \mathfrak{p}^{\lceil \frac{k+v(3)}{2}\rceil}$. The elements of $\mathfrak{p}^{\lceil \frac{k+v(3)}{2}\rceil}$ form $q^{\lfloor \frac{k+v(3)}{2} \rfloor}$ cosets modulo $\mathfrak{p}^{k+ v(3)}$. Since $k=1$, we the solution set to \eqref{eq:quadratic-equation} is the union of at most $q^{1/2+v(3)/2}$ cosets modulo $\mathfrak{p}^{k+v(3)}$.\\

\noindent
\textbf{Case 3:} $v(\beta/\alpha)=1$, and $k>1$. Then $v(-3 \beta / \alpha)= 1+v(3)< k+ v(3)$. Hence, if \eqref{eq:quadratic-equation} holds then $1+v(3)$ must be even, and we must have $v(t')=(1+v(3))/2$. If $1+v(3)$ is even, $t'$ must be of the form $\varpi^{(1+v(3))/2} t''$ for some $t'' \in \mathcal{O}^\times$, and \eqref{eq:quadratic-equation} is equivalent to $(t'')^2 \in -\frac{3\beta}{\varpi^{1+ v(3)} \alpha}+ \mathfrak{p}^{k-1}$. Now, $-3 \beta / \varpi^{1+ v(3)} \alpha \in \mathcal{O}^\times$, and $k-1>0$, so the solution set to this congruence is either empty or the union of $2$ cosets modulo $\mathfrak{p}^{k-1}$. Hence the solution set to \eqref{eq:quadratic-equation} is either empty or the union of $2 q^{1/2+v(3)/2}$ cosets modulo $\mathfrak{p}_F^{k + v(3)}$.\\

\noindent
From the three cases above, we conclude that the solution set to \eqref{eq:quadratic-equation} is the union of at most $2 q^{1/2+v(3)/2}$ cosets modulo $\mathfrak{p}^{k+v(3)}$. Hence the solution set to \eqref{eq:quadratic-equation-original} is the union of at most $2 q^{1/2+v(3)/2}$ cosets of $\mathfrak{p}^{k}$ in $\mathfrak{p}^{-\min \left\{ v(3),m \right\}}$. Hence the integral in Equation \eqref{second_integral} is bounded in absolute value by
\begin{equation*}
2q^{\frac{1}{2}+\frac{v(3)}{2}} q^{-m}q^{-k} = 2 q^{\frac{1}{2}+\frac{v(3)}{2}} q^{-m}q^{\frac{v(\alpha)}{2}+ \frac{l}{2}}.
\end{equation*}
It follows that
\begin{equation*}
\begin{split}
  \lvert \mathrm{Ai}(a;b) \rvert& \leq 2 q^{ \frac{1}{2} + \frac{1}{2} v(3)}q^{-m}q^{\frac{1}{2} v(\alpha)+ \frac{l}{2}}q^{-\frac{1}{3} v(a)}\\
                  & = 2 q^{\frac{1}{2} +\frac{1}{2}v(3)}q^{-m+\frac{1}{2}(v(a) + 3m)-\frac{1}{3} v(a)}\\
                                & = 2q^{ 1 + \frac{1}{2} v(3)}q^{\frac{1}{6} v(a)+\frac{1}{2} m}\\
                                & \leq 2 q^{1 +\frac{1}{2}v(3)}q^{\frac{1}{6} v(a) +\frac{1}{2}(-\frac{1}{2} v(a)+\frac{1}{2}v(b))}\\
                                & = 2 q^{1 + \frac{1}{2}v(3)}q^{-\frac{1}{12} v(a) +\frac{1}{4} v(b)}.
\end{split}
\end{equation*}
\noindent
Finally, we must consider the edge case $\frac{1}{3} v(a)-1 \leq v(b) < \frac{1}{3} v(a)$. Then one sees that $-\frac{1}{3} v(a)-1 \leq m \leq -\frac{1}{3} v(a)$ so bounding trivially in Equation \eqref{first_integral} gives
\begin{equation*}
\lvert \mathrm{Ai}(a;b) \rvert \leq q^{-m} q^{\min \left\{ v(3),m \right\}}q^{-\frac{v(a)}{3}} \leq q^{1+ v(3)}.
\end{equation*}
Since $-\frac{1}{12} v(a)+\frac{1}{4} v(b) \geq -\frac{1}{4} $, it follows that $\lvert \mathrm{Ai}(a;b) \rvert \leq q^{\frac{5}{4}+v(3)} q^{-\frac{1}{12}v(a) +\frac{1}{4} v(b) }$. In all cases, we have, say, $\lvert \mathrm{Ai}(a;b) \rvert \leq 2 q^{2+v(3)} q^{-\frac{1}{12} v(a)+\frac{1}{4} v(b)}$.

\item We make the same substitution $t \mapsto \varpi^m t$ as in (2), but this time we can take $m= \lfloor -\frac{1}{3} v(a) \rfloor$ and use Lemma \ref{airy_lemma} to obtain
  \begin{equation*}
\int_{t \in \mathcal{O}} \psi(a t^3 + b t)\, d t= q^{-m} \int_{t \in \mathfrak{p}^{- \min \{v(3),m\}}}\psi (\alpha t^3 + \beta t)\, d t
\end{equation*}
for some $\alpha, \beta \in F$. Bounding the integral trivially, it follows that $\lvert \mathrm{Ai}(a;b) \rvert \leq  q^{-\frac{1}{3} v(a)}q^{v(3)}q^{\frac{1}{3} v(a)+1} \leq q^{v(3)+1}$.
\end{enumerate} 
\end{proof}

\begin{remark}\label{airy_remark}
In the range $v(a) \leq v(b) < \frac{1}{3} v(a)-3$, it is possible to evaluate $\mathrm{Ai}(a;b)$ in terms of Gau{\ss} sums as we now explain. Assume for simplicity that $v(3)=0$, and $v(a)$ is sufficiently negative. As in the proof of Case (2) in  Proposition \ref{airy}, we set $m := \lfloor -\frac{1}{2} v(a)+\frac{1}{2} v(b) \rfloor$, $\alpha := a \varpi^{3m}$ and $\beta := \beta \varpi^{m}$. Due to our assumptions on $v(a)$ and $v(b)$, we have $v(\alpha)=-2k-l$ where $k \geq 2$ and $l \in \left\{ 0,1 \right\}$. If $f(t) := \alpha t^3 + \beta t$, the $s$-integral in \eqref{second_integral} vanishes unless $t^2 \in -\beta / 3 \alpha + \mathfrak{p}^k$. As we saw in the proof $v(\beta/\alpha) \in \left\{ 0,1 \right\}$ so since $k \geq 2$, this congruence does not have any solutions unless $-\beta / 3 \alpha$ is a square of valuation $0$, say $-\beta/ 3 \alpha=\gamma^2$ for some $\gamma \in \mathcal{O}^\times$. Note that $\pm\gamma $ are the roots of $f'(t)$, and $t^2 \in -\beta/3 \alpha + \mathfrak{p}^k$ if and only if $t \in \pm \gamma + \mathfrak{p}^k$. If $t \in \pm \gamma + \mathfrak{p}^k$, say $t = \pm \gamma + x$ where $x \in \mathfrak{p}^k$, we have
 \begin{equation*}
f(t) = f(\pm \gamma)+ x^2 f''(\pm \gamma) + \O(\varpi^{3k})
\end{equation*}
so the integral in \eqref{second_integral} reduces to a linear combination of the two integrals
\begin{equation*}
\int_{x \in \mathfrak{p}^k} \psi(x^2 f''(\pm \gamma))\, d x,
\end{equation*}
and these are Gau{\ss} sums that we can evaluate using Lemma \ref{one_dimensional_gauss_sum}.
\end{remark}

\section{Formulas for the Whittaker function}\label{formulas_for_the_whittaker_function}

Let $F$ be as in Section \ref{geometric_considerations}, and assume that the residue characteristic $p$ is odd. Suppose $\pi$ is an infinite dimensional, irreducible, admissibile representation of $\mathrm{GL}_2(F)$. When $\pi$ is unitary, has trivial central character and sufficiently large conductor, we give formulas for the Whittaker newvector $W_\pi$ of $\pi$. These assumptions apply when $\pi$ is the $p$\textsuperscript{th} component in the automorphic representation generated by a newform for $\Gamma_0(p^n)$ with trivial central character, and $n \geq 3$.\\

\noindent
When $\pi$ is as above, and the conductor exponent is at least $3$, there are two possibilities for $\pi$. By \cite[Theorem 12]{Godement1970}, $\pi$ must be a principal series representation or a supercuspidal representation, and by \cite[Theorem 20.2]{bushnell_henniart}, if $\pi$ is supercuspidal it must be dihedral since $p$ is odd. In summary, the two possibilities are:
\begin{itemize}
\item $\pi = \chi \boxplus \chi^{-1}$ is the principal series representation induced from a unitary character $\chi: F^\times \rightarrow S^1$ and its inverse;
\item $\pi$ is a dihedral supercuspidal representation with trivial central character attached to some quadratic extension $E/F$ and unitary character $\xi$ of $E^\times$.
\end{itemize}
In both cases, $\pi$ can be constructed from a quadratic space $E$ over $F$ and a unitary character $\xi$ of $E^\times$ by means of the Weil representation. This classification allows us to give a uniform description of the Whittaker function in terms of $E$ and $\xi$. 

\subsection{Quadratic spaces}\label{quadratic_spaces}
Let $E$ be a two dimensional \'etale algebra over $F$. In practice, this means that either $E=F \times F$, or $E$ is a quadratic extension of $F$, and we refer to these two cases as $E/F$ being split or non-split respectively. $E/F$ becomes a quadratic space over $F$ by equipping it with the trace pairing, but we will not need this fact in our calculations.\\

\noindent
If $E/F$ is a quadratic extension, we write $e(E/F)$ and $f(E/F)$ (or just $e$ and $f$) for the ramification index and inertia degree of $E/F$ respectively. The ring of integers of $E$ is denoted $\mathcal{O}_E$. We fix a representative $\zeta$ for the non-trivial square class in $\mathcal{O}_F^\times /(\mathcal{O}_F^\times)^2$, so $E = F(\sqrt{\zeta})$ if $E/F$ is unramified, and $E =F(\sqrt{\varpi})$ or $E = F(\sqrt{\zeta \varpi})$ if $E/F$ is ramified. We let $d = d(E/F)$ denote the valuation of the discriminant of $E/F$, which can be calculated as $d=e-1$, and normalise the Haar measure on $E$ such that $\mathrm{vol}(\mathcal{O}_E)=q^{-\frac{d}{2}}$. Let $\mathfrak{p}_E$ denote the maximal ideal of $\mathcal{O}_E$ and $\Omega$ a generator of $\mathfrak{p}_E$. If $E/F$ is unramified, we choose $\Omega = \varpi$, and if $E/F$ is ramified, we can assume that $\Omega^2 = \varpi$ by replacing $\varpi$ by $\zeta \varpi$ if necessary. The trace map $E \rightarrow F$ is denoted $\mathrm{Tr}_{E/F}$. Define an additive character $\psi_E$ on $E$ by $\psi_E := \psi\circ \mathrm{Tr}_{E/F}$. The log-conductor of $\psi_E$ can then be calculated as $n(\psi_E)=-d/f$ \cite[Lemma 2.3.1]{schmidt}. Our normalisation of the Haar measure on $E$ ensures that it is self-dual with respect to $\psi_E$.\\

\noindent
If $E=F \times F$ is split, we write $\mathcal{O}_E $ for $\mathcal{O}_F \times \mathcal{O}_F$, $\mathfrak{p}_E$ for $\mathfrak{p}_F \times \mathfrak{p}_F$ and $\Omega :=(\varpi, \varpi)$ for a generator of $\mathfrak{p}_E$. For convenience, we also set $f=f(E/F)=2$, $e=e(E/F)=1$ and $d=d(E/F)=0$. We define the norm and the trace map by the formulas $\mathrm{Tr}_{E/F}(x_1 , x_2)=x_1 + x_2$, and $N_{E/F}(x_1 , x_2)= x_1 x_2$. In analogy with the non-split case, we set $\psi_E := \psi \circ \mathrm{Tr}_{E/F}$, and the Haar measure is again self-dual with respect to $\psi_E$.\\

\noindent
When $E/F$ is as above, and $\xi: E^\times \rightarrow S^1$ is a unitary character that does not factor through the norm map $N_{E/F}: E^\times \rightarrow F^\times$, one can use the Weil representation to construct an irreducible admissible representation $\omega_\xi$ of $\mathrm{GL}_2(F)$ attached to $(E/F, \xi)$ \cite[Section 4.8]{Bump_1997}. If $E=F \times F$, and $\xi(x_1 , x_2)= \chi_1(x_1) \chi_2( x_2)$ for some unitary characters $\chi_1 , \chi_2 : E^\times \rightarrow S^1$, we recover the principal series representation $\chi_1 \boxplus \chi_2$. If $E/F$ is a quadratic extension, $\omega_\xi$ is a supercuspidal representation of $\mathrm{GL}_2(F)$ and such representations are called dihedral. Because the residue characteristic of $F$ is odd, all supercuspidal representations of $\mathrm{GL}_2(F)$ are dihedral \cite[Theorem 20.2]{bushnell_henniart}.

In both cases, the invariants attached to $\omega_\xi$ can be expressed in terms of $\xi$. If $E/F$ is split, and $\xi = \chi_1 \otimes \chi_2$, we have
\begin{equation*}
c(\omega_\xi)= a(\chi_1)+ a(\chi_2), \quad L(\omega_\xi ,s)=L (s, \chi_1) L(s,\chi_2)\textrm{ and }\varepsilon(s, \omega_\xi)= \varepsilon(s, \chi_1) \varepsilon(s, \chi_2). 
\end{equation*}
When $E/F$ is non-split, the corresponding relations read
\begin{equation*}
c(\omega_\xi)= f a(\xi)+d, \quad L(\omega_\xi,s )=1 \textrm{ and } \varepsilon(s,\omega_\xi)= \gamma_0 \varepsilon(s,\xi)
\end{equation*}
for some $\gamma_0 \in S^1$ only depending on $E$ \cite{schmidt}. The additive character used to define the espilon factors are the unramified character $\psi$ of $F$ that we fixed in Section \ref{geometric_considerations}, and the character $\psi_E$ of $E$ defined above. \\

\noindent
In our case, we assume that $\omega_\xi$ has trivial central character. When $E/F$ is split, this simply means that $\xi = \chi \otimes \chi^{-1}$ for some unitary character $\chi: F^\times \rightarrow S^1$. When $E/F$ is non-split, $\omega_\xi$ is a supercuspidal representation with trivial central character which implies that $\omega_\xi$ is \emph{twist-minimal} i.e. $c(\omega_\xi) \leq c(\mu \otimes \omega_\xi)$ for all characters $\mu$ of $F^\times$ \cite[Lemma 2.1]{hu_nelson_saha}. Moreover, when $E/F$ is ramified, and $\omega_\xi$ is twist-minimal, the conductor exponenent of $\omega_\xi$ must be odd \cite[Proposition 3.5]{Tunnell1978}. Equivalently, $a(\xi)$ must be even.\\

\noindent
As the last preparation before stating the formulas for the Whittaker function, we describe how multiplicative oscillation of the character $\xi$ translates into additive oscillation of $\psi_E$. Recall that if $L$ is a finite extension of $\mathbb{Q}_p$, we define $\kappa_L := \lceil \frac{e(L/\mathbb{Q}_p)}{p-1}\rceil$. By \eqref{logarithm}, there is $b_\xi \in \mathcal{O}_E^\times$ uniquely determined modulo $\mathfrak{p}_E^{a(\xi)- \kappa_E}$ such that 
\begin{equation*}
\xi(z)= \psi_E \left( \frac{b_\xi}{\Omega^{a(\xi)-n(\psi_E)}} \log_E(z) \right)
\end{equation*}
for all $z \in 1+ \mathfrak{p}_E^{\kappa_E}$. When $E = F \times F$ is the split extension, and $\xi = \chi \otimes \chi^{-1}$, this formula should be interpreted in the following way: $\log_E(x_1 , x_2):= (\log_F(x_1),\log_F(x_2))$ where $\log_F$ denotes the $p$-adic logarithm on $F$, $a(\xi)= a(\chi)$, $n(\psi_E )=0$ and $\kappa_E = \kappa_F$. The purpose of the following lemma is to show that $b_\xi$ takes a particularly simple form when $\omega_\xi$ has trivial central character. This observation simplifies the formulas for the Whittaker function from \cite{assing_thesis,assing_on_the_size} considerably.

\begin{lemma}\label{logarithm_lemma}
Suppose $\omega_\xi$ has trivial central character, and $a(\xi) \geq \kappa_E$. Then there is $b_0 \in \mathcal{O}_F^\times$ such that we can take
\begin{equation*}
b_\xi =
\begin{cases}
  (b_0,-b_0) & \textrm{if }E = F \times F\\
  b_0 \sqrt{\zeta} & \textrm{if $E/F$ is an unramified quadratic extension}\\
  b_0 & \textrm{if $E/F$ is a ramified quadratic extension}
\end{cases}.
\end{equation*}
In particular, $b_\xi^2 =(-1)^{n-1} N_{E/F }(b_\xi) \in \mathcal{O}_F^\times$.
\end{lemma}

\begin{proof}
If $E= F \times F$ is split, and $\omega_\xi$ has trivial central character, $\xi = \chi \otimes \chi^{-1}$. We can then take $b_\xi =(b_0,-b_0)$ where $b_0 \in \mathcal{O}_F^\times$ satisfies $\chi(z)=\psi(b_0 \varpi^{-a(\chi)} \log_F(z))$.

  Suppose now that $E/F$ is non-split. Then $\omega_\xi$ having trivial central character is equivalent to $\xi \mid_{F^\times}= \chi_{E/F}$ where $\chi_{E/F}$ is the quadratic character of $F^\times$ with kernel $N_{E/F}(E^\times)$. Since $p$ is odd, $\chi_{E/F}(z)=1$ for all $z \in 1 + \mathfrak{p}_F$. Hence
\begin{equation}\label{logarithm_expression}
1 = \xi(z)= \psi_E \left( \frac{b_\xi}{\Omega^{a(\xi)-n(\psi_E)}} \log_E(z) \right)=\psi \left( \mathrm{Tr}\left( \frac{b_\xi}{\Omega^{a(\xi)-n(\psi_E)}} \right) \log_F(z) \right)
\end{equation}
for all $z \in (1+\mathfrak{p}_F )\cap(1+ \mathfrak{p}_E^{\kappa_E})$. We now consider two cases. In the first case, we assume that $E/F$ is unramified. Then $\Omega = \varpi$, $n(\psi_E)=0$ and $\kappa_E = \kappa_F$. Writing $b_\xi = b_1 + b_2 \sqrt{\zeta}$ for some $b_1, b_2 \in \mathcal{O}_F$, we have
\begin{equation*}
\psi \left( \frac{2 b_1}{\varpi^{a(\xi)}} \log_F(z)\right)=1
\end{equation*}
for all $z \in 1 + \mathfrak{p}_F^{\kappa_F}$. The logarithm surjects onto $\mathfrak{p}_F^{\kappa_F}$ so, since $\psi$ has conductor $\mathcal{O}_F$, it follows that $b_1 \equiv 0 \pmod{\mathfrak{p}_F^{a(\xi)-\kappa_F}}$. Since $b_\xi$ is only determined modulo $\mathfrak{p}_E^{a(\xi)-\kappa_E}$, we can indeed assume that $b_1 =0$.

If $E/F$ is ramified, $E = F(\Omega)$, and we can write $b_\xi= b_1 + b_2 \Omega $ for some $b_1 , b_2 \in \mathcal{O}_F$. We have $n(\psi_E)=-1$, and $a(\xi)$ is even because $\omega_\xi$ is twist-minimal. Moreover, we can assume that $\Omega^2$ is a uniformiser in $\mathcal{O}_F$. The last expression in \eqref{logarithm_expression} equals
\begin{equation*}
\psi \left(\frac{2 b_2 }{\Omega^{a(\xi)}}  \log_F (z)\right).
\end{equation*}
This time, the domain of $z$ is $1+ \mathfrak{p}_F^{\lceil \frac{\kappa_E}{2} \rceil}$ so we conclude that $b_2 \in \mathfrak{p}_F^{a(\xi)/2-\lceil \frac{\kappa_E}{2}\rceil}$. In $\mathcal{O}_E$, this means that $b_2 \Omega \in \mathfrak{p}_E^{a(\xi)+1-2 \lceil \frac{\kappa_E}{2}\rceil}\subset \mathfrak{p}_E^{a(\xi)-\kappa_E}$ so we can assume $b_2 =0$.
\end{proof}

\subsection{Overview and statements}
In this section we state the formulas for the Whittaker newvector $W_\pi$ evaluated at the matrices $a(y) \kappa$ where $\kappa $ traverses the matrices in Proposition \ref{representatives}, and we retain the setup from the previous section. The qualitative behaviour of $W_\pi$ is summarised in Table \ref{heuristic} in the introduction. The only place where the coarse description of $W_\pi$ in Table \ref{heuristic} depends on the isomorphism class of $\pi$ is in the transition range where $\kappa = \left(
\begin{smallmatrix}
1&0\\
\varpi^{n/2}&1 \\
\end{smallmatrix}
\right)$, although the precise formulas $W_\pi$ will depend on $\pi$ for most of the matrices $\kappa$. If $\kappa = \left(
\begin{smallmatrix}
1&0\\
\varpi^{n/2}&1 \\
\end{smallmatrix}
\right)$, and $\pi$ is a principal series representation, $W_\pi(a(y)\kappa)$
is supported on all of $0 < \lvert y \rvert \leq 1$ whereas if $\pi$ is supercuspidal, $W_\pi(a(y) \kappa)$ is only supported on $\lvert y \rvert=1$. This is not suprising because the Kirillov model of $\pi$ consists of compactly supported functions on $F^\times$ exactly when $\pi$ is supercuspidal. Another feature is that when $\pi$ is a supercuspidal representation attached to a ramified quadratic extension, the conductor exponent $n$ of $\pi$ is odd, so we cannot have $\kappa =
(\begin{smallmatrix}
1&0\\
\varpi^{n/2}&1 \\
\end{smallmatrix})$, and there is no transition range. We now state the formulas for $W_\pi(a(y)\kappa)$.

\begin{theorem}\label{whittaker_formulas}
Let $\xi$ be a unitary character of $E^\times$ that does not factor through the norm map $N_{E/F}: E^\times \rightarrow F^\times$, and let $b:= b_\xi $ for $b_{\xi}$ as in Lemma \ref{logarithm_lemma}. Let $\pi = \omega_\xi$ be the representation attached to $(E/F,\xi)$. Let $n:= c(\pi)$ denote the conductor exponent of $\pi$, and suppose\footnote{If $F= \mathbb{Q}_p$, and $p \geq 5$, it is enough to assume $n \geq 8$. If $F = \mathbb{Q}_p$, and $p \geq 7$, we only need $n \geq 4$. See the proof Lemma \ref{lem:transition-cubic-integral} and \cite{assing_thesis} in particular \cite[Lemma 3.4.15]{assing_thesis}.} $n \geq \max \{16 \kappa_F, 8 v(5), 16 v(3)+4\}$. If we set $(n_1 , n_2)=(\lfloor \frac{n}{2} \rfloor, \lceil \frac{n}{2}\rceil)$, then $W(a(y) \kappa)$ is given by the following formulas.\\

\noindent
\textbf{Case 1:} $\kappa = \left(
\begin{smallmatrix}
1&0\\
1&1 \\
\end{smallmatrix}
\right)$. Then
\begin{equation}\label{case1}
W_p(a(y) \kappa)= \mathbf{1}_{\lvert y \rvert=q^n} \varepsilon(\tfrac{1}{2},\pi)\psi(y)
\end{equation}
\noindent
\textbf{Case 2:} $\kappa = \left(
\begin{smallmatrix}
1&0\\
\varpi^\gamma&1 \\
\end{smallmatrix}
\right)$ for $0 < \gamma < \frac{n}{2}$. Then
\begin{equation}\label{case2}
W_\pi(a(y) \kappa) = \mathbf{1}_{\lvert y \rvert = q^{n-2 \gamma}} C_{\pi , \kappa} \xi^{-1}(\Omega^{d} x_0+b ) \psi(y \varpi^{- \gamma}+x_0 \varpi^{-n_1})
\end{equation}
where $C_{\pi , \kappa} \in S^1$, and $x_0$ denotes the unique solution in $\mathcal{O}_F$ to the equation
\begin{equation*}
\varpi^{ n_2-\gamma} x^2 +(-1)^ny \varpi^{n - 2 \gamma}x-b^2 \varpi^{n_1 - \gamma}=0.
\end{equation*}

\noindent
\textbf{Case 3:} $\kappa = \left(
\begin{smallmatrix}
1&0\\
\varpi^{n/2}&1 \\
\end{smallmatrix}
\right)$. Then $E/F$ is either split or an unramified quadratic extension. If $E/F $ is split, $W_\pi(a(y) \kappa)$ is supported on $0<\lvert y \rvert \leq 1$, and if $E/F$ is an unramified quadratic extension, $W_\pi(a(y) \kappa)$  is only supported on $\lvert y \rvert=1$. When $E/F$ is split, and $\xi= \chi \otimes \chi^{-1}$, we have
\begin{equation}\label{case3.1}
W_\pi(a(y) \kappa)= \lvert y \rvert^{\frac{1}{2}} \sum_{\pm} C_{\lvert y \rvert, \pm} \chi^{\pm}(y) \quad \textrm{for} \quad 0<\lvert y \rvert \leq q^{-\frac{n}{2}}, 
\end{equation}
and for $q^{-\frac{n}{2}}< \lvert y \rvert<1$, we have
\begin{equation}\label{case3.2}
W_{\pi}(a(y) \kappa) = \lvert y \rvert^{\frac{1}{2}} \sum_{\pm} C_{\lvert y \rvert, \pm} \xi^{-1}(x_{\pm} \pm  b) \psi((x_{\pm}+y) \varpi^{-\frac{n}{2}})
\end{equation}
for some $C_{\lvert y \rvert,\pm} \in S^1$ only depending on $\lvert y \rvert$ (in fact, $C_{\lvert y \rvert,\pm}$ only depends on $\pi$ for $\lvert y \rvert<q^{-n}$), and where $x_{\pm} \in \mathcal{O}_F$ denotes the solution to the equation $x^2 + y x -b^2=0$ satsifying $x_{\pm}\equiv \mp b_0 \pmod{\mathfrak{p}_F}$ where $b=(b_0 , - b_0)$ as in Lemma \ref{logarithm_lemma}.\\

\noindent
When $\lvert y \rvert=1$ and $E/F$ is split or an umramified quadratic extension, we define $\Delta = \Delta(y):= 1+ 4b^2y^{-2}$. For $q^{- \frac{n}{4}}<  \lvert \Delta \rvert \leq 1$, we have
\begin{equation}\label{case3.3}
W_\pi(a(y) \kappa)= \mathbf{1}_{\Delta \in \mathcal{O}_F^2} \lvert \Delta \rvert^{-\frac{1}{4}} \sum_{\pm} T_{\Delta , \pm}(y) \xi^{-1}(x_{\pm} \pm b) \psi((x_{\pm}+y) \varpi^{-\frac{n}{2}})
\end{equation}
where $T_{\Delta, \pm}(y) : \mathcal{O}_F^\times \rightarrow S^1$ depends only on $\pi$, $\lvert \Delta \rvert$ and $y$ modulo $\mathfrak{p}_F^{1 + v(\Delta)}$, and $x_{\pm} \in \mathcal{O}_F$ denote the two solutions to the equation $x^2  + y x - b^2=0$.\\

\noindent
When $0 \leq \lvert \Delta \rvert \leq q^{- \frac{n}{4} }$, we have
\begin{equation}\label{case3.4}
W_\pi(a(y) \kappa)= q^{\frac{n}{12}} T_\Delta(y)\xi^{-1}(-\tfrac{1}{2} y+ b) \psi((-\tfrac{3}{4} y - b^2 y^{-1}) \varpi^{-\frac{n}{2}}) \mathrm{Ai}(U,W).  
\end{equation}
where $T_\Delta(y) \ll 1$ is non-zero and depends only on $\pi$, the ramification of $3$ in $F$, and $y$ modulo $\mathfrak{p}_F$. $\mathrm{Ai}$ denotes the $p$-adic Airy function defined in Section \ref{cubic_phase}. $U$ and $W$ are the rational functions in $y$ given by 
\begin{equation*}
U(y) := \frac{32 b^2(3 y^2 + 4 b^2 )}{3(y^2 -4 b^2)^3}\varpi^{r-\rho - 3 \delta+ 3 \eta} \textrm{ and }W(y):= \frac{y^2 + 4 b^2}{y^2 - 4 b^2}\varpi^{-r-\rho -\delta+ \eta }
\end{equation*}
where $r$ and $\rho$ are defined by $\frac{n}{2}=2r+\rho$ and $\rho \in \left\{ 0,1 \right\}$, $\delta := \lfloor \frac{r}{2} \rfloor$, and $\eta = 0$ if $\rho=1$ and $r$ is odd, and $\eta = \rho$ otherwise . When $\lvert \Delta  \rvert \gg q^{-\frac{n}{3}}$, $W_\pi(a(y) \kappa)$ is of size $\O(\lvert \Delta \rvert^{-\frac{1}{4}})$\footnote{In fact, we can evaluate $W_{\pi}(a(y) \kappa)$ explicitly in terms of Gau{\ss} sums, see Subcase 3.2.2 in the proof below and Remark \ref{airy_remark}. This observation justifies describing $W_p(a(y) \kappa)$ as oscillatory for $1 \leq \lvert \Delta \rvert \ll q^{-\frac{n}{3}}$ in Table \ref{heuristic}.}, and when $\lvert \Delta \rvert \ll q^{-\frac{n}{3}}$, $W_\pi(a(y)\kappa)$ is of size $\O(q^{\frac{n}{12}})$.\\

\noindent
\textbf{Case 4:} $\kappa = \left(
\begin{smallmatrix}
1&0\\
\varpi^\gamma&1 \\
\end{smallmatrix}
\right)$ for $\frac{n}{2}< \gamma < n$. Then
\begin{equation}\label{case4}
W_\pi( a(y) \kappa)= \mathbf{1}_{\lvert y \rvert=1} T_{\pi , \kappa}(y) \xi^{-1}(x_0 + \Omega^d \varpi^{ \gamma -n_2}b) \psi((y + x_0)\varpi^{- \gamma})
\end{equation}
where  $T_{\pi , \kappa}: \mathcal{O}_F^\times \rightarrow S^1$ is either constant or only depends on the square class of $y$ in $\mathcal{O}_F^\times $, and $x_0$ denotes the unique solution in $\mathcal{O}_F^\times$ to the equation
\begin{equation*}
x^2 + y x - b^2 \varpi^{2 \gamma -n}=0.
\end{equation*}

\noindent
\textbf{Case 5:} $\kappa = \left(
\begin{smallmatrix}
1&0\\
\varpi^\gamma&1 \\
\end{smallmatrix}
\right)$ for $\gamma \geq n$. Then
\begin{equation}\label{case5}
W_\pi(a(y) \kappa)=\mathbf{1}_{\lvert y \rvert=1}.
\end{equation}
\noindent
\textbf{Case 6:} $\kappa = \left(
\begin{smallmatrix}
0&1\\
1&\varpi^\gamma \\
\end{smallmatrix}
\right)$ for $1 \leq \gamma \leq \infty$. Then
\begin{equation}\label{case6}
W_\pi(a(y)\kappa)= \mathbf{1}_{\lvert y \rvert=q^n }\varepsilon(\tfrac{1}{2},\pi).
\end{equation}
\end{theorem}

\subsection{Proofs}

In this section, we prove Theorem \ref{whittaker_formulas} case by case but in a slightly different order so that we postpone the hardest case (Case 3) to the end. Following Lemma \ref{transform}, we write $y = w \varpi^s$ where $w \in \mathcal{O}_F^\times$, and $s \in \mathbb{Z}$.\\

\noindent
\textbf{Case 1:} $\kappa = \left(
\begin{smallmatrix}
1&0\\
1&1 \\
\end{smallmatrix}
\right)$. By Lemma \ref{transform}, $W_\pi(a(y)\kappa)= \psi((1+\varpi^n)y)W_\pi(g_{s,0,w^{-1}})$. The result now follows from \cite[Lemma 3.1 and Lemma 3.6]{assing_on_the_size} using that $\pi$ is self-dual (since it has trivial central character) to see that $\varepsilon(\frac{1}{2}, \tilde{\pi})=\varepsilon(\frac{1}{2}, \pi)$.\\

\noindent
\textbf{Case 2:} $\kappa = \left(
\begin{smallmatrix}
1&0\\
\varpi^\gamma&1 \\
\end{smallmatrix}
\right)$ for $0< \gamma < n/2$. We again use Lemma \ref{transform} to write $W_\pi( a(y) \kappa)=\psi(\varpi^{- \gamma}(1+ \varpi^n)y)W_\pi(g_{-2 \gamma +s,\gamma, w^{-1}})$. By \cite[Lemma 3.1 and Lemma 3.6]{assing_on_the_size} the second factor vanishes unless $-2 \gamma +s=-n$, i.e. $s = 2 \gamma -n$. Thus $W_\pi(a(y)\kappa)= \mathbf{1}_{\lvert y \rvert=q^{n-2 \gamma}}\psi(\varpi^{-\gamma}y) W_\pi(g_{-n,\gamma , w^{-1}})$. Using \cite[Lemma 3.4.1, Lemma 3.4.4 and Lemma 3.4.16]{assing_thesis}, we find that
\begin{equation*}
W_\pi(g_{-n,\gamma,w^{-1}})= C_{\pi , \kappa} \xi^{-1}(\Omega^dx_0 + b)\psi(x_0 \varpi^{-n_1})
\end{equation*}
where\footnote{In the case $E = F \times F$, we were not able to recover the exact same expression for the leading constant as in \cite[Lemma 3.4.16]{assing_thesis}. Instead, we found this expression by applying \cite[Lemma 4.6]{assing_on_the_size} to \cite[Lemma 3.6]{assing_on_the_size} and evaluating the Gau{\ss} sums that appear using Lemma \ref{two_dimensional_gauss_sum}. In the non-split case, \cite[Lemma 3.3.1]{assing_thesis} (and \cite[Lemma 3.1]{assing_on_the_size}) miss a factor of $\xi^{-1}(\Omega^{t/f})$ when giving an integral representation for $W_{\pi}(g_{t,l,v})$ in the case $0<l<n$ (here $f$ is the inertial degree of $E/F$). Indeed, this factor is dropped after the third equality in the expression for $W_\pi(g_{t,l,v})$ in \cite[p. 73]{assing_thesis}.  }
\begin{equation*}
C_{\pi , \kappa} :=
\begin{cases}
  \gamma_F(-1,\tfrac{n}{2})\gamma_F(1, \tfrac{n}{2}) & E=F \times F \\
  \xi(\Omega^{n/2})\gamma_0 \gamma_F(-1,\tfrac{n}{2}) & E/F \textrm{ unramified}\\
  \xi(\Omega^{n})\gamma_0                           & E/F \textrm{ ramified}
\end{cases},
\end{equation*}
and $x_0$ is the unique solution in $\mathcal{O}_F$ to the equation
\begin{equation*}
(-1)^nw^{-1} \varpi^{n_2 - \gamma } x^2 +x+ w^{-1}(-1)^{n-1}b^2 \varpi^{n_1 - \gamma}=0.
\end{equation*}
Since $y = w \varpi^{2 \gamma -n}$, the left-hand side rewrites as $(-1)^n \varpi^{n_2 -\gamma}x^2 +(-1)^ny \varpi^{n - 2 \gamma}x-b^2\varpi^{n_1 - \gamma}$. We remark that in the split case $E=F \times F$, we have translated the solution to the quadratic equation in \cite[Lemma 3.4.16]{assing_thesis} by $b_0$ in order to get an expression similar to the non-split case. Moreover, if $E/F$ is a ramified quadratic extension, \cite[Lemma 3.4.4]{assing_thesis} gives a different expression for $W_\pi$ when $\gamma < \frac{n}{4}$. If we read Case I in the proof of the lemma, we see that we can instead choose $\Omega x_0 + b$ as the representative for the unique stationary point modulo $\mathfrak{p}_E^{n_2}$ and recover the above expression.\\

\noindent
\textbf{Case 4}: $\kappa = \left(
\begin{smallmatrix}
1&0\\
\varpi^\gamma&1 \\
\end{smallmatrix}
\right)$ for $n/2< \gamma <n$. By Lemma \ref{transform}, we have $W_\pi(a(y) \kappa)= \psi(\varpi^{- \gamma}(1+ \varpi^n)y) W_\pi(g_{-2 \gamma +s, \gamma, w^{-1}})$. By \cite[Lemma 3.1 and Lemma 3.6]{assing_on_the_size} the second factor now vanishes unless $-2 \gamma + s =-2 \gamma$ i.e. unless $\lvert y \rvert=1$. Hence $W_\pi(a(y) \kappa )= \mathbf{1}_{\lvert y \rvert=1} \psi(\varpi^{-\gamma}y)W_\pi(g_{-2 \gamma,\gamma,w^{-1}})$. We now use \cite[Lemma 3.4.1, Lemma 3.4.4 and Lemma 3.4.16]{assing_thesis} to see that
\begin{equation*}
W_\pi(g_{-2 \gamma,\gamma, w^{-1}})= T_{\pi , \kappa}(y) \xi^{-1}(x_0 + \Omega^d b \varpi^{\gamma - n_2}) \psi(x_0 \varpi^{- \gamma})
\end{equation*}
where\footnote{In the case $E= F \times F$, we were not able to recover the exact same expression for the leading constant as in \cite[Lemma 3.4.16]{assing_thesis}. Instead, we found this expression by applying \cite[Lemma 4.6]{assing_on_the_size} to the fourth case of \cite[Lemma 3.6]{assing_on_the_size} and evaluating the Gau{\ss} sums that appear using Lemma \ref{two_dimensional_gauss_sum}. In the non-split case, we must make the same correction as in Case 2. Note that $\Omega^{- 2 \gamma / f}= \varpi^{- \gamma}$ since $\Omega= \varpi$ in the unramified case, and $\Omega^2 = \varpi$ in the ramified case. Since $\pi$ has trivial central character, it follows that $\xi^{-1}(\Omega^{- 2\gamma / f})=\xi(\varpi^\gamma)=  \chi_{E/F}(\varpi)^\gamma =(-1)^\gamma$. }
\begin{equation*}
T_{\pi , \kappa}(y):=
\begin{cases}
  \gamma_F(-1,\tfrac{n}{2})\gamma_F(1, \tfrac{n}{2}) & E = F \times F\\
 \gamma_0 & E/F \textrm{ unramified}\\
(-1)^\gamma\gamma_0 \gamma_F(2y,1) & E/F \textrm{ ramified}
\end{cases},
\end{equation*}
and $x_0$ is the unique solution in $\mathcal{O}_F^\times$ to $x^2 +y x - b^2 \varpi^{2 \gamma -n}$. In the split case $E =F \times F$, we have again translated the solution from \cite[Lemma 3.4.16]{assing_thesis} by $b_0 \varpi^{\gamma -\frac{n}{2}}$ in order to arrive at this equation.\\

\noindent
\textbf{Case 5:} $\kappa = \left(
\begin{smallmatrix}
1&0\\
\varpi^\gamma&1 \\
\end{smallmatrix}
\right)$ for $\gamma \geq n$. Then $\kappa \in K_1(n)$ so $W_\pi(a(y)\kappa)=W_\pi(a(y)) = \mathbf{1}_{\lvert y \rvert=1}$ by \cite[Equation (1.5)]{assing_on_the_size}.\\

\noindent
\textbf{Case 6:} $\kappa = \left(
\begin{smallmatrix}
0&1\\
1&\varpi^\gamma \\
\end{smallmatrix}
\right)$ for $1 \leq \gamma \leq \infty$. By Lemma \ref{transform}, $W_\pi(a(y)\kappa)= \psi(- \varpi^n y) W_\pi(g_{s,0,1})$, and by \cite[Lemma 3.1 and Lemma 3.6]{assing_on_the_size}, $W_\pi(g_{s,0,1})= \mathbf{1}_{\lvert y \rvert=q^{n}} \varepsilon(\frac{1}{2},\tilde{\pi})=\mathbf{1}_{\lvert y \rvert=q^n} \varepsilon(\frac{1}{2} , \pi)$. Here we use that $\pi$ is self-dual (because it has trivial central character).\\

\noindent
\textbf{Case 3:} $\kappa = \left(
\begin{smallmatrix}
1&0\\
\varpi^{n/2}&1 \\
\end{smallmatrix}
\right)$. Then $E/F$ cannot be a ramified quadratic extension because $n$ must be even. By Lemma \ref{transform},
\begin{equation*}
W_\pi(a(y) \kappa)=\psi(\varpi^{-n/2}(1+\varpi^n)y) W_\pi(g_{-n+s,\frac{n}{2},w^{-1}}).
\end{equation*}
If $E/F$ is an unramified quadratic extension, $W_\pi(g_{-n+s,\frac{n}{2},w^{-1}})$ is only supported for $s=0$. Indeed, if we read Case II in the proof of \cite[Lemma 5.1]{assing_on_the_size}, we see that for $W_\pi(g_{-n+s,\frac{n}{2},w^{-1}})$ to be supported, there must exist $x \in \mathcal{O}_F$ such that $x \varpi^{\frac{s}{2}} \in b_0 + \mathfrak{p}_F^{\lfloor \frac{n}{4} \rfloor}$ but this is impossible for $s\neq 0$ since $b_0 \in \mathcal{O}_F^\times$. However, if $E = F \times F$ is split, $W_\pi(g_{-n+s,\frac{n}{2},w^{-1}})$ is supported for all $s \geq 0$, see the last part of \cite[Lemma 3.6]{assing_on_the_size}. As in \cite[Lemma 3.4.15]{assing_thesis}, we will assume that $n \geq 16 \kappa_F$.\\

\noindent
\textbf{Subcase 3.1:} $E = F \times F$, and $s>0$. Recall that $b =(b_0 , - b_0)$ for some $b_0 \in \mathcal{O}_F^\times$. By the last part of \cite[Lemma 3.6]{assing_on_the_size}, we have \footnote{The $\varepsilon$-factor $\varepsilon(\frac{1}{2} , \chi_i^{-1} \chi_j)$ in the cited lemma should be $\varepsilon(\frac{1}{2},\chi_i \chi_j^{-1})$, c.f. the errata posted on the author's webpage. } 
\begin{equation*}
\begin{split}
  W_\pi(g_{-n+s, \frac{n}{2},w^{-1}})&=q^{-\frac{s}{2}} \sum_{\pm} \varepsilon(\tfrac{1}{2},\chi^{\pm 1}) \varepsilon(\tfrac{1}{2}, \chi^{\pm 2}) \chi^{\pm 1}(w \varpi^{-\frac{n}{2}+s})\\
  &= \lvert y \rvert^{\frac{1}{2}} \sum_{\pm} C_{s,\pm}^{(1)} \chi^{\pm}(y)
\end{split}
\end{equation*}
for $s>n-2$ where $C_{s,\pm }^{(1)}:=\varepsilon(\tfrac{1}{2},\chi^{\pm 1}) \varepsilon(\tfrac{1}{2}, \chi^{\pm 2}) \chi^{\pm 1}(\varpi^{-\frac{n}{2}})$. Here we use that $y=w \varpi^s$ to get the last expression. When $\frac{n}{2} \leq s \leq n-2$, it follows from \cite[Lemma 3.4.16]{assing_thesis} that $W_\pi(g_{-n+s,\frac{n}{2},w^{-1}})$ is given by
\begin{equation*}
\begin{gathered}
 \zeta_F(1)^{-2} q^{\frac{n}{2}} q^{-\frac{s}{2}} \sum_{\pm} G(w^{-1}\varpi^{-\frac{n}{2}}, \chi^{\pm 1}) G(\varpi^{- \frac{n}{2}},\chi^{\mp 2}) \chi^{\pm 1}(\varpi^{-n+s})\\
  =\lvert y \rvert^{\frac{1}{2}} \sum_{\pm} C_{s,\pm}^{(2)} \chi^{\pm}(y)
\end{gathered}
\end{equation*}
where $C_{s, \pm}^{(2)}:= \varepsilon(\frac{1}{2},\chi^{\mp 1}) \varepsilon(\frac{1}{2}, \chi^{\mp 2} ) \chi^{\pm 1}(\varpi^{-\frac{3n}{2}})$, and we have used Lemma \ref{general_gauss_sum} to evaluate the Gau{\ss} sums.

We now assume that $0<s < \frac{n}{2}$. Then it follows from Equation (3.7) in \cite[Lemma 3.6]{assing_on_the_size} that \footnote{The factor $\chi_1(v^{-1}) \varepsilon(\frac{1}{2} , \chi_2^{-1})$ in Equation (3.7) in \cite[Lemma 3.6]{assing_on_the_size} should not be there, as is evident from the proof. See also \cite[Lemma 3.3.9]{assing_thesis}.}
\begin{equation*}
W_\pi(g_{-n+s,\frac{n}{2},w^{-1}})= \zeta_F(1)^{-2} q^{\frac{n}{2}}q^{-\frac{s}{2}} \sum_{l_2 =1}^{n/2} \chi(\varpi)^{-n+s+2 l_{2}} K_{l_2} 
\end{equation*}
where $ K_{ l_2}$ is the oscillatory integral
\begin{equation*}
K(\chi \otimes \chi^{-1},(\varpi^{-n+s+ l_2},\varpi^{- l_2}),w^{-1} \varpi^{-\frac{n}{2}})
\end{equation*}
whose definition can be found in \cite[p. 57]{assing_thesis}. By \cite[Lemma 3.4.15]{assing_thesis}, $K_{l_2}$ is non-zero only for $l_2 = \frac{n}{2}-s$ and $l_2 = \frac{n}{2}$, thus we have to evaluate $K_{\frac{n}{2}-s}$ and $K_{\frac{n}{2}}$. From the definition of $K$, one sees that $K_{\frac{n}{2}}$ is obtained from $K_{\frac{n}{2}-s}$ by replacing $\chi$ by $\chi^{-1}$ so we only need to evaluate $K_{\frac{n}{2}-s}$. Formulas for $K_{\frac{n}{2}-s}$ can be found in \cite[Lemma 3.4.15]{assing_thesis}, but they lead to different expressions in the cases $0<s<\frac{n}{4}$ and $\frac{n}{4} \leq s < \frac{n}{2}$. We use \cite[Lemma 4.6]{assing_on_the_size} to evaluate $K_{\frac{n}{2}-s}$ and give an expression that holds for all $0<s<\frac{n}{2}$. For $0<s<\frac{n}{4}$, it will agree with the expression in \cite{assing_thesis}, and for $\frac{n}{4} \leq s< \frac{n}{2}$, it will simplify to the same expression as in \cite{assing_thesis}. To evaluate $K_{\frac{n}{2}-s}$ using \cite[Lemma 4.6]{assing_on_the_size}, we write $\frac{n}{2}=2 r + \rho$ where $r \geq 1$ and $\rho \in \left\{ 0,1 \right\}$. The set of critical points then consists of pairs $( x_1 , x_2) \in((\mathcal{O}_F / \mathfrak{p}_F)^\times )^2$ such that
\begin{equation*}
\begin{pmatrix}
b_0+x_1 + w^{-1} x_1 x_2 \\
-b_0+x_2 \varpi^{s}+w^{-1} x_1 x_2& \\
\end{pmatrix} \in(\mathfrak{p}_F^r)^2.
\end{equation*}
By eliminating $w^{-1} x_1 x_2$, this is equivalent to
\begin{equation*}
\begin{pmatrix}
x_1 -x_2 \varpi^s +2b_0\\
w^{-1} \varpi^s x_2^2 +(\varpi^s -2b_0 w^{-1}) x_2 -b_0 \\
\end{pmatrix} \in(\mathfrak{p}_F^r )^2.
\end{equation*}
The latter congruence has a unique solution $\tilde{x_0} \in (\mathcal{O}_F / \mathfrak{p}_F^r )^\times$, in fact this quadratic equation equation has a unique solution $\tilde{x_0} \in \mathcal{O}_F^\times$ so the unique critical point is $(\tilde{x_0} \varpi^s -2b_0,\tilde{x_0})$. If we let $x_0:=\tilde{x_0} \varpi^s -b_0$, the critical point is $(x_0 -b_0, (x_0 +b_0) \varpi^{-s})$, and $x_0$ is the unique solution in $\mathcal{O}_F$ to the equation $x^2+y x - b^2 =0$ satisfying $x_0 \equiv -b_0 \pmod{\mathfrak{p}}$. By \cite[Lemma 4.6]{assing_on_the_size}, we find that
\begin{multline*}
 K_{\frac{n}{2}-s}= \zeta_F(1)^2 q^{- \frac{n}{2}} \chi(x_0 -b) \chi^{-1}( (x_0+b) \varpi^{-s})\\
  \cdot \psi((x_0-b)\varpi^{-\frac{n}{2}}+(x_0 +b)\varpi^{-s} \varpi^{-\frac{n}{2}+s} +w^{-1}(x_0 -b)(x_0 +b)\varpi^{-s} \varpi^{-\frac{n}{2}}) G \left( \frac{\varpi^{- \rho}}{2} A,0 \right)
\end{multline*}
where
\begin{equation*}
A :=
\begin{pmatrix}
-b& w^{-1}(x_0 -b)(x_0 + b) \varpi^{-s}\\
w^{-1}(x_0 -b)(x_0 +b)\varpi^{-s}&b \\
\end{pmatrix}.
\end{equation*}
We have $\chi(x_0 -b) \chi^{-1}((x_0 +b) \varpi^{-s}) = \chi(\varpi^s) \xi^{-1}(x_0 +b)$. Moreover, using that $x_0^2 + y x - b^2 =0$, and $y= w \varpi^s$, we find that
\begin{equation*}
w^{-1}(x_0 -b)(x_0 +b)\varpi^{-s} = w^{-1}(x_0^2- b^2)\varpi^{-s}=w^{-1}(-x_0 y) \varpi^{-s}=-x_0
\end{equation*}
so the expression inside $\psi$ reduces to $x_0 \varpi^{-\frac{n}{2}}$. We also have
\begin{equation*}
\det A = -b^2 - x_0^2 = y x_0 - 2 b^2 \equiv -2 b^2 \pmod{\mathfrak{p}}.
\end{equation*}
as $y \in \mathfrak{p}$. Hence, by Lemma \ref{two_dimensional_gauss_sum}, we have
\begin{equation*}
  G \left( \frac{\varpi^{- \rho}}{2}A,0 \right)=
  \begin{cases}
    1 & \textrm{if }\rho =0\\
    q^{-1} \chi_F(-2b^2) \varepsilon(\tfrac{1}{2} , \chi_F)^2 & \textrm{if }\rho = 1
\end{cases}, 
\end{equation*} 
which, directly from definitions, is seen to equal $q^{- \rho}\gamma_F(-2,\frac{n}{2}) \gamma_F(1,\frac{n}{2})$ in both cases. Therefore, we arrive at the formula
\begin{equation*}
  K_{\frac{n}{2}-s} = \zeta_F(1)^2 q^{-\frac{n}{2}}\gamma_F(-2, \tfrac{n}{2}) \gamma_F(1,\tfrac{n}{2}) \chi(\varpi^s) \xi^{-1}(x_0  + b)  \psi(x_0\varpi^{-\frac{n}{2}}).
\end{equation*}
We conclude that
\begin{equation*}
W_\pi(g_{-n, \frac{n}{2}, w^{-1}})= \lvert y \rvert^{\frac{1}{2}} \sum_{\pm} C_{s, \pm}^{(3)} \xi^{-1}(x_{\pm}\pm b) \psi(x_{\pm}\varpi^{-\frac{n}{2}}) 
\end{equation*}
where $x_{\pm}$ is the solution in $\mathcal{O}_F$ to the equation $x^2 + y x -b^2=0$ satisfying $x_{\pm} \equiv \mp b_0\pmod{\mathfrak{p}}$, and $C_{s, \pm}^{(3)}:= \gamma_F(-2,\frac{n}{2}) \gamma_F(1, \frac{n}{2}) \chi^{\pm 1}(\varpi^s)$.\\

\noindent
\textbf{Subcase 3.2:} $E = F \times F$ or $E/F$ unramified, and $\lvert y \rvert=1$. Recall $y = w \varpi^s$ where $w \in \mathcal{O}_F$ so we assume that $s=0$, and $y=w$. If we define $\Delta= \Delta(y):= 1+4 b^2 y^{-2}$, there will be several subcases depending on the size of $\Delta$.\\

\noindent
\textbf{Subcase 3.2.1:} $\lvert  \Delta\rvert=1$. By \cite[Lemma 3.4.1 and Lemma 3.4.15]{assing_thesis}, $W_\pi(g_{-n,\frac{n}{2},y^{-1}})$ vanishes unless $\Delta$ is a square in $\mathcal{O}_F^\times$, and
\begin{equation*}
W_\pi(g_{-n,\frac{n}{2},y^{-1}}) = \mathbf{1}_{\lvert y \rvert=1, \Delta \in(\mathcal{O}_F^\times)^2} \sum_{\pm} T_{\Delta , \pm}(y) \xi^{-1}(x_{\pm} + b) \psi(x_{\pm} \varpi^{-\frac{n}{2}})
\end{equation*}
where $x_{\pm}$ denote the two solutions in $\mathcal{O}_F^\times$ to the equation $x^2 +yx - b^2=0$, and \footnote{We were not able to recover the exact same leading constants as in \cite[Lemma 3.4.15]{assing_thesis} and arrived at this expression using \cite[Lemma 4.6]{assing_on_the_size} to evaluate $K_{\frac{n}{2}}$. In the non-split case, we must add the factor $\xi^{-1}(\Omega^{-n/2})=(-1)^{n/2}$ as in the Case 2 and Case 4.}
\begin{equation*}
T_{\Delta,\pm}(y):=
\begin{cases}
  \gamma_F(- b^2-x_{\pm}^2, \frac{n}{2})\gamma_F(1,\frac{n}{2}) & E = F \times F \\
(-1)^{n/2}\gamma_0 \gamma_F(\zeta(-b^2-N_{E/F}(x_{\pm}+ b) y^{-2}, \frac{n}{2}) & E/F \textrm{ unramified}
\end{cases}.
\end{equation*}
Here we note that the values of $ T_{\Delta, +}(y)$ and $ T_{\Delta, -}(y)$, and their coupling with $\xi^{-1}(x_{\pm}+ b) \psi(x_{\pm}\varpi^{-\frac{n}{2}})$ depend only on $y$ modulo $\mathfrak{p}$.\\

\noindent
\textbf{Subcase 3.2.2:} $\lvert \Delta \rvert<1$. In this case, we found some minor issues with the calculations in \cite{assing_thesis} that are explained in Remark \ref{rem:corrections} following the statement of Lemma \ref{lem:evaluate-gauss-sums}. We redo the calculations here starting from the integral representation of $W_\pi(g_{-n, \frac{n}{2}, y^{-1}})$. It turns out that the split and non-split cases can be treated uniformly. We have 
\begin{equation*}
W_\pi(g_{-n, \frac{n}{2}, y^{-1}}) =  C_\pi q^{\frac{n}{2}}\int_{x \in \mathcal{O}_E^\times} \xi^{-1}(x) \psi( (\mathrm{Tr}_{E/F}(x)+ y^{-1} N_{E/F}(x))\varpi^{-\frac{n}{2}}) \, d_E x
\end{equation*}
where
\begin{equation*}
C_\pi :=
\begin{cases}
  1 & \textrm{if $E = F \times F$}\\
  (-1)^{\frac{n}{2}} \gamma_0 & \textrm{if $E/F$ is an unramified quadratic extension}
\end{cases},
\end{equation*}
and the number $\gamma_0 \in S^1$ is defined in Section \ref{quadratic_spaces} and depends only on $E$. In the split case, $E = F \times F$, this formula follows from  Equation (3.7) in \cite[Lemma 3.6]{assing_on_the_size}\footnote{There is a misprint in the statement of the lemma. The factor $\chi_1(v^{-1}) \varepsilon(\frac{1}{2},\chi_2^{-1})$ should not be there, as is evident from the proof of the lemma. See also the last part of \cite[Lemma 3.3.9]{assing_thesis}.}, and \cite[Lemma 5.11]{assing_on_the_size}. The cited references use the measure $d_E^\times x$ on $\mathcal{O}_E^\times$ in this case, and the factor $\zeta_F(1)^{-2}$ in \cite[Lemma 3.6]{assing_on_the_size} disappears when we change the measure to $d_E x$. If $E/F$ is an unramified quadratic extension, the formula follows from \cite[Lemma 3.1]{assing_on_the_size} \footnote{As in Subcase 3.2.1, we must add a factor of $\xi^{-1}(\varpi^{-n/2})=(-1)^{n/2}$ in the non-split case.}.\\

\noindent
It shall be convenient to use a specific choice of coordinates on $E$ viewed as a two-dimensional vector space over $F$. We define an $F$-basis $e_1 , e_2$ for $E$ as follows:
\begin{equation*}
e_1 :=
\begin{cases}
  (1,1) & \textrm{if $ E/F$ is split}\\
  1 & \textrm{if $E/F$ is non-split}
\end{cases} \quad \textrm{and} \quad  e_2 :=
\begin{cases}
  (1,-1) &  \textrm{if $E/F$ is split}\\
  \sqrt{\zeta} & \textrm{if $E/F$ is non-split}
\end{cases}
\end{equation*}
where $\zeta$ is defined in Section \ref{quadratic_spaces}. It will be convenient to set $\zeta := 1$ if $E = F \times F$ so that we have the relations
\begin{equation*}
b = b_0 e_2 \quad \textrm{and} \quad e_i e_j :=
\begin{cases}
  e_1 & \textrm{if $(i,j)=(1,1)$}\\
  e_2 & \textrm{if $(i,j)=(1,2),(2,1)$}\\
  \zeta e_1 & \textrm{if $(i,j)=(2,2)$}
\end{cases}.
\end{equation*}
We also note that, in the split case, the change of basis matrix between $e_1, e_2$ and the standard basis is $\left(
\begin{smallmatrix}
1&1\\
1&-1 \\
\end{smallmatrix}
\right)$ whose determinant is $-2$ which is a unit. Hence, if $x = x_1 e_1 + x_2 e_2$, we have $d_E x = d_F x_1 d_F x_2$, and it is easy to see that this relation also holds in the non-split case. Finally, we observe that if $x = x_1 e_1 + x_2 e_2$, then the norm and trace of $x$ are given by
\begin{equation*}
\mathrm{Tr}_{E/F}(x)= 2 x_1 \quad \textrm{and} \quad N_{E/F}(x) = x_1^2 - \zeta x_2^2. 
\end{equation*}

Following \cite{assing_on_the_size}, we let $\frac{n}{2}=2 r +\rho$ where $r \geq 1$ and $\rho \in \left\{ 0,1 \right\}$. The first step in our calculations is the following lemma that expresses the above integral in terms of Gau{\ss} sums.

\begin{lemma}\label{lem:first-whittaker-formula}
The Whittaker function $W_\pi(g_{-n, \frac{n}{2}, y^{-1}})$ is given by the following expression:
\begin{equation*}
C_\pi q^{\rho} \sum_{x \in(\mathcal{O}_E / \mathfrak{p}_E^r)^\times} \xi^{-1}(x) \psi \left((\mathrm{Tr}_{E/F}(x)+ y^{-1} N_{E/F}(x))\varpi^{-\frac{n}{2}} \right) G \left( \varpi^{-\rho} A_x , \varpi^{-r-\rho} B_x \right),
\end{equation*}
where, if $x = x_1 e_1 + x_2 e_2$, we define
\begin{equation*}
A_x :=
\begin{pmatrix}
y^{-1} N_{E/F}(x)&b_0 \zeta\\
b_0 \zeta& -\zeta y^{-1} N_{E/F}(x) \\
\end{pmatrix} \quad \textrm{and} \quad  B_x :=
2\begin{pmatrix}
x_1 + y^{-1} N_{E/F}(x)\\
\zeta(x_2 - b_0) \\
\end{pmatrix}.
\end{equation*}
\end{lemma}

\begin{proof}
We can write the $\mathcal{O}_E^{\times}$-integral as
\begin{multline*}
  q^{-2 r}\sum_{x \in(\mathcal{O}_E / \mathfrak{p}_E^r)^\times} \int_{z \in \mathcal{O}_E} \xi^{-1}(x + \varpi^r z)\\ \psi((\mathrm{Tr}_{E/F}(x + \varpi^r z)+ y^{-1} N_{E/F}(x + \varpi^r z))\varpi^{-\frac{n}{2}})d_E z, 
\end{multline*}
and we now express the $z$-integral as a Gau{\ss} sum. We have
\begin{equation*}
\begin{split}
  \xi^{-1}(x + \varpi^r z)& = \xi^{-1}(x) \xi^{-1} \left( 1+ \frac{z}{x}\varpi^r \right)\\
                          & = \xi^{-1}(x) \psi_E \left( -b \log_E \left( 1+ \frac{z}{x}\varpi^r \right) \varpi^{-\frac{n}{2}} \right)\\
                          & = \xi^{-1}(x) \psi_E \left(- b \frac{z}{x}\varpi^{-r-\rho} +b \frac{z^2}{2x^2}\varpi^{-\rho} \right)
\end{split}
\end{equation*}
where we have truncated the Taylor expansion of the logarithm after the second term using our assumption $n \geq 16 \kappa_F$ which also ensures that the Taylor expansion is convergent. We also recall that $\psi_E := \psi \circ \mathrm{Tr}_{E/F}$. Moreover,
\begin{equation*}
\begin{split}
N_{E/F}(x + \varpi^rz) &= N_{E/F }(x) N_{E/F} \left( 1+ \frac{z}{x}\varpi^r \right)\\
& = N_{E/F}(x)  \left( 1+ \mathrm{Tr}_{E/F}\left(\frac{z}{x}\right) \varpi^r + N_{E/F} \left( \frac{z}{x} \right) \varpi^{2r}\right).
\end{split}
\end{equation*}
Hence the integral becomes
\begin{multline*}
  q^{-2r} \sum_{x \in(\mathcal{O}_E / \mathfrak{p}_E^r )^\times} \xi^{-1}(x) \psi((\mathrm{Tr}_{E/F}(x)+ y^{-1} N_{E/F}(x))\varpi^{-\frac{n}{2}})\\
  \int_{z \in \mathcal{O}_E } \psi_E  \left( \left( -\frac{b}{x}+1+\frac{ y^{-1} N_{E/F}(x)}{ x} \right) z \varpi^{- r - \rho}+ \left( \frac{y^{-1} N_{E/F}(z)}{2}+\frac{b z^2}{2 x^2} \right)\varpi^{-\rho} \right)\, d_E z .
\end{multline*}
We now make the substitution $z\mapsto x z $ in the $z$-integral which leaves $d_E z$ unchanged and transforms the $z$-integral into
\begin{equation*}
\int_{z \in \mathcal{O}_E} \psi_E \left( \left( -b+x+y^{-1} N_{E/F}(x) \right)z \varpi^{- r- \rho} +  \left( \frac{y^{-1} N_{E/F}(x z)}{2} + \frac{b}{2}z^2\right)\varpi^{- \rho} \right)\,d_E z.
\end{equation*}
We now compute that
\begin{equation*}
  \mathrm{Tr}_{E/F}\left(\frac{y^{-1} N_{E/F}(x z )}{2} +\frac{ b z^2}{2}\right) = z^t A_x z,
\end{equation*}
and
\begin{equation*}\mathrm{Tr}_{E/F} \left((- b + x + y^{-1} N_{E/F}(x))z \right) = z^t B_x,
\end{equation*}
where $A_x$ and $B_x$ are as in the statement of the lemma, and where we view $z$ as a column vector in its coordinates with respect to $e_1$ and $e_2$. Noting that $q^{\frac{n}{2}-2r}= q^{\rho}$, the proof is complete.
\end{proof}

We now move towards evaluating the Gau{\ss} sums in the above lemma. By Lemma \ref{two_dimensional_gauss_sum}, $G(\varpi^{-\rho}A_x , \varpi^{-r-\rho}B_x)$ vanishes unless $B_x \in(\mathfrak{p}_F^r)^2$. From the definition of $B_x$, we see that this condition is equivalent to
\begin{equation*}
x_2 \equiv b_0 \pmod{\mathfrak{p}_F^r}\quad \textrm{and} \quad x_1^2 + y x_1 - b^2 \equiv 0 \pmod{\mathfrak{p}_F^r}. 
\end{equation*}
The solutions to the above quadratic congruence can be parametrised using \cite[Lemma 4.2]{assing_on_the_size}. The discriminant of the above quadratic equation is $y^2 \Delta$. If we define
\begin{equation*}
Y :=
\begin{cases}
  Y_0 & \textrm{if $v(\Delta)<r$, and $\Delta= Y_0^2 \varpi^{2 \delta_0}$ is a square with $v(\Delta )= 2 \delta_0$ }\\
  0   & \textrm{if $v(\Delta)<r$, and $\Delta$ is not a square}\\
  0   & \textrm{if $v(\Delta) \geq r$}
\end{cases},
\end{equation*}
and
\begin{equation*}
\delta =
\begin{cases}
  \delta_0 & \textrm{if $v(\Delta)<r$, and $\Delta= Y_0^2 \varpi^{2 \delta_0}$ is a square with $v(\Delta )= 2 \delta_0$ }\\
  0 & \textrm{if $v(\Delta)<r$, and $\Delta$ is not a square}\\
  \lfloor \frac{r}{2} \rfloor & \textrm{if $v(\Delta) \geq r$}
\end{cases},
\end{equation*}
then \cite[Lemma 4.2]{assing_on_the_size} tells us that the set of solutions to the quadratic congruence above is $\left\{ -\frac{1}{2}y \pm \frac{1}{2}Y \varpi^\delta y + \alpha \varpi^{r - \delta}\,:\, \alpha \in \mathcal{O}_F / \mathfrak{p}_F^\delta \right\}$, or empty if $v(\Delta)<r$ and $\Delta \notin \mathcal{O}_F^2$. We now define
\begin{equation*}
A_{\pm}:= -\frac{1}{2} y \pm \frac{1}{2} Y \varpi^\delta y + b, \quad \textrm{and}\quad B_{\pm} := \sigma(A_{\pm})=-\frac{1}{2} y \pm \frac{1}{2} Y \varpi^\delta y -b 
\end{equation*}
where $\sigma$ is the non-trivial element of $\mathrm{Aut}(E/F)$. Then we see that $B_x \in(\mathfrak{p}_F^r)^2$ if and only if $x$ lies in the set $S\subset (\mathcal{O}_E / \mathfrak{p}_E^r )^\times$ defined as
\begin{equation*}
  S :=
\begin{cases}
  \emptyset & \textrm{if $v(\Delta) < r$, and $\Delta$ is not a square }\\
  \left\{ A_{\pm}+\alpha \varpi^{r-\delta}\,:\, \alpha \in \mathcal{O}_F / \mathfrak{p}_F^\delta \right\} & \textrm{otherwise}
\end{cases}.  
\end{equation*}
The task is now to evaluate $G(\varpi^{-\rho}A_x , \varpi^{-r-\rho}B_x)$ when $x \in S$, and that is the purpose of the following lemma.

\begin{lemma}\label{lem:evaluate-gauss-sums}
Assume that $S$ is non-empty, and let $x = A_{\pm}+ \alpha \varpi^{r-\delta} \in S$. Then $G(\varpi^{- \rho} A_x , \varpi^{- r - \rho} B_x) $ equals $1$ if $\rho = 0$. When $\rho = 1$, we have 
\begin{equation*}
G(\varpi^{-\rho}A_x,\varpi^{-r-\rho}B_x) = q^{-\frac{\rho}{2}} \gamma_F(2y, \tfrac{n}{2})
\begin{cases}
  \mathbf{1}_{\alpha \in \mathfrak{p}_F} & \textrm{if $v(\Delta) <r $}\\
  \mathbf{1}_{\alpha^2 \equiv \frac{\Delta y^2}{4 \varpi^r} \bmod{\mathfrak{p}_F}} & \textrm{if $v(\Delta)=r$, and $r$ is even}\\
  0 & \textrm{if $v(\Delta) = r$, and $r$ is odd}\\
  \mathbf{1}_{\alpha \in \mathfrak{p}_F} & \textrm{if $v(\Delta)>r$, and $r$ is even}\\
  1 & \textrm{if $v(\Delta)>r$, and $r$ is odd}
\end{cases}.
\end{equation*}
\end{lemma}

\begin{remark}\label{rem:corrections}
Similar calculations are carried out in \cite[p. 94 and p. 135]{assing_thesis}, and it is implicitly used that the support of the Gau{\ss} is $\alpha \in \mathfrak{p}_F^\rho$ in all cases. As the above lemma shows, this is not always the case. The paper \cite{assing_on_the_size} only considers the case $\rho = 0$ where the Gau{\ss} is identically equal to $1$ for all $x \in S$.
\end{remark}

\begin{proof}[Proof of Lemma \ref{lem:evaluate-gauss-sums}]
If $\rho = 0$, the lemma is trivially true so we assume that $\rho = 1$. The Gau{\ss} sum depends on $A_x$ only through its image in $M_{2 \times 2}(\mathcal{O}_F / \mathfrak{p}_F)$. Since $x \in S$, we know that $x_1^2 + y x_1 - b^2 \equiv 0 \pmod{\mathfrak{p}_F}$, and since the discriminant of this equation is $y^2 \Delta \equiv 0 \pmod{\mathfrak{p}_F}$, it follows that $x_1 \equiv -y/2 \pmod{\mathfrak{p}_F}$. Hence
\begin{equation*}
\begin{split}
  A_x & =
\begin{pmatrix}
y^{-1} (x_1^2- \zeta x_2^2 )& b_0 \zeta\\
b_0 \zeta &- \zeta y^{-1}(x_1^2 - \zeta x_2^2) \\
\end{pmatrix}\\
  & \equiv
\begin{pmatrix}
y^{-1} (\tfrac{y^2}{4}-\zeta b_0^2)&b_0 \zeta\\
b_0 \zeta&- \zeta y^{-1}(\tfrac{y^2}{4}-\zeta b_0^2) \\
\end{pmatrix} \pmod{\mathfrak{p}_F}\\
      & \equiv
\begin{pmatrix}
y/2&b_0 \zeta \\
b_0 \zeta & - \zeta y /2 \\
\end{pmatrix} \pmod{\mathfrak{p}_F}
\end{split}
\end{equation*}
where we have used that $\frac{y^2}{4}-\zeta b_0^2 = \frac{y^2}{4}(2-\Delta) \equiv \frac{y^2}{2} \pmod{\mathfrak{p}_F}$ to obtain the last congruence. The determinant of the last matrix is $- \frac{y^2 \zeta}{4}\Delta \equiv 0 \pmod{\mathfrak{p}_F}$ so it follows that the image of $A_x$ in $M_{2 \times 2}(\mathcal{O}_F / \mathfrak{p}_F)$ has rank $1$ and is independent of $x$. Moreover, all entries of $A_x$ are non-zero modulo $\mathfrak{p}_F$. By Lemma \ref{two_dimensional_gauss_sum}, it follows that $G(\varpi^{-1} A_x , \varpi^{-r-1} B_x)$ vanishes unless $x_1 + y^{-1} N_{E/F}(x) \in \mathfrak{p}_F^{r+1}$, in which case it equals $q^{-\frac{1}{2}} \gamma_F(2y,\frac{n}{2})$. We now consider two cases. In the first case, we assume that $v(\Delta)<r$ so $x_1 = x_0 + \alpha \varpi^{r-\delta}$ where $x_0$ is a solution to $x^2 + y x - b^2 =0$ (if $\Delta$ is a square, otherwise the Gau{\ss} sum vanishes). Since $\delta = \frac{1}{2} v(\Delta) < \frac{1}{2} r$, we have $\varpi^{2(r- \delta)} \in \mathfrak{p}_F^{r +1}$. Hence
\begin{equation*}
\begin{split}
  x_1 + y^{-1} N_{E/F}(x) & = x_0 + \alpha \varpi^{r-\delta} + y^{-1}((x_0 + \alpha \varpi^{r-\delta})^2- b^2)\\
                          & \equiv (1+2 y^{-1} x_0)\alpha \varpi^{r-\delta} \pmod{\mathfrak{p}_F^{r+1}}\\
  & = \pm \alpha Y \varpi^r
\end{split}
\end{equation*}
using that $x_0 = -\frac{y}{2} \pm \frac{1}{2} Y \varpi^\delta y$. Hence we need $\alpha \in \mathfrak{p}_F$ as desired.

We now assume that $v(\Delta) \geq r$. Then $x_1 = -\frac{1}{2} y + \alpha \varpi^{r - \delta}$, and hence
\begin{equation*}
\begin{split}
  x_1 + y^{-1} N_{E/F}(x) & = -\frac{1}{2} y + \alpha \varpi^{r-\delta}+  y^{-1} \left(\left(-\frac{1}{2} y + \alpha \varpi^{r-\delta}\right)^2 - b^2\right)\\
  & =  y^{-1} \alpha^2 \varpi^{2(r- \delta)}- \frac{y}{4}\Delta.
\end{split}
\end{equation*}
Hence we need $\alpha^2 \varpi^{2(r - \delta)} \equiv \frac{1}{4} y^2\Delta \pmod{\mathfrak{p}_F^{r+1}}$. To analyse this congruence, we must consider two subcases. Recall that $\delta = \lfloor \frac{r}{2} \rfloor$. The first subcase is $v(\Delta) = r$. If $r$ is odd, then $2(r- \delta)=r+1$ so we get $0 \equiv \Delta \pmod{\mathfrak{p}_F^{r+1}}$ which is impossible when $v(\Delta) = r$. If $r$ is even, the congruence reduces to $\alpha^2 \equiv \frac{\Delta y^2}{4 \varpi^r} \pmod{\mathfrak{p}_F}$ as desired. We now assume that $v(\Delta) \geq r+1$. Then the congruence reduces to $\alpha^2 \varpi^{2(r - \delta)} \equiv 0 \pmod{\mathfrak{p}_F^{r+1}}$. If $r$ is even, this congruence is equivalent to $\alpha \in \mathfrak{p}_F$, and if $r$ is odd, this congruence imposes no conditions on $\alpha$. The proof of the lemma is now complete.
\end{proof}

Before continuing with our evaluation of the Whittaker function, we make some remarks in the edge case when $\rho = 1$ and $v(\Delta)=r$, and explain how it can be merged with the case $v(\Delta)<r$. If $r$ is even, the Gau{\ss} sum vanishes unless $\alpha^2 \equiv \frac{\Delta y^2}{4 \varpi^r} \pmod{\mathfrak{p}_F}$. Since the right-hand side is a unit, this congruence has a solution if and only if the right-hand side is a square which in turn is equivalent to $\Delta$ being a square as $r$ is even. If $\Delta = \varpi^{r} Y_0^2$, we therefore need $\alpha \equiv \pm \frac{1}{2} Y_0 y \pmod{\mathfrak{p}_F}$, and consequently, $G(\varpi^{-1} A_x , \varpi^{-r-1} B_x)$ vanishes unless
\begin{equation*}
  x = -\frac{1}{2} y \pm \frac{1}{2} Y_0 \varpi^{\delta} y +b + \alpha ' \varpi^{r - \delta +1}
\end{equation*}
for some $\alpha' \in  \mathcal{O}_F / \mathfrak{p}_F^{\delta -1}$. If $r$ is odd, $\Delta$ is in particular not a square, and we have shown that $G(\varpi^{-1}A_x,\varpi^{-r-1}B_x)=0$. Motivated by these observations, we modify the definitions of $Y$ and $\delta$ slightly so that they now read
\begin{equation*}
Y :=
\begin{cases}
  Y_0 & \textrm{if $v(\Delta)<r+ \rho$, and $\Delta= Y_0^2 \varpi^{2 \delta_0}$ is a square with $v(\Delta )= 2 \delta_0$ }\\
  0   & \textrm{if $v(\Delta)<r+ \rho$, and $\Delta$ is not a square}\\
  0   & \textrm{if $v(\Delta) \geq r+ \rho$}
\end{cases},
\end{equation*}
and
\begin{equation*}
\delta =
\begin{cases}
  \delta_0 & \textrm{if $v(\Delta)<r+\rho$, and $\Delta= Y_0^2 \varpi^{2 \delta_0}$ is a square with $v(\Delta )= 2 \delta_0$ }\\
  0 & \textrm{if $v(\Delta)<r+ \rho$, and $\Delta$ is not a square}\\
  \lfloor \frac{r}{2} \rfloor & \textrm{if $v(\Delta) \geq r+ \rho$}
\end{cases}.
\end{equation*}
Moreover, we keep the same expressions for $A_{\pm}$ and $B_{\pm}$ using the new definitions of $Y$ and $\delta$ instead. We remark that $A_{\pm}, B_{\pm} \in \mathcal{O}_E^\times $. Indeed, $A_{\pm}=x_0+b$ for some $x_0 \in \mathcal{O}_F$ satisfying $x_0^2 + y x_0 - b^2 \equiv 0 \pmod{\mathfrak{p}_F}$, and since $\Delta \in \mathfrak{p}_F$, we have $x_0 \equiv -y/2 \pmod{\mathfrak{p}_F}$ so
\begin{equation*}
N_{E/F}(A_{\pm}) \equiv \frac{y^2}{4}- b^2 = \frac{y^2}{4}(2-\Delta) \equiv \frac{y^2}{2}\not\equiv 0 \pmod{\mathfrak{p}_F}.
\end{equation*}
By the lemma above, it follows that the support of the Gau{\ss} sum can be restricted to the set $S^\ast \subset S$ defined by
\begin{equation*}
S^\ast =
\begin{cases}
  \emptyset & \textrm{if $v(\Delta)<r+\rho$, and $\Delta$ is not a square}\\
  \left\{ A_{\pm}+ \alpha \varpi^{r- \delta + \eta}\,:\, \alpha \in \mathcal{O}_F / \mathfrak{p}_F^{\delta - \eta} \right\} & \textrm{otherwise}
\end{cases},
\end{equation*}
where
\begin{equation*}
\eta :=
\begin{cases}
  0 & \textrm{if $v(\Delta) \geq r + \rho$, $\rho = 1$, and $r$ is odd}\\
  \rho & \textrm{otherwise}
\end{cases}.
\end{equation*}
Now, the above lemma can be summarised more concisely as
\begin{equation*}
G(\varpi^{-\rho} A_x , \varpi^{-r-\rho}B_{x}) = q^{- \frac{\rho}{2}} \gamma_F(2y, \tfrac{n}{2}) \mathbf{1}_{x \in S^\ast}.
\end{equation*}
We can now give an expression for $W_\pi(g_{-n, \frac{n}{2}, y^{-1}})$ in terms of an oscillatory integral with cubic phase.

\begin{lemma}\label{lem:transition-cubic-integral}
Assume that $S^\ast \neq \emptyset$. Then the Whittaker function $W_\pi(g_{-n,\frac{n}{2}, y^{-1}})$ is given by the following expression:
\begin{multline*}
  \tilde{T}_{\Delta}(y) q^{\delta + \frac{\rho}{2}- \eta} \sum_{\pm} \xi^{-1}(A_{\pm}) \psi((\mathrm{Tr}(A_{\pm})+ y^{-1} N_{E/F }(A_{\pm}))\varpi^{-\frac{n}{2}})\\
  \cdot \int_{t \in \mathcal{O}_F } \psi \left( U_{\pm} t^3 + V_{\pm} t^2 + W_{\pm} t \right)\, d t
\end{multline*}
where
\begin{equation*}
\begin{split}
 \tilde{T}_{\Delta}(y) & :=
C_\pi \gamma_F(2y,\tfrac{n}{2})\begin{cases}
  1& \textrm{if $v(\Delta) < r+ \rho$}\\
  \frac{1}{2} & \textrm{if $v(\Delta) \geq r + \rho$}
\end{cases};\\
  U_{\pm} &:= -\frac{1}{3} \mathrm{Tr}_{E/F} \left( \frac{b}{A_{\pm}^3} \right)\varpi^{r-\rho-3 \delta + 3 \eta};\\
  V_{\pm} &:= \left( y^{-1} + \frac{1}{2} \mathrm{Tr}_{E/F}\left( \frac{b}{ A_{\pm }^2} \right) \right) \varpi^{-\rho - 2 \delta + 2 \eta};\\
  W_{\pm} &:= \left( 1 \pm Y \varpi^\delta - \mathrm{Tr}_{E/F} \left( \frac{b}{A_{\pm}} \right) \right) \varpi^{-r-\rho - \delta + \eta}.
\end{split}
\end{equation*} 
\end{lemma}

\begin{proof}
By Lemma \ref{lem:first-whittaker-formula} and the above expression for the Gau{\ss} sum, we find that $W_\pi(g_{-n, \frac{n}{2}, y^{-1}})$ is given by
\begin{multline*}
 \tilde{T}_{\Delta}(y) q^{\frac{\rho}{2}}  \sum_{\pm} \sum_{\alpha \in \mathcal{O}_F / \mathfrak{p}_F^{\delta-\eta}} \xi^{-1}(A_{\pm}+\alpha \varpi^{r- \delta + \eta})\\ \psi((\mathrm{Tr}_{E/F}(A_{\pm}+ \alpha \varpi^{r-\delta+\eta})+ y^{-1} N_{E/F}(A_{\pm}+ \alpha \varpi^{r-\delta+ \eta}))\varpi^{-\frac{n}{2}}),
\end{multline*}
where $\tilde{T}_{\Delta}(y)$ is as in the statement of the lemma. We now use the $p$-adic logarithm to express $\xi^{-1}$ in terms of $\psi$. We have
\begin{equation*}
\xi^{-1}(A_{\pm}+ \alpha \varpi^{r - \delta + \eta}) = \xi^{-1}(A_{\pm}) \xi^{-1} \left( 1+ \frac{\alpha \varpi^{r-\delta + \eta}}{A_{\pm}} \right),
\end{equation*}
and
\begin{equation*}
\begin{split}
\xi^{-1} \left( 1+ \frac{\alpha \varpi^{r- \delta + \eta}}{A_{\pm}} \right)& = \psi_E \left( -b \log_E \left(1+ \frac{\alpha \varpi^{r- \delta + \eta}}{A_{\pm}} \right) \varpi^{-\frac{n}{2}}\right)\\
                                                      & =  \psi_E \left( -b \left(\frac{\alpha \varpi^{r- \delta + \eta}}{ A_{\pm}}-\frac{\alpha^2 \varpi^{2r-2\delta + 2\eta}}{2A_{\pm}^2}+ \frac{\alpha^3 \varpi^{3r-3\delta + 3\eta}}{3A_{\pm}^3 }  \right)\varpi^{- \frac{n}{2}}\right)
 \end{split}
\end{equation*}
The Taylor expansion of the logarithm is convergent since $r - \delta + \eta \geq \frac{n}{8} \geq \kappa_F = \kappa_E$, and $A_{\pm} \in \mathcal{O}_E^\times$. Moreover, we claim that $j(r-\delta+\eta)- v(j) \geq \frac{n}{2}$ for all $j \geq 4$ so that we can truncate the expansion after the third term. Indeed, from the relation $r- \delta + \eta \geq \frac{n}{8}$, it follows that the inequality holds for $j=4$ as $p$ is odd. If $j \geq 5$, it is enough to have $n \geq \frac{8 v(j)}{j -4}$ because $r - \delta + \eta \geq \frac{n}{8}$. This inequality holds for $j=5$, since we assume $n \geq 8v(5)$. If $j \geq 6$, we may assume that $v(j)>0$, and hence $j \geq p$. It suffices to show that $\frac{8 v(j)}{j-4} \leq 16 \frac{v(p)}{p-1}$ since $n \geq 16 \kappa_F$. We have $v(j) \leq v(p) \log_p j$, so we reduce to showing
\begin{equation*}
\frac{\log j}{j-4} \leq \frac{2\log p}{p-1}.
\end{equation*}
The left hand side is a strictly decreasing function in $j$ for $j \geq 6$. When $p \geq 7$, the above holds for $j=p$ and hence for all $j \geq p$. If $p = 5$, the conditions $j \geq 6$ and $v(j)>0$ force $j \geq 10$, and we see that the above inequality holds for $j=10$ and hence all $j \geq 10$. If $p=3$, we see that the above holds for $j=6$, and hence for all $j \geq 6$. If $F = \mathbb{Q}_p$, the inequality $n \geq \frac{8 v(j)}{j-4}$ holds for $n \geq 8$ as $v(j) \leq \log j \leq j-4$ for $j \geq 6$, and $v(j)=1=j-4$ if $p=j=5$. If $F = \mathbb{Q}_p$, and $p \geq 7$, $v(j) \leq \log_7(j) \leq \frac{1}{2} \log j \leq \frac{1}{2} j -2$ so we only need $n \geq 4$.

Using $\psi_E := \psi \circ \mathrm{Tr}_{E/F}$, we find that the above is equal to
\begin{multline*}
 \psi \Bigg( - \mathrm{Tr}_{E/F}\left(\frac{b}{A_{\pm}}\right)\alpha \varpi^{-r-\rho-\delta + \eta}\\ + \frac{1}{2} \mathrm{Tr}_{E/F} \left( \frac{b}{A_{\pm}^2} \right) \alpha^2 \varpi^{-\rho-2 \delta + 2 \eta}-\frac{1}{3} \mathrm{Tr}_{E/F} \left( \frac{b}{ A_{\pm}^3}\right) \alpha^3 \varpi^{r-\rho - 3 \delta + 3 \eta}  \Bigg).
\end{multline*}
Next, we write the norm as
\begin{multline*}
N_{E/F}(A_{\pm} + \alpha \varpi^{r-\delta + \eta})\\=N_{E/F} (A_{\pm})+N_{E/F}(A_{\pm})\mathrm{Tr}_{E/F}\left( \frac{1}{A_{\pm}}\right)\alpha \varpi^{r- \delta + \eta}+ \alpha^2 \varpi^{2r - 2 \delta + 2 \eta}.
\end{multline*}
Now we notice that
\begin{equation*}
\begin{split}
  N_{E/F}(A_{\pm}) \mathrm{Tr}_{E/F} \left( \frac{1}{A_{\pm}} \right) &= A_{\pm} B_{\pm} \left( \frac{1}{A_{\pm}}+\frac{1}{B_{\pm}} \right)\\
                                                                      &= A_{\pm} + B_{\pm}\\
  & = -y \pm Y \varpi^\delta y.
\end{split}
\end{equation*}
Putting the above calculations together and writing the sum over $\alpha \in \mathcal{O}_F/ \mathfrak{p}_F^{\delta - \eta}$ as an integral over $t \in \mathcal{O}_F$ (introducing a factor of $q^{\delta - \eta}$) completes the proof of the lemma. 
\end{proof}

The task is now to simplify the cubic integral, and we start by determining the valuations of $U_{\pm}$, $V_{\pm}$ and $W_{\pm}$.

\begin{lemma}\label{lem:coefficient-valuations}
Assume that $S^\ast$ is non-empty. Then we have
\begin{equation*}
\begin{split}
  v(U_{\pm}) &= r- \rho -3 \delta+ 3 \eta  - v(3);\\
  v(V_{\pm}) &=
\begin{cases}
  -\rho - \delta+ 2 \eta  & \textrm{if }v(\Delta)<r+ \rho\\
  2 v(\Delta) -\rho- 2 \delta+ 2 \eta & \textrm{if }v(\Delta) \geq r+ \rho
\end{cases};\\
  v(W_{\pm}) &=
\begin{cases}
\infty & \textrm{if }v(\Delta)< r+ \rho\\
v(\Delta)-r- \rho - \delta+ \eta  & \textrm{if }v(\Delta) \geq r+ \rho\\
\end{cases}.
\end{split}
\end{equation*}
\end{lemma}

\begin{proof}
We start by computing $v(U_{\pm})$. The claimed expression for $v(U_{\pm})$ is equivalent to $\mathrm{Tr}_{E/F}(b_\xi/ A_{\pm}^3) \in \mathcal{O}_F^\times$ which in turn is equivalent to $B_{\pm}^3 - A_{\pm}^3 \in \mathcal{O}_E^\times$ since $A_{\pm},B_{\pm} \in \mathcal{O}_E^\times$. By a direct computation, $B_{\pm}^3 - A_{\pm}^3 \equiv -\frac{1}{2} y^2 b(2+\Delta) \equiv  -y^2 b\not\equiv 0 \pmod{\mathfrak{p}_E}$ because $\lvert \Delta \rvert<1$.
  
To determine valuations of $V_{\pm}$ and $W_{\pm}$, we are forced to fully expand their expressions. Since $A_{\pm }, B_{\pm} \in \mathcal{O}_E^\times$, the valuation of $V_{\pm}$ equals $-\rho - 2 \delta+ 2 \eta $ plus that of
\begin{multline*}
  y^{-1}(A_{\pm} B_{\pm})^2 + \frac{1}{2}b(B_{\pm}^2 -A_{\pm}^2)\\
  =\frac{1}{16}y^3 \left[ Y^4 \varpi^{4 \delta}\mp 4 Y^3 \varpi^{3 \delta}  +  Y^2 \varpi^{2 \delta}(8-2 \Delta)\mp 4 Y \varpi^\delta+ \Delta^2\right].
\end{multline*}
If $v(\Delta) \geq r+ \rho$, we have $Y=0$, and $v(V_{\pm})=v(\Delta^2)-\rho -2 \delta+ 2 \eta = 2 v(\Delta)- \rho - 2 \delta+ 2 \eta$ as desired. If $v(\Delta)<r + \rho$, $Y \varpi^{\delta}= \Delta^{\frac{1}{2}}$ is some square root of $\Delta$. Hence the above simplifies to $\frac{1}{16} y^3 (2\Delta^2 \mp 4\Delta^{\frac{3}{2}} +\Delta(8 - 2 \Delta) \mp 4\Delta^{\frac{1}{2}})$ which has valuation $\frac{1}{2} v(\Delta)= \delta$. Hence $v(V_{\pm})= -\rho - \delta+ 2 \eta $.

A similar calculation shows that the valuation of $W_{\pm}$ equals $-r-\rho-\delta+\eta$ plus that of
\begin{equation*}
  (1 \pm Y \varpi^{\delta}) A_{\pm} B_{\pm}-b(B_{\pm}-A_{\pm})
  = \frac{1}{4} y^2\left[ \pm Y^3 \varpi^{3 \delta}-Y^2 \varpi^{2 \delta}\mp Y \varpi^\delta \Delta+ \Delta\right] .
\end{equation*}
When $v(\Delta) \geq r+ \rho $, we have $Y=0$ so $v(W_{\pm})=v(\Delta)-r-\rho - \delta+ \eta$. If $v(\Delta)<r+ \rho$, we find that the above equals $0$.
\end{proof}

We now use the material from Section \ref{cubic_phase} to obtain the desired formulas for the Whittaker function. The two cases $v(\Delta)< r +\rho$ and $v(\Delta) \geq r + \rho$ must be considered separately.\\

\noindent
Assume first that $v(\Delta)< r+ \rho$. We further assume that $\Delta$ is a square since otherwise the Whittaker function vanishes. Hence $W_{\pm}= 0$, and $A_{\pm}= x_{\pm}+ b$ where $x_{\pm}$ denote the two solutions to the equation $x^2 + y x - b^2=0$. Since $W_{ \pm}= 0$, we can appeal to Lemma \ref{easy_cubic_integral}, but we must first check that $v(U_{\pm}) \geq 0$ when $v(V_{\pm}) \geq 0$, and that $v(V_{\pm})<0$ implies $v(3 U_{\pm})> v(V_{\pm})$, and $v(U_{\pm}) \geq 3 \lceil \frac{v(V_{\pm})}{2}\rceil$. Here we will need the assumption $n \geq 16 v(3)+4$, and we record this fact in a separate lemma.

\begin{lemma}\label{lem:valuation-lemma-3}
Assume that $v(\Delta)<r +\rho $, and $n \geq 16 v(3)+ 4$. Then $v(V_{\pm}) \geq 0$ implies $v(U_{\pm}) \geq 0$, and $v(V_{\pm})<0$ implies $v(3 U_{\pm})> v(V_{\pm})$, and $v(U_{\pm}) \geq 3 \lceil \frac{v(V_{\pm})}{2}\rceil$.
\end{lemma}

\begin{proof}
  Since $v(\Delta)< r + \rho$, we have $\eta = \rho$ so $v(U_{\pm}) = r + 2 \rho - 3 \delta - v(3)$, and $v(V_{\pm})=\rho - \delta$. Assume that $v(V_{\pm}) \geq 0$, i.e. $\rho \geq \delta$. Then $v(U_{\pm}) = r + 2 \rho - 3\delta - v(3) \geq r - \rho - v(3)$. This expression is trivially non-negative when $v(3)=0$, so assume $v(3)>0$. Since $r \geq \frac{n-2}{4}$, and $\rho \leq 1$, it follows that the last expression is non-negative for $n \geq 4 v(3)+6$ which is implied by $n \geq 16 v(3)+4$ when $v(3)>0$.

  We now assume that $v(V_{\pm})<0$. Recalling that $\delta = \frac{1}{2} v(\Delta)$ when $v(\Delta)< r + \rho$, we see that $v(3 U_{\pm})> v(V_{\pm})$ reduces to $v(\Delta)<r+\rho$ which is true. Hence it remains to show $v(U_{\pm}) \geq 3 \lceil \frac{v(V_{\pm})}{2}\rceil $. Let $v_1 := \delta - \rho$ and $v_2 := r- \rho$. The assumption $v(V_{\pm}) <0$ means that $v_1 \geq 1$, and the condition $v(\Delta) < r +\rho$ means that $2 v_1 < v_2$ (recall that $v(\Delta) = 2 \delta$). The desired inequality can be written as $v_2 - 3 v_1 - v(3) \geq 3 \lceil -\frac{v_1}{2}\rceil$ which is equivalent to $v_2 - v(3) \geq 3 \lceil \frac{v_1}{2}\rceil$.

  From the inequality $2 v_1< v_2$, we infer that $\lceil \frac{v_1}{2}\rceil \leq \lfloor \frac{v_2 +1}{4} \rfloor$ so it suffices that $v_2 - 3 \lfloor \frac{v_2 +1}{4} \rfloor \geq v(3)$. If we write $v_2$ as $4 q_2 + r_2$ with $q_2 \geq 0$ and $0 \leq r_2 <4$, we find that
\begin{equation*}
v_2 - 3 \Big\lfloor \frac{v_2 +1}{4} \Big\rfloor=
\begin{cases}
  q_2 & \textrm{if $r_2 =0,3$}\\
  q_2 + r_2 & \textrm{if $r_2 = 1,2$}
\end{cases}.
\end{equation*}
In particular, this expression is at least $q_2$ in all cases. If $v_2 \geq 4 v(3)$, we have $4q_2+ r_2 \geq 4v(3) $ which forces $q_2 \geq v(3)$ since $0 \leq r_2 <4$. Hence it suffices that $v_2 \geq 4v(3)$. Recalling that $v_2 = r - \rho$, and $\frac{n}{2}=2 r  + \rho$, we see that $n = 4 v_2 + 6 \rho$. Hence if $n \geq 16 v(3)+4$, we have $4 v_2 + 6 \rho \geq 16 v(3)+4$. If $\rho = 0$, we deduce that $v_2 \geq 4 v(3)$. If $\rho = 1$, we can upgrade the inequality to $4 v_2 + 6 \geq 16 v(3)+6$ because $4 v_2 + 6$ is even but not divisible by $4$, whereas $16 v(3)+4$ is divisible by $4$. Hence $n \geq 16 v(3)+ 4$ is a sufficient condition in all cases. 
\end{proof}

We can now apply Lemma \ref{easy_cubic_integral} to the cubic integral in Lemma \ref{lem:transition-cubic-integral}, and see that it equals $q^{ \frac{\rho-\delta}{2}} \gamma_F((V_{\pm})_0 , \tfrac{1}{2} v(\Delta)- \rho)$ where $(V_{\pm})_0$ denotes the unit part of $V_{\pm}$. We can compute this unit part explicitly. Up to a power of $\varpi$, $V_{\pm}$ is given by
\begin{equation*}
y^{-1} + \frac{1}{2} \mathrm{Tr}_{E/F} \left( \frac{b}{A_{\pm}} \right)= y^{-1} +\frac{b(B_{\pm}^2 - A_{\pm}^2)}{2(A_{\pm} B_{\pm})^2}.
\end{equation*}
Since $v(\Delta) < r + \rho$, we have $A_{\pm} = x_{\pm}+b$, and $B_{\pm}= x_{\pm}-b$ where $x_{\pm}= - \frac{1}{2} y \pm Y \varpi^\delta y$ is a solution to $x^2 + y x - b^2 =0$. Hence $B_{\pm}^2 -A_{\pm}^2= -4 b x_{\pm}$, and $A_{\pm}B_{\pm}=x_{\pm}^2 - b^2 = - yx_{\pm}$. Hence the above becomes
\begin{equation*}
y^{-1} - \frac{2b^2}{y^2 x_{\pm}} = y^{-1} - \frac{2 x_{\pm}(x_{\pm}+ y)}{y^2}= \mp y^{-1} Y \varpi^{\delta},
\end{equation*}
and we find that $(V_{\pm})_0 = \mp y^{-1} Y$. Since $\Delta = Y^2 \varpi^{2 \delta}$, we find that the value of $Y$ modulo $\mathfrak{p}_F$ is determined by the value of $y$ modulo $\mathfrak{p}_F^{1+ v(\Delta)}$. Hence $\gamma_F((V_{\pm})_0 , \tfrac{1}{2} v(\Delta)- \rho)$ depends only on $v(\Delta)$ and $y$ modulo $\mathfrak{p}_F^{1+ v(\Delta)}$.

To complete the evaluation of the Whittaker function, we note that $\mathrm{Tr}_{E/F}(A_{\pm})= 2 x_{\pm}$, and $N_{E/F}(A_{\pm})= x_{\pm}^2 - b^2 = -y x_{\pm}$. By Lemma \ref{lem:transition-cubic-integral}, it follows that
\begin{equation*}
W_\pi(g_{-n, \frac{n}{2}, y^{-1}}) = \mathbf{1}_{\Delta \in \mathcal{O}_F^2} \lvert \Delta \rvert^{-\frac{1}{4}} \sum_{\pm} T_{\Delta, \pm}(y) \xi^{-1}(x_{\pm}+b) \psi(x_{\pm} \varpi^{-\frac{n}{2}}) 
\end{equation*}
where $T_{\Delta , \pm }(y):=  \tilde{T}_{\Delta}(y) \gamma_F((V_{\pm})_0 , \frac{1}{2} v(\Delta)- \rho)$. The value of $T_{\Delta, \pm}(y)$ and its coupling with $\xi^{-1}(x_{\pm}+ b) \psi(x_{\pm }\varpi^{- \frac{n}{2}})$ depend only on $v(\Delta)$ and the value of $y$ modulo $\mathfrak{p}_F^{1+ v(\Delta)}$.\\

\noindent
Finally, we assume that $v(\Delta) \geq r+ \rho$. Then $A_{\pm}= -\frac{1}{2} y + b$ so $\mathrm{Tr}_{E/F}(A_{\pm})= -y$, and $N_{E/F}(A_{\pm })=\frac{1}{4} y^2 - b^2$. Moreover, it is clear from Lemma \ref{lem:coefficient-valuations} that $v(V_{\pm}) \geq 0$ so we can drop the quadratic term in Lemma \ref{lem:transition-cubic-integral}, and the cubic integral equals $q^{\frac{v(U)}{3}} \mathrm{Ai}(U,W)$ where $U := U_{\pm}$ and $W:= W_{\pm}$. Since $v(U)=r- \rho -3 \delta + 3 \eta - v(3)$, we find that
\begin{equation*}
W_{\pi}(g_{-n, \frac{n}{2}, y^{-1}})= T_{\Delta}(y) q^{\frac{n}{12}} \xi^{-1}(-\tfrac{1}{2} y + b) \psi((-\tfrac{3}{4}y- b^2 y^{-1} )\varpi^{- \frac{n}{2}}) \mathrm{Ai}(U(y), W(y))
\end{equation*}
where $T_{\Delta}(y)=2q^{-v(3)/3} \tilde{T}_{\Delta}(y) \ll_F 1$ which depends only on $\lvert \Delta \rvert$ and $y$ modulo $\mathfrak{p}_F$. By a direct calculation, we find that
\begin{equation*}
U(y)=\frac{32 b^2(3 y^2 + 4 b^2)}{3( y^2 - 4 b^2)^3} \varpi^{r- \rho  - 3 \delta + 3 \eta}\quad \textrm{and} \quad W(y) = \frac{y^2 + 4 b^2}{y^2 - 4 b^2}\varpi^{-r-\rho- \delta + \eta}.
\end{equation*}
As the last step, we bound $W_{\pi}(g_{-n, \frac{n}{2}, y^{-1}})$ using Proposition \ref{airy}. Assume first that $v(U)<0$. We have $v(W) < \frac{1}{3} v(U)$ if and only if $v(\Delta)< \frac{1}{3}(4r+ 2 \rho - v(3))\approx \frac{n}{3}$. Hence, for $q^{- \frac{n}{3}} \ll_F \lvert \Delta  \rvert \ll_F  q^{- \frac{n}{4}}$, it follows that the Whittaker function is bounded by
\begin{equation*}
q^{\frac{n}{12}-\frac{1}{12} v(U)+ \frac{1}{4} v(W)} \ll_F  q^{\frac{n}{12}- \frac{1}{12}(- \frac{n}{8})+ \frac{1}{4}(v(\Delta)- \frac{3n}{8})}= \lvert \Delta \rvert^{-\frac{1}{4}}.
\end{equation*}
When $v(W) \geq \frac{1}{3} v(U)$, it follows from the last part of Proposition \ref{airy} that $\mathrm{Ai}(U,W) \ll_F  1$ so the Whittaker function is bounded by $q^{\frac{n}{12}}$. To complete the proof, we must consider the case $v(U) \geq 0$. Then it is easy to see that $n \leq C_F $ for some $C_F$ that depends only on $F$, and we may recover the above bounds, potentially with worse implied constants.
\section{Dyadic decomposition}\label{dyadic_decomposition}
Recall that $f$ is a newform for $\Gamma_0(p^n)$, and that $f$ generates an automorphic representation $\pi \cong \otimes_v ' \pi_v$. We write $W_p$ for Whittaker newvector of $\pi_p$, and $W_b$ for the balanced variant, $W_b(g):= W_p(a(\varpi^{-n_2})g a(\varpi^{n_2}))$. Here we recall that $(n_1, n_2):=(\lfloor n/2 \rfloor,\lceil n/2\rceil)$. Given the formulas from the previous section, we dyadically decompose $W_b$. To translate the formulas for $W_p$ into formulas for $W_b$, we use the following relations which follow from the transformation properties of $W_p$. For $0 \leq \gamma \leq \infty$, we have
\begin{equation}\label{balanced_to_unbalanced1}
W_b \left(a(t)\begin{pmatrix}
    1&0\\
    \varpi^\gamma&1 \\
  \end{pmatrix}
    \right) = W_p \left( a(t)
               \begin{pmatrix}
                 1&0\\
                 \varpi^{n_2 + \gamma}&1 \\
               \end{pmatrix}
               \right).
\end{equation}
For $\gamma > n_2$, we have
\begin{equation}\label{balanced_to_unbalanced2}
 W_b \left( a(t)
    \begin{pmatrix}
      0&1\\
      1&\varpi^\gamma \\
    \end{pmatrix}
    \right) = W_p \left( a(\varpi^{-2n_2}t)
               \begin{pmatrix}
                 0&1\\
                 1&\varpi ^{\gamma -n_2} \\
               \end{pmatrix}
               \right),
\end{equation}
and for $1 \leq \gamma \leq n_2$, we have
\begin{equation} \label{balanced_to_unbalanced3}
W_b \left( a(t)
    \begin{pmatrix}
      0&1\\
      1&\varpi^\gamma \\
    \end{pmatrix} \right) = \psi_p(\varpi^{-n_2-\gamma}t)W_p \left( a(- \varpi^{-2 \gamma}t)
                             \begin{pmatrix}
                               1&0\\
                               \varpi^{n_2-\gamma}&1 \\
                             \end{pmatrix}\right)
\end{equation}
As in Table \ref{heuristic}, we can tabulate the behaviour of the balanced newvector:


\begin{table}[H]
\centering 
\renewcommand{\arraystretch}{2}
\caption{The behaviour of the balanced newvector on the matrices $a(t) \kappa$.}\label{heuristic_balanced}
\begin{tabular}{|c|c|c|c|c|c|}
\hline
$\kappa$                        & $\gamma$                             & \multicolumn{2}{|c|}{Support}                       & Type          & Max        \\
\hline             
\multirow{2}{*}{$\left(
		\begin{smallmatrix}
			0&1\\
			1&\varpi^\gamma \\
		\end{smallmatrix}
  \right)$}  & $1\leq \gamma \leq n_2$                      & \multicolumn{2}{|c|}{$\lvert t \rvert = q^{n_1 - n_2}$}                        & Osc.   & $\O\left(1\right)$       \\ \cline{2-6}
                         & $n_2 < \gamma \leq \infty$                     & \multicolumn{2}{|c|}{$\lvert t \rvert=q^{n_1 - n_2}$}                        & Const.     & $\O(1)$       \\               
\hline
\multirow{4}{*}{$\left(
			\begin{smallmatrix}
				1&0\\
				\varpi^\gamma&1 \\
			\end{smallmatrix}\right)$}  
		
		& $\gamma=0$, $n$ odd & \multicolumn{2}{|c|}{$\lvert t \rvert=1$}                                              & Osc.   & $\O(1)$       \\ \cline{2-6}
		
		& \multirow{3}{*}{ $\gamma= 0$, $n$ even}   &\multicolumn{2}{|c|}{$\pi_p$ principal series, $0<|t|<1$} & Osc.   & $\O(|t|^{\frac{1}{2}})$     \\ \cline{3-6}
		
                         &                                &\multirow{2}{*}{\makecell{$\lvert t\rvert=1$, and $\Delta \in (\mathcal{O})^2$\\if $q^{-\frac{n}{4}}<\lvert \Delta \rvert \leq 1$}} & $q^{-\frac{n}{3}} \ll \lvert \Delta \rvert \leq 1$     & Osc.   & $\O(\lvert \Delta \rvert^{-\frac{1}{4}})$    \\ \cline{4-6}
  
                         &                                &                                 & $ \lvert \Delta \rvert \ll q^{-\frac{n}{3}}$         & Airy          & $\O\left(q^{\frac{n}{12}}\right)$     \\ \cline{2-6}
  
                         & $0<\gamma < n_1$                      & \multicolumn{2}{|c|}{$\lvert t \rvert=1$}                        & Osc.    & $\O\left(1\right)$      \\ \cline{2-6}
  
                         & $n_1 \leq \gamma \leq \infty$                     & \multicolumn{2}{|c|}{$|t|=1$}                        & Const.      & $\O\left(1\right)$       \\  
\hline
\end{tabular}

\end{table}


\noindent
For $\kappa$ as in the table above, we write $\kappa W_b$ for the function $t\mapsto W_b(a(t) \kappa)$. Observe that when $\kappa \neq \left(
\begin{smallmatrix}
1&0\\
1&1 \\
\end{smallmatrix}
\right)$, $\kappa W_b$ is only supported on $\lvert t \rvert=1$ or $\lvert t \rvert = q^{n_1 - n_2}$. When $\kappa = \left(
\begin{smallmatrix}
1&0\\
1&1 \\
\end{smallmatrix}
\right)$, $\kappa W_b$ is supported on multiple dyadic ranges if $\pi_p$ is a principal series representation, and we dyadically decompose
\begin{equation*}
\kappa W_b = \sum_{s=0}^{n -1} \kappa W_b^s + \kappa W_b^{\infty}
\end{equation*}
where
\begin{equation*}
\kappa W_b^s := \mathbf{1}_{\lvert t \rvert=q^{-s}} \kappa W_b, \quad \textrm{and} \quad \kappa W_b^{\infty} := \mathbf{1}_{\lvert t \rvert \leq q^{-n}} \kappa W_b.
\end{equation*}
When $s=0$, $\pi_p$ can also be a dihedral supercuspidal representation attached to an unramified quadratic extension, and there are more cases to consider. Recalling that $\Delta := 1+ 4 b^2 t^{-2}$, we set $u:= v(\Delta)$ and decompose $\kappa W_b^0$ as follows:
\begin{equation*}
  \kappa W_b^0 =  \sum_{u=0}^{n_2-1} \kappa W_b^{0,u}+\kappa W_b^{0,\infty} \end{equation*}
where
\begin{equation*} \kappa W_b^{0,u}:= \mathbf{1}_{\lvert \Delta \rvert= q^{-u}} \kappa W_b^{0},\quad \textrm{and} \quad \kappa W_b^{0,\infty}:= \mathbf{1}_{\lvert \Delta \rvert \leq q^{-n_2}} \kappa W_b^0.
\end{equation*} 

\section{First estimates}\label{first_estimates_section}
Recall the estimate $
\lVert f \rVert_4^4\asymp \sum_{\kappa} p^{-\gamma}\mathcal{N}(\kappa W_b)
$ from the beginning of Section \ref{unfolding_section} where $\kappa$ runs through the matrices in Proposition \ref{representatives}. The purpose of this section is to estimate $\mathcal{N}(\kappa W_b)$. We start with some observations. If $\kappa = \left(
\begin{smallmatrix}
0&1\\
1&\varpi^\gamma \\
\end{smallmatrix}
\right)$ and the conductor exponent $n$ is odd, then we see from Table \ref{heuristic_balanced} that $\kappa W_b$ is supported only on $\lvert t \rvert = q^{-1}$. But from the definition of $\mathcal{N}(\kappa W_b)$ in Section \ref{unfolding_section}, we only integrate $\kappa W_b$ over elements of $\mathcal{O}_F$ so $\mathcal{N}(\kappa W_b)$ vanishes identically. Hence when $\kappa = \left(
\begin{smallmatrix}
0&1\\
1&\varpi^\gamma \\
\end{smallmatrix}
\right)$, we can assume that $n$ is even, and $\kappa W_b$ is supported on $\lvert t \rvert=1$.\\

\noindent
When $\kappa = \left(
\begin{smallmatrix}
1&0\\
1&1 \\
\end{smallmatrix}
\right)$, we have $\mathcal{N}(\kappa W_b) \prec \sum_{s=0}^{n_2-1} \mathcal{N}(\kappa W_b^s)+ \kappa W_b^{\infty}$ by the power mean inequality, and thus we only need to estimate $\mathcal{N}(W)$ when $W(t)$ is a function supported on a single dyadic range. To this end, it will be necessary to introduce the following variant of $\mathcal{N}$. Suppose $W: F^\times \rightarrow \mathbb{C}$ has support contained in $\lvert t \rvert = q^{-s}$ for some non-negative integer $s$. For $a \in \mathbb{F}_p^\times$, we define $\mathcal{N}_a(W)$ by the same expression as $\mathcal{N}$ but with summation over $m$ restricted to $m \equiv a p^s \pmod{p^{s+1}}$. By the power mean inequality, $\mathcal{N}(W) \ll  \sum_{a \in \mathbb{F}_p^\times} \mathcal{N}_a(W)$, so it is sufficient to estimate each $\mathcal{N}_a(W)$ individually, and it is easy to see that the unfolding arguments in Section \ref{unfolding_section} hold with $\mathcal{N}_a$ in place of $\mathcal{N}$. The point of working with $\mathcal{N}_a(W)$ instead of $\mathcal{N}(W)$ is that, after unfolding, we have now fixed the valuation of $m_1 + m_2$ to be $s$ when $W$ is supported on $\lvert t \rvert=q^{-s}$. This makes it easier to handle the integrals $I_p(\mathbf{m})$, and, in some cases, it even seems essential for getting the right estimates.

Returning to the fourth moment of $f$, we have the following bound on $\lVert f \rVert_4^{4}$

\begin{multline}\label{first_bound}
(p^{n})^\varepsilon \sum_{a \in \mathbb{F}_p^\times} \Bigg[ \sum_{\kappa \neq \left(
\begin{smallmatrix}
1&0\\
1&1 \\
\end{smallmatrix}
\right)}p^{-\gamma} \mathcal{N}_a(\kappa W_b) + \sum_{s=1}^{n-1} \mathcal{N}_a\left( \left(
\begin{smallmatrix}
1&0\\
1&1 \\
\end{smallmatrix}
\right) W_b^s  \right)+ \mathcal{N}_a \left( \left(
\begin{smallmatrix}
1&0\\
1&1 \\
\end{smallmatrix}
\right) W_b^{\infty} \right) \\+ \sum_{u < n_2} \mathcal{N}_a(\left(
\begin{smallmatrix}
1&0\\
1&1 \\
\end{smallmatrix}
\right) W_b^{0,u}) + \mathcal{N}_a(\kappa W_b^{0,\infty}) \Bigg]
\end{multline}
for any $\varepsilon >0$. The purpose of this section is to estimate $\mathcal{N}_a(W)$ given estimates for the integrals $I_p(\mathbf{m})$, and the starting point is Lemma \ref{archimedean_lemma}. We recall that for a measureable function $W: F^\times \rightarrow \mathbb{C}$ and a tuple of positive integers $\mathbf{m} =(m_1 , m_2 , m_3 , m_4)$ satisfying $m_1 + m_2 = m_3 + m_4$, we set
\begin{equation}\label{definition_of_I_p}
I_p(\mathbf{m}):= \int_{t \in \mathcal{O}^\times} W(m_1 t) W(m_2 t) \overline{W(m_3 t) W(m_4 t)} d^\times t.
\end{equation}
The following lemma explains how to estimate $\mathcal{N}_a(W)$, if one knows that $I_p(\mathbf{m})$ vanishes for sufficiently many tuples $\mathbf{m}$. To make this more precise, we observe that if $m_1 + m_2 = m_3 + m_4$, then $(m_3 - m_1)(m_3 - m_2)= m_1 m_2 - m_3 m_4$, and we define
\begin{equation*}
v_{\mathbf{m}}:= v(m_3 - m_1)+v(m_3 - m_2)= v(m_1 m_2 - m_3 m_4).
\end{equation*}
Since $m_1 + m_2 = m_3 + m_4$ and $m_1 m_2 = m_3 m_4$ implies $\left\{ m_1, m_2 \right\}= \left\{ m_3 , m_4 \right\}$, the number $v_{\mathbf{m}}$ measures how far $\mathbf{m}$ is from the ``diagonal'' $\left\{ m_1 , m_2 \right\}= \left\{ m_3 , m_4 \right\}$ where there is no hope of finding cancellation in \eqref{definition_of_I_p}. The lemma below shows how $\mathcal{N}_a(W)$ can be estimated if $I_p(\mathbf{m})$ vanishes for all $\mathbf{m}$ with $v_{\mathbf{m}}<v_0$ for some fixed non-negative integer $v_0$, i.e. if we are sufficiently far away from the diagonal.

\begin{lemma}\label{first_estimates}
Let $W: F^\times \rightarrow  \mathbb{C}$ be measurable, and suppose $\lVert W \rVert _\infty \ll(p^n)^A$ for some constant $A>0$. Assume that $v_0 \leq n$ is a non-negative integer such that $I_p(\mathbf{m})$ vanishes when $v_{\mathbf{m}}< v_0$. Then the following estimates hold:
\begin{enumerate}
\item If $W = \kappa W_b$ where $\kappa = \left(
\begin{smallmatrix}
1&0\\
\varpi^{\gamma}&1 \\
\end{smallmatrix}
\right)$ or $\kappa = \left(
\begin{smallmatrix}
0&1\\
1&\varpi^\gamma \\
\end{smallmatrix}
\right)$ with $1 \leq \gamma \leq \infty$, then
\begin{equation*}
\mathcal{N}_a(\kappa W_b) \prec 1+ p^{n_2-v_0} \lVert \kappa W_b \rVert _\infty^4.
\end{equation*}
\item If $W = \kappa W_b^s$ where $\kappa = \left(
\begin{smallmatrix}
1&0\\
1&1 \\
\end{smallmatrix}
\right)$ and $1 \leq s < n$, then
\begin{equation*}
\mathcal{N}_a(\kappa W_b^s) \prec 1+p^{n_2 +s- v_0} \lVert \kappa W_b^s \rVert _\infty^4.
\end{equation*}
\item If $W= \kappa W_b^{\infty}$ where $\kappa = \left(
\begin{smallmatrix}
1&0\\
1&1 \\
\end{smallmatrix}
\right)$, then
\begin{equation*}
\mathcal{N}(\kappa W_b^{\infty}) \prec 1.
\end{equation*}
\item If $W = \kappa W_b^{0,u}$ where $\kappa = \left(
\begin{smallmatrix}
1&0\\
1&1 \\
\end{smallmatrix}
\right)$ and $0 \leq u <n_2$, then
\begin{equation*}
\mathcal{N}_a(\kappa W_b^{0,u}) \prec 1+p^{n_2 - u - \max \left\{ v_0, 2u \right\}} \lVert \kappa W_b^{0,u} \rVert _\infty^4.
\end{equation*}
\item If $W = \kappa W_b^{0,\infty}$ where $\kappa = \left(
\begin{smallmatrix}
1&0\\
1&1 \\
\end{smallmatrix}
\right)$, then
\begin{equation*}
\mathcal{N}_a(\kappa W_b^{0,\infty})\prec 1+ p^{- n_2} \lVert \kappa W_b^{0,\infty} \rVert _\infty^4.
\end{equation*}
\end{enumerate}
\end{lemma}

\begin{proof}
\phantom{}
\begin{enumerate}
\item As explained in the beginning of this section, we can assume that when $\kappa = \left(
\begin{smallmatrix}
0&1\\
1&\varpi^\gamma \\
\end{smallmatrix}
\right)$ then $n$ is even, and $\kappa W_b$ has support contained in $\lvert t \rvert=1$. Let $\varepsilon >0$ be given. By Lemma \ref{archimedean_lemma}, we have
  \begin{equation*}
\mathcal{N}_a(\kappa W_b) \ll_\varepsilon(p^n)^\varepsilon \left[ 1+ p^{- n_2} \sum_{\substack{
      m_1 \leq 2 p^{(1+\varepsilon)n_2}\\ m_j \in [ 2^{-1} m_1 , 2 m_1 ]\\
      p \nmid m_1 , m_2 , m_3 , m_4\\
      m_1 + m_2 = m_3 + m_4
}} \frac{\lambda(m_1)^4}{m_1} \lvert I_p(\mathbf{m}) \rvert \right].
\end{equation*}
Here we have used that $\kappa W_b(t)$ is only supported on $\lvert t \rvert = 1$, so if $\mathbf{m} =(m_1 , m_2 , m_3 , m_4)$ and $p \mid m_j$ for some $j$, then $I_p(\mathbf{m})=0$. Hence the sum can be restricted to $m_j$ not divisble by $p$. We omit the condition that $m_j \equiv a \pmod{p}$ for all $j=1,2,3,4$ since this will only give a saving of size $p$. We can bound $I_p(\mathbf{m})$ by $\lVert \kappa W_b \rVert _\infty^4 \mathbf{1}_{v_{\mathbf{m}} \geq v_0}$ with the result that the sum above is bounded by
\begin{equation}\label{eq:double-sum}
p^{-n_2} \lVert \kappa W_b \rVert _\infty^4  \sum_{\substack{
    m_1 \leq 2 p^{(1 + \varepsilon) n_2}\\
    p \nmid m_1
}} \frac{\lambda(m_1)^4}{m_1} \sum_{m_2 \in  [2^{-1} m_1 , 2 m_1]} N(m_1 , m_2)
\end{equation}
where $N(m_1 , m_2) := \#\left\{ m_3 \in [2^{-1}  m_1, 2 m_1 ]\,:\, v(m_3 - m_1)+v(m_3 - m_2) \geq v_0 \right\}$. We estimate $N(m_1 , m_2)$ using a dyadic decomposition. Since $v(m_1 - m_2)$ is at least the minimum of $v(m_3 - m_1)$ and $v(m_3 - m_2)$, $N(m_1 , m_2)$ is bounded by 
\begin{equation*}
\sum_{\substack{
  c+d = v_0\\
  \min \left\{ c,d \right\} \leq v(m_1 - m_2)
}} \# \left\{ m_3 \in [2^{-1}  m_1, 2 m_1 ] \,:\, v(m_3 - m_1) \geq c \textrm{ and } v(m_3 - m_2) \geq d \right\}
\end{equation*}
\begin{equation*}
  \ll \sum_{\substack{
    c+d=v_0 \\
  \min \left\{ c,d \right\} \leq v(m_1 - m_2)
}} p^{- \max \left\{ c,d \right\}} m_1.
\end{equation*}
We can write $\max \left\{ c,d \right\}= c+d - \min \left\{ c,d \right\}$. If $m_1 \neq m_2$, this is bounded below by $v_0-v(m_1 - m_2)$, and since $v_0 \prec 1$, this leads to the estimate $N(m_1 , m_2) \ll_\varepsilon(p^n)^\varepsilon m_1 p^{-v_0 + v(m_1 - m_2)}$. If $m_1 = m_2$ (so $v(m_1 - m_2)=\infty$), we instead have $\max  \left\{ c,d \right\} \geq \frac{1}{2} v_0$, leading to $N(m_1 , m_2) \ll_\varepsilon(p^n)^\varepsilon m_1 p^{-\frac{1}{2} v_0}$. We conclude that the double-sum in Equation \eqref{eq:double-sum} is bounded by 
\begin{equation*}
(p^n)^\varepsilon \sum_{\substack{
    m_1 \leq 2 p^{(1+\varepsilon)n_2}\\
    p \nmid m_1
}} \lambda(m_1)^4 \left[ p^{-\frac{1}{2} v_0}+p^{-v_0} \sum_{0 \leq v \leq (1+ \varepsilon ) n_2} p^v \sum_{\substack{
    m_2 \in [2^{-1}  m_1, 2m_1 ] \\
    v(m_1 - m_2)=v
}}1 \right]
\end{equation*}
for any $\varepsilon >0$. The number of $m_2 \in [2^{-1}m_1, 2 m_1 ]$ satisfying $v(m_1 - m_2) = v$ is $\ll m_1 p^{-v} \ll p^{n_2(1+\varepsilon)} p^{-v}$ so the the double sum in the square brackets above is $\ll_\varepsilon(p^n)^\varepsilon p^{n_2-v_0}$, and since $-\frac{1}{2} v_0 \leq n_2 - v_0$, we conclude that the above is bounded by 
\begin{equation*}
 (p^n)^\varepsilon p^{n_2 - v_0} \sum_{\substack{
      m_1 \leq 2 p^{(1+\varepsilon) n_2}\\
      p \nmid m_1
}} \lambda(m_1)^4. 
\end{equation*}
for all $\varepsilon >0$. By the fourth moment bound \eqref{fourth_moment} on the Hecke eigenvalues, the proof is complete.

\item Because $\kappa W_b^s(t)$ is only supported on $v(t)=s$, $I_p(\mathbf{m})$ vanishes unless $v(m_j)=s$ for all $j=1,2,3,4$. If we write $m_j = p^s m_j '$ where $p \nmid m_j$ then $v_{\mathbf{m}}=2s+v_{\mathbf{m} '}$ where $\mathbf{m} ' =(m_1', m_2 ' , m_3 ' , m_4 ')$. The sum on the right hand side in Lemma \ref{archimedean_lemma} becomes
\begin{equation*}
p^{-n_2 +s} \sum_{\substack{
    m_1' \leq 2 p^{(1+ \varepsilon )n_2 -s}\\ m_j' \in [2^{-1}  m_1', 2 m_1' ]\\
    p \nmid m_1' , m_2' , m_3' , m_4' \\
    m_1' + m_2' =  m_3' + m_4'
}} \frac{\lambda(m_1')^4}{ m_1' } \lvert I_p(\mathbf{m}') \rvert.
\end{equation*}
The proof is now identical to case (i) with $n_2 -s$ in place of $n_2$ and $v_0 -2s$ in place of $v_0$.

\item If follows from the definition of $\kappa W_b^{\infty}$ that $I_p(\mathbf{m})$ vanishes unless $p^n$ divides all of $m_1$, $m_2$, $m_3$ and $m_4$. But in Lemma \ref{archimedean_lemma}, we only sum over $m_j \ll p^{(1+ \varepsilon) n_2}$ for $j=1,2,3,4$ so the sum is in fact empty, and we get desired the bound $\mathcal{N}_a(\kappa W_b^\infty) \prec 1$.   
\item This time, $\kappa W_b^{0,u}(t)$ is only supported on $\left\{ t \in \mathcal{O}^\times\,:\, v(t^2 + 4 b^2 )=u\right\}$. For the integrand in $I_p(\mathbf{m})$ to be supported we therefore need $m_1 , m_2 , m_3 , m_4 \in \mathcal{O}^\times$, and $m_1^2 \equiv m_2^2 \equiv m_3^2 \equiv m_4^2 \pmod{p^u}$. Since we work with $\mathcal{N}_a$ instead of $\mathcal{N}$, we have $m_j \equiv a \pmod{p}$ for $j=1,2,3,4$ so this forces $m_1 \equiv m_2 \equiv m_3 \equiv m_4 \pmod{p^u}$. Now the support of the integrand is contained in $\left\{ t \in \mathcal{O}^\times \,:\, v((m_1 t)^2 + 4b^2) = u \right\}$ which has volume $\ll p^{-u}$ with respect to $d^\times t$. Hence $\mathcal{N}_a(\kappa W_b^{0,u})$ is bounded by 
\begin{equation}\label{first_estimates_case_(3)}
(p^n)^\varepsilon \left[ 1 + p^{-n_2 - u} \lVert \kappa W_b^{0,u} \rVert^4 \sum_{\substack{
      m_1 \leq 2 p^{(1+\varepsilon) n_2}\\ m_j \in [2^{-1}  m_1, 2m_1 ]\\
      m_j \equiv a \pmod{p}\\
      m_j \equiv m_1 \pmod{p^u}\\
      m_1 + m_2 = m_3 + m_4
}} \frac{\lambda(m_1)^4}{m_1}\right]
\end{equation}
for all $\varepsilon >0$. Since the $m_j$ are pairwise congruent modulo $p^{u}$, it follows that $v_{\mathbf{m}} \geq 2u$. We now bound the sum inside the square brackets by
\begin{equation*}
\sum_{\substack{
    m_1 \leq 2 p^{(1+\varepsilon) n_2}\\
    p \nmid m_1
  }} \frac{\lambda(m_1)^4}{m_1} \sum_{\substack{m_2 \in [2^{-1} m_1, 2 m_1 ]\\ v(m_1 - m_2) \geq u}} N(m_1 , m_2)
\end{equation*}
where, this time, $N(m_1 , m_2)$ is defined as the cardinality of the set
\begin{equation*}
 \left\{ m_3 \in [2^{-1} m_1, 2 m_1 ]\,:\, v(m_3 - m_1)+v(m_3 - m_2)\geq \max \left\{ v_0 , 2u \right\} \right\}.
\end{equation*}
Arguing as in case (ii) with $\max \left\{ 2u, v_0 \right\}$ in place of $v_0$, we find that the above is
\begin{equation*}
\begin{split}
& \ll_\varepsilon(p^n)^\varepsilon \sum_{\substack{
    m_1 \leq 2 p^{(1+ \varepsilon)n_2}\\
    p \nmid m_1
}} \lambda(m_1)^4 \left[ p^{-\frac{1}{2} \max \left\{ v_0 , 2u \right\}}+p^{n_2 - \max \left\{v_0 , 2u \right\}} \right] \\ & \ll(p^n)^\varepsilon p^{n_2 - \max \left\{ v_0 ,2u \right\}} \sum_{\substack{
  m_1 \leq 2 p^{(1+\varepsilon)n_2}\\
  p \nmid m_1
}} \lambda(m_1)^4
\end{split}
\end{equation*}
since the conditions $u < n_2$ and $v_0 \leq n$ imply that $n_2 - \max \left\{ v_0 , 2u \right\} \geq - \frac{1}{2} \max \left\{ v_0 , 2u \right\}$. Using the fourth moment bound on the Hecke eigenvalues completes the proof in this case.

\item In this case, we get a similar estimate for $\mathcal{N}_a(\kappa W_b^{0,\infty})$ as in \eqref{first_estimates_case_(3)}, but with $n_2$ in place of $u$, and $\lVert \kappa W_b^{0,\infty} \rVert _\infty $ in place of $\lVert \kappa W_b^{0,u} \rVert _\infty$. The conditions $m_1 \ll p^{(1+\varepsilon) n_2}$, $m_j \asymp m_1$ and  $m_j \equiv m_1\pmod{p^{n_2}}$ imply that there are only $ \ll p^{\varepsilon n_2 } $ possibilities for each of $m_1$, $m_2$ and $m_3$ given $m_1$. Since $\lambda(m_1)^4 \ll  m_1$, the sum inside the square brackets is $\ll_\varepsilon(p^{n})^\varepsilon p^{ n_2}$ so the desired estimate follows.
\end{enumerate}
\end{proof}

\section{Estimating via size and support}\label{estimating_via_size_and_support}
In section we estimate using Lemma \ref{first_estimates} only taking into account the size and support properties of the $p$-adic Whittaker function, and not the cancellation that might occur in $I_p(\mathbf{m})$. We will only make reference to Table \ref{heuristic_balanced}, and not the explicit formulas for the Whittaker function in Theorem \ref{whittaker_formulas}. 

\begin{lemma}
We have the following estimates:
\begin{enumerate}
\item For $\kappa = \left(
\begin{smallmatrix}
1&0\\
\varpi^\gamma&1 \\
\end{smallmatrix}
\right)$ or $\kappa=\left(
\begin{smallmatrix}
0&1\\
1&\varpi^\gamma \\
\end{smallmatrix}
\right)$ where $1 \leq \gamma \leq \infty$, we have $\mathcal{N}_a (\kappa W_b) \prec p^{n_2}$.
\item For $\kappa = \left(
\begin{smallmatrix}
1&0\\
1&1 \\
\end{smallmatrix}
\right)$ and, $\frac{n}{6} \leq s <n$ or $s=\infty$, we have $\mathcal{N}_a (\kappa W_b^s) \prec 1$.
\item For $\kappa = \left(
\begin{smallmatrix}
1&0\\
1&1 \\
\end{smallmatrix}
\right)$ and $u \geq \frac{n}{4}$, we have $\mathcal{N}_a(\kappa W_b^{0,u}) \prec 1$.
\item For $\kappa = \left(
\begin{smallmatrix}
1&0\\
1&1 \\
\end{smallmatrix}
\right)$, we have $\mathcal{N}_a (\kappa W_b^{0,\infty}) \prec 1$. 
\end{enumerate}
\end{lemma}

\begin{proof}
\phantom{}
\begin{enumerate}
\item From Table \ref{heuristic_balanced}, we see that $\lVert \kappa W_b \rVert _\infty \ll 1$ so the claim follows from Lemma \ref{first_estimates} with $v_0 = 0$.
\item In Lemma \ref{first_estimates}, we have already proved that $\mathcal{N}_a(\kappa W_b^\infty)\prec 1$ so assume $\frac{n}{6}\leq s<n$. Consulting Table \ref{heuristic_balanced} again, we see that $\lVert \kappa W_b^s \rVert _\infty \ll p^{-\frac{s}{2}}$. Since $\kappa W_b^{s}(t)$  is only supported on $v(t) =s$, it follows that $I_p(\mathbf{m})$ vanishes unless $v(m_j)=s$ for $j=1,2,3,4$. In this case $v_{\mathbf{m}} \geq 2s$ so we can apply Lemma \ref{first_estimates} with $v_0 =2s$ to obtain
\begin{equation*}
\mathcal{N}_a(\kappa W_b^{s}) \prec 1 + p^{n_2 +s - 2s}(p^{-\frac{s}{2}})^4 \ll 1
\end{equation*}
for $s \geq \frac{n}{6}$ as desired (recall that $n_2 :=\lceil \frac{n}{2} \rceil$).
\item Here we have two subcases depending on $u= v(\Delta)$ where $\Delta = 1+4 b^2 t^{-2}$. If $p^{-\frac{n}{3}} \ll \lvert \Delta \rvert \ll p^{-\frac{n}{4}}$, we see from Table \ref{heuristic_balanced} that $\lVert \kappa W_b^{0,u} \rVert _\infty \ll p^{\frac{u}{4}}$. Applying Lemma \ref{first_estimates} with $v_0 =0$ gives
  \begin{equation*}
\mathcal{N}_a(k W_b^{0,u}) \prec 1 + p^{n_2 -3u}(p^{\frac{u}{4}})^4 \ll 1
\end{equation*}
when $u \geq \frac{n}{4}$. The second case is $p^{-n_2} \ll \lvert \Delta  \rvert \ll p^{-\frac{n}{3}}$ where $\lVert \kappa W_b^{0,u} \rVert _\infty \ll p^{\frac{n}{12}}$. Then Lemma \ref{first_estimates} with $v_0 =0$ gives
\begin{equation*}
\mathcal{N}_a(\kappa W_b^{0,u}) \prec 1 +  p^{n_2 - 3u}(p^{\frac{n}{12}})^4 \ll 1+p^{\frac{5n}{6}-3u} \ll 1
\end{equation*}
since $\frac{5n}{6}-3u <0$ when $u \geq \frac{n}{3}$.
\item We again have $\lVert \kappa W_b^{0,\infty} \rVert _\infty  \ll p^{\frac{n}{12}}$ so applying Lemma \ref{first_estimates} with $v_0 =0$ gives
\begin{equation*}
\mathcal{N}_a(\kappa W_b^{0,\infty}) \prec 1 + p^{-n_2}(p^{\frac{n}{12}})^4 \ll 1
\end{equation*}
which completes the proof. 
\end{enumerate}
\end{proof}

\noindent
From \eqref{first_bound}, it follows that we have reduced the $L^4$-norm bound to the estimates below. In these cases, it turns out the Whittaker function takes a very specific form which allows us to prove vanishing of $I_p(\mathbf{m})$ for sufficiently many tuples $\mathbf{m}$. 

\begin{proposition}\label{final_estimates}
We have the following estimates
\begin{enumerate}
\item For $\kappa = \left(
\begin{smallmatrix}
1&0\\
\varpi^\gamma&1 \\
\end{smallmatrix}
\right)$ or $\kappa = \left(
\begin{smallmatrix}
0&1\\
1&\varpi^\gamma \\
\end{smallmatrix}
\right)$, and $\gamma < n_2$, we have $\mathcal{N}_a(\kappa W_b) \prec p^\gamma$.
\item For $\kappa = \left(
\begin{smallmatrix}
1&0\\
1&1 \\
\end{smallmatrix}
\right)$ and $0<s<\frac{n}{6}$, we have $\mathcal{N}_a(\kappa W_b^s) \prec 1$.
\item For $\kappa = \left(
\begin{smallmatrix}
1&0\\
1&1 \\
\end{smallmatrix}
\right)$ and $0 \leq u <\frac{n}{4}$, we have $\mathcal{N}_a(\kappa W_b^{0,u}) \prec 1$. 
\end{enumerate} 
\end{proposition}

\section{The strategy for proving Proposition \ref{final_estimates}}\label{the_strategy_for_proving_proposition_11.2}

We briefly explain how we obtain the estimates in Proposition \ref{final_estimates}. We use Lemma \ref{first_estimates}, but this time taking into account when $I_p(\mathbf{m})$ vanishes. In each case, we can exhibit vanishing of $I_p(\mathbf{m})$ when $\mathbf{m}$ is sufficiently far away from the diagonal $\left\{ m_1 , m_2 \right\}= \left\{ m_3 , m_4 \right\}$. The proof essentially starts by using the explicit formulas for the Whittaker function to show that $I_p(\mathbf{m})$ is a finite linear combination of integrals of the form
\begin{equation*}
\int_{z \in \mathfrak{p}} \psi(\varpi^{-n_2}\Phi(z)) d z
\end{equation*}
where $\Phi(z) \in F [[z]]$ is a power series whose coefficients can be expressed in terms of $\mathbf{m}$. Then some very careful calculations reveals that $\Phi(z)= \sum_{k \geq 0} \frac{1}{k}a_k z^k$ where $v(a_1) \leq v(a_k)$ for all $k$, and $v(a_1)$ can be bounded above in terms of $v_{\mathbf{m}}$. By the lemma below, we can simply substitute away $\Phi(z)$ to obtain an integral of the form
\begin{equation*}
\int_{z \in \mathfrak{p}} \psi(\varpi^{- n_2 + v(a_1)} z) d z
\end{equation*}
which vanishes when $-n_2 + v(a_1) <-1$. This condition will translate into $v_{\mathbf{m}} < v_0$ for some $v_0$ depending only on $n_2$ and the parameters $\gamma$, $s$ and $u$ in the statement of Proposition \ref{final_estimates}. We can then apply Lemma \ref{first_estimates} to this value $v_0$ and deduce the desired estimates.

\begin{lemma}\label{non_linear_substitution}
Let $\Phi(z)= \sum_{k \geq 0} \frac{a_k}{k !} z^k \in F [[z]]$ be a power series with coefficients satisfying $v(a_k) \geq v(a_1)$ for all $k \geq 1$. Then
  \begin{equation*}
\int_{z \in \mathfrak{p}} \psi(\Phi(z)) d z =
\begin{cases}
  0 & \textrm{if }v(a_1) < -1\\
  q^{-1} \psi (\Phi(0)) & \textrm{otherwise}
\end{cases}.
\end{equation*}
\end{lemma}

\begin{proof}
Write $\Phi(z)= \Phi(0) + \varpi^{v(a_1)} \sum_{k \geq 1}\frac{a_k '}{k!} z^k$ where $a_k ' = \varpi^{- v(a_k)} a_k$. Then $a_1 ' \in \mathcal{O}^\times$, and $a_k \in \mathcal{O}$ for all $k \geq 1$, and it is well know that there is a power series which is both a left- and right-inverse to $\sum_{n \geq 1} \frac{a_k'}{k!} z^k$ with respect to composition, see \cite[Ch. IV, Lemma 2.4 and Lemma 5.4]{silverman}. Hence the map $k\mapsto \sum_{k \geq 1} \frac{a_k'}{k!} z^k$ is a measure preserving homeomorphism $\mathfrak{p} \rightarrow \mathfrak{p}$ (convergence follows from \cite[Ch. IV, Lemma 6.3]{silverman} and measure-preserving from $a_k ' \in \mathcal{O}^\times$). Hence we can substitute $\sum_{k \geq 1} \frac{a_k '}{k !}z^k \mapsto z$ to obtain
\begin{equation*}
\int_{z \in \mathfrak{p}} \psi(\Phi(z)) d z = \psi(\Phi(0)) \int_{z \in \mathfrak{p}} \psi(\varpi^{v(a_n)}z) d z,
\end{equation*}
and the result is now clear.
\end{proof}

Proving the inequalities $v( a_1 ) \leq v(a_k)$ in each of the three cases will require a detailed analysis of the explicit formulas for the coefficients $a_k$. They turn out to be linear combinations of expressions of the type $\alpha_1^l+ \alpha_2^l-\alpha_3^l - \alpha_4^l$ where $\alpha_1, \alpha_2 , \alpha_3 , \alpha_4$ can be described in terms of $m_1 , m_2 , m_3 , m_4 $. In all three cases, we will need the following elementary lemma to handle valuations of such expressions.

\begin{lemma}\label{elementary}
Let $\alpha_1, \alpha _2, \alpha _3 , \alpha _4 \in \mathcal{O}^\times$. Then
\begin{equation*}
\min \left\{ v(\alpha_1 + \alpha_2 - \alpha_3 - \alpha_4), v(\alpha_1^{-1} + \alpha_2^{-1} - \alpha_3^{-1} - \alpha_4^{-1}) \right\} \leq v(\alpha_1^l + \alpha_2^l - \alpha_3^l - \alpha_4^l)
\end{equation*}
for all odd integers $l$. If additionally $\alpha_1 + \alpha_2 \in \mathcal{O}^\times$, then this holds for all $l$. 
\end{lemma}

\begin{proof}
Assume that $l$ is odd. The left-hand side is invariant under $(\alpha_1 , \alpha_2 ,  \alpha_3 , \alpha_4 )\mapsto(\alpha_1^{-1}, \alpha_2^{-1} , \alpha_3^{-1} ,  \alpha_4^{-1}) $ so it is enough to prove the inequality for $l \geq -1$. We now use induction on $l$, with the base case $l= \pm 1$ being obvious. For the induction step, we use the identity
\begin{equation*}
  \alpha_1^l + \alpha_2^l =(\alpha_1^{l-2}+\alpha_2^{l-2})((\alpha_1 + \alpha_2)^{2}-2 \alpha_1 \alpha_2)-(\alpha_1^{l -4}+ \alpha_2^{l-4})\alpha_1^{2}\alpha_2^{2}.
\end{equation*}
If $v$ denotes the minimum on the left-hand side of the desired inequality, the induction hypothesis implies
\begin{equation*}
\alpha_1^l + \alpha_2^l - \alpha_3^l - \alpha_4^l \equiv -2(\alpha_1^{l-2}+\alpha_2^{l-2})(\alpha_1 \alpha_2 - \alpha_3 \alpha_4) -(\alpha_1^{l-4}+\alpha_2^{l-4})(\alpha_1^{2} \alpha_2^2 - \alpha_3^2 \alpha_4^2) \pmod{\mathfrak{p}^v}.
\end{equation*}
Since $v(\alpha_1^{l'}+\alpha_2^{l'}) \geq v(\alpha_1 + \alpha_2)$ for any odd integer $l'$, it suffices to show that $v(\alpha_1 + \alpha_2)+ v(\alpha_1 \alpha_2 - \alpha_3 \alpha_4) \geq v$. This follows from the congruence $\alpha_1^{-1} + \alpha_2^{-1} \equiv \alpha_3^{-1} + \alpha_4^{-1} \pmod{\mathfrak{p}^v}$ by putting both sides on a common denominator and using $\alpha_1 + \alpha_2 \equiv \alpha_3 + \alpha_4 \pmod{\mathfrak{p}^v}$.

For the last part, assume $v(\alpha_1 + \alpha_2)=0$. By what we have already done, it is enough to consider $l=0$ or $l=2^{k}$ for some $k \geq 0$. It is clearly true for $l=0$, and for $l=2^k$, we induct on $k$ with the base case $k=0$ being obvious. For the induction step, we write
\begin{equation*}
\alpha_1^{l}+\alpha_2^l =(\alpha_1^{l/2} + \alpha_2^{l/2})^2 - 2 (\alpha_1 \alpha_2)^{l/2}
\end{equation*}
so by induction hypothesis it is enough to show that $v \leq v((\alpha_1 \alpha_2)^{l/2}-(\alpha_3 \alpha_4)^{l/2})$. The right-hand side is at least $v(\alpha_1 \alpha_2 - \alpha_3  \alpha_4)$. As in the case when $l$ is odd, we have $v(\alpha_1 + \alpha_2)+v(\alpha_1  \alpha_2 - \alpha_3 \alpha_4) \geq v$, but we assume $v(\alpha_1 + \alpha_2)=0$ so this suffices.
\end{proof}

\section{The proof of Proposition \ref{final_estimates}}\label{the_proof_of_proposition_11.2}

\subsection{Case (1)}
In this case $0<\gamma<n_2$, and $\kappa = \left(
\begin{smallmatrix}
1&0\\
\varpi^\gamma&1 \\
\end{smallmatrix}
\right)$ or $\kappa = \left(
\begin{smallmatrix}
0&1\\
1&\varpi^\gamma \\
\end{smallmatrix}
\right)$. As we explained in Section \ref{first_estimates_section}, we can assume that $n$ is even if $\kappa = \left(
\begin{smallmatrix}
0&1\\
1&\varpi^\gamma \\
\end{smallmatrix}
\right)$. For a fixed tuple $\mathbf{m} =(m_1 , m_2 , m_3 , m_4)$ with $m_1 + m_2 = m_3 + m_4$, we consider the integral
\begin{equation*}
I_p(\mathbf{m})= \int_{t \in \mathcal{O}_F^\times}W(m_1 t) W(m_2 t) \overline{W(m_3 t) W(m_4 t)}\, d^\times t,
\end{equation*}
and show that it vanishes when $v(m_1 m_2 - m_3 m_4)<v_0$ where $v_0$ is some integer only depending on $n$ and $\gamma$. Here we define $W(t)$ as $W_b(a(t) \kappa)$ if $\kappa = \left(
\begin{smallmatrix}
1&0\\
\varpi^\gamma&1 \\
\end{smallmatrix}
\right)$. When $\kappa = \left(
\begin{smallmatrix}
0&1\\
1&\varpi^\gamma \\
\end{smallmatrix}
\right)$, it will be convenient to work with a slightly different expression for $W(t)$, and we set $W(t) := \overline{W_b(a(-t) \kappa)}$. It is clear that $I_p(\mathbf{m})$ vanishes when $W(t) = W_b(a(t) \kappa)$ if and only if $I_p(\mathbf{m})$ vanishes when $W(t) = \overline{W_b(a(-t)\kappa)}$. We work with this definition of $W(t)$ because it allows us to simultaneously treat the two cases $\kappa = \left(
\begin{smallmatrix}
1&0\\
\varpi^\gamma&1 \\
\end{smallmatrix}
\right)$ and $\kappa = \left(
\begin{smallmatrix}
0&1\\
1&\varpi^\gamma \\
\end{smallmatrix}
\right)$.

The first step is to write down an expression for $W(t)$. To make the integral $I_p(\mathbf{m})$ easier to handle, we would like an formula only involving the additive character $\psi$, and thus we use the logarithm to rewrite the multiplicative character $\xi$ in terms of $\psi$. For this approach to work, we must partition the range of integration $\mathcal{O}^\times$ into residue classes modulo $\mathfrak{p}_F$. Fix $t_0 \in \mathcal{O}^\times /(1+ \mathfrak{p}_F) $. Recall from Section \ref{quadratic_spaces} that we define $d$ as $0$ if $n$ is even and $1$ if $n$ is odd. Moreover, $b \in \mathcal{O}_E^\times$ is the constant associated to the $\xi$ via Lemma \ref{logarithm_lemma}.

\begin{lemma}\label{lem:expression-case1}
Let $B:= \Omega^d \varpi^\gamma b$. For $t \in t_0 + \mathfrak{p}_F$, we have
\begin{equation*}
W(t)= \mathbf{1}_{\lvert t \rvert=1} C \psi \left( \left(-B \log_E \left( \frac{R(t) + B}{ R(t) -B} \right)+t+R(t) \right) \varpi^{-\gamma -n_2}\right)
\end{equation*}
where $C$ is a constant of size $1$ (depending on $\kappa$ and $t_0$), and $R(t)$ is the unique solution in $\mathcal{O}_F^\times$ to the equation $R^2 +t R- B^2=0$.
\end{lemma}

\begin{proof}
We start with the case $\kappa = \left(
\begin{smallmatrix}
1&0\\
\varpi^\gamma&1 \\
\end{smallmatrix}
\right)$ where $0<\gamma < n_2$. By Equation \eqref{balanced_to_unbalanced1}, $W(t)=W_p(a(t)
(\begin{smallmatrix}
1&0\\
\varpi^{n_2 + \gamma }&1 \\
\end{smallmatrix}))$ so using Equation \eqref{case4}, we have
\begin{equation*}
W(t) = \mathbf{1}_{\lvert t \rvert=1} T_{\pi , \kappa}(t) \xi^{-1}(R(t) +B) \psi((t + R(t))\varpi^{-\gamma - n_2})
\end{equation*}
where $T_{\pi ,\kappa}(t) \in S^1$ depends at most on the square clase of $t$ in $\mathcal{O}_F^\times$, and $R(t)$ is as in the statement of the lemma. Since $t$ lies in fixed class modulo $\mathfrak{p}_F$, the square class of $t$ is fixed, and $T_{\pi , \kappa}(t) $ is in fact constant. From the equation $R^2 +t R- B^2=0$, and the fact that $R(t) \in \mathcal{O}_F^\times$, we see that $R(t) \equiv -t_0$ for all $t \in t_0 + \mathfrak{p}_F$. Since $F = \mathbb{Q}_p$, and $p \geq 3$, it follows that $\kappa_E =1$ (even if $E/F$ is ramified) so $\log_E(1+z)$ converges for $z \in \mathfrak{p}_E$, and we have
\begin{equation*}
\begin{split}
  \xi^{-1}(R(t) +B) &= \xi^{-1}(- t_0) \xi^{-1}((-t_0)^{-1}(R(t)+B))\\
                   &= \xi(- t_0) \psi_E \left(- \frac{b}{\Omega^{a(\xi)- n(\psi_E)}} \log_E((-t_0)^{-1}(R(t)+B)) \right).
\end{split}
\end{equation*}
We have $\Omega^{a(\xi)-n(\psi_E)}= \Omega^{-d}\varpi^{n_2}$ so $-b/ \Omega^{a(\xi)- n(\psi_E)}= -B \varpi^{-\gamma -n_2}$. Recall that $\psi_E := \psi \circ \mathrm{Tr}_{E/F}$, and if $\sigma$ denotes the non-trivial automorphism of $E/F$, we have $\mathrm{Tr}_{E/F}(x)=x + \sigma(x)$. Since $\sigma(B)=-B$, and $\log_E$ is a homomorphism, we find that
\begin{equation*}
\mathrm{Tr}_{E/F}\left( -B \log_E((-t_0)^{-1}(R(t) +B)) \right)= -B \log_E \left( \frac{R(t)+B}{ R(t) -B} \right), 
\end{equation*}
and this gives us the desired expression for $W(t)$ when $\kappa = \left(
\begin{smallmatrix}
1&0\\
\varpi^\gamma&1 \\
\end{smallmatrix}
\right)$.

We now consider the case $\kappa = \left(
\begin{smallmatrix}
0&1\\
1&\varpi^\gamma \\
\end{smallmatrix}
\right)$ where $W(t):= \overline{W_b(a(-t)\kappa )}$. Moreover, we can assume that $n$ is even so in particular, $d=0$, and $B = b \varpi^\gamma$. By Equation \eqref{balanced_to_unbalanced3} and Equation \eqref{case2}, we have
\begin{equation*}
\begin{split}
  \overline{W(t)}&= \psi(-t \varpi^{-\gamma - n_2})W_p(a(\varpi^{-2 \gamma}t) \left(
\begin{smallmatrix}
1&0\\
\varpi^{n_2 - \gamma}&1 \\
\end{smallmatrix}
        \right))\\
  &= \mathbf{1}_{\lvert t \rvert=1} C_{\pi , \kappa} \xi^{-1}(R_0(t)+b) \psi(R_0(t) \varpi^{-n_2})
\end{split}
\end{equation*}
where $R_0(t)$ is the unique solution in $\mathcal{O}_F$ to the equation $\varpi^{\gamma}R^2 +tR-b^2 \varpi^{\gamma}=0$. Since $\gamma >0$, we must have $R_0(t) \in \mathfrak{p}_F$ for all $t \in \mathcal{O}_F^\times$. Hence $R_0(t) +b \equiv b \pmod{\mathfrak{p}_E}$ for all $t$, and
\begin{equation*}
\begin{split}
  \xi^{-1}(R_0(t)+b )&=\xi^{-1}(b) \xi^{-1}(1+ b^{-1} R_0(t))\\
                    &= \xi^{-1}(b) \psi_E \left( - B \varpi^{-\gamma -n_2} \log_E(1+b^{-1} R_0(t)) \right).
\end{split}
\end{equation*}
Using $\psi_E := \psi \circ \mathrm{Tr}_{E/F}$, and $\sigma(b)=-b$, we find that
\begin{equation*}
\psi_E(-B \varpi^{-\gamma-n_2} \log_E(1+ b^{-1} R_0(t))) = \psi \left( -B \varpi^{-\gamma - n_2} \log_E \left(\frac{1+b^{-1} R_0 (t)}{1-b^{-1} R_0(t)}\right) \right).
\end{equation*}
If $R(t)$ denotes the unique solution in $\mathcal{O}_F^\times$ to $R^2 +t R-B^2=0$ as in the previous case, we have $R_0(t) = \varpi^{-\gamma}(-t- R(t))$ so
\begin{equation*}
\frac{1+b^{-1} R_0(t)}{1-b^{-1} R_0(t)}=\frac{1+ b^{-1}\varpi^{-\gamma}(-t-R(t))}{1-b^{-1} \varpi^{- \gamma}(-t-R(t))}=\frac{B-t-R(t)}{B+t+ R(t)}
\end{equation*}
recalling that $B := b \varpi^{\gamma}$. Since
\begin{equation*}
\frac{B-t-R(t)}{B+t+R(t)} \cdot \frac{R(t) +B}{R(t) -B} = \frac{-(R(t)^2 +t R(t) - B^2)-t B}{(R(t)^2 + t R(t) - B^2)- tB}=1 
\end{equation*}
we have
\begin{equation*}
\log_E \left( \frac{1+b^{-1} R_0(t)}{1-b^{-1} R_0(t)} \right) = - \log_E \left( \frac{R(t) + B }{R(t)-B} \right).
\end{equation*}
This gives the desired expression for $W(t)$. 
\end{proof}

\noindent
Let $(\eta_1 , \eta_2 , \eta_3 , \eta_4)=(1,1,-1,-1)$. Then we deduce that $I_p(\mathbf{m})$ is a finite linear combination of integrals of the form
\begin{equation*}
\int_{t \in t_0 + \mathfrak{p}_F} \psi(\varpi^{-\gamma - n_2} \Phi(t)) \, d^\times t
\end{equation*}
where $t_0 \in \mathcal{O}^\times$, 
\begin{equation*}
\Phi(t):= \sum_{j=1}^4 \eta_j H(m_j t) \quad \textrm{and} \quad H(t):= -B \log_E \left( \frac{R(t)+B}{R(t)-B} \right)+ R(t). 
\end{equation*}
Here the $t$-part in the expression for $W(t)$ in Lemma \ref{lem:expression-case1} cancels when taking the weighted sum in $\Phi(t)$ as $\sum_j \eta_j m_j=0$. We also that $\mathbf{m}$ is fixed so the value of $m_j t$ modulo $\mathfrak{p}_F$ is also fixed for $j=1,2,3,4$. The power series $\Phi(t)$ depends on $\kappa$, $\mathbf{m}$ and $t_0$, but nevertheless, we will find a condition for the above integral to vanish only in terms of $v(m_1 m_2 - m_3 m_4)$. Note that since $B \in \mathfrak{p}_E$, one can, by solving the equation for $R(t)$ and Taylor-expanding the square root that appears, give an expression for $R(t)$ as power series that converges for all $t \in \mathcal{O}_F^\times$. Hence it makes sense to differentiate $H(t)$, and the result is the following lemma:

\begin{lemma}
We have
\begin{equation*}
H '(t)= -\frac{1}{2t} \sqrt{t^2 + 4 B^2}-\frac{1}{2}
\end{equation*}
where the square root is chosen to be congruent to $t$ modulo $\mathfrak{p}_E$. 
\end{lemma}

\begin{proof}
By basic differentiation, we have
\begin{equation*}
\begin{split}
  H '(t) &= -B \cdot \frac{R(t) -B}{R(t) +B} \cdot \frac{R '(t)(R(t) -B)-(R(t) +B) R'(t)}{(R(t)-B)^2}+R'(t)\\
         & = R '(t) \left[ \frac{2 B^2}{R(t)^2 -B^2}+1 \right].
\end{split}
\end{equation*}
Since $R(t)^2 +t R(t) - B^2=0$, and $R(t) \in \mathcal{O}_F^\times$, we find that $R(t) =-\frac{1}{2}(t+\sqrt{t^2 + 4 B^2})$ where the square root is chosen to be $t$ modulo $\mathfrak{p}_E$. Hence $R(t)^2 - B^2 = -t R(t) =\frac{1}{2} t(t + \sqrt{t^2 + 4 B^2})$, and by implicit differentiation,
\begin{equation*}
R'(t) = -\frac{R(t)}{2 R(t) +t}=-\frac{1}{2}\frac{t + \sqrt{t^2 + 4 B^2}}{\sqrt{t^2 + 4 B^2}}
\end{equation*}
so
\begin{equation*}
H'(t) = -\frac{1}{2} \frac{t + \sqrt{t^2 + 4 B^2}}{\sqrt{t^2 + 4 B^2}} \left[ \frac{4 B^2}{t(t + \sqrt{t^2 + 4 B^2})}+1 \right],
\end{equation*}
and this simplifies to the desired expression.
\end{proof}

\noindent
It is now possible to analyse the phase function $\Phi (t) $ and make Lemma \ref{non_linear_substitution} applicable. Recall that we defined $v_{\mathbf{m}}$ as $v(m_1 m_2 - m_3 m_4)$. 

\begin{lemma}\label{power_series_coefficients1}
Suppose $m_1 m_2 \neq m_3 m_4$. Then $v(\frac{1}{k!}\Phi^{(k)}(t))=d+2 \gamma +v_{\mathbf{m}}$ for all $k \geq 1$ and $t \in t_0 + \mathfrak{p}_F$.
\end{lemma}

\begin{proof}
By the previous lemma, $\Phi'(t)=-\frac{1}{2} \sum_{j=1}^4 \eta_j \frac{1}{t}\sqrt{(m_j t)^2 + 4 B^2}$. The square root has a convergent Taylor-expansion
\begin{equation*}
\sqrt{(m_j t)^2 + 4 B^2} = m_j t \left[ 1+ \sum_{i \geq 1} \binom{1/2}{i}(4 B^2)^i m_j^{-2i}t^{-2i} \right].
\end{equation*}
Convergence follows from $B \in \mathfrak{p}_E$, and the fact that $\binom{1/2}{i} \in \mathcal{O}_F$ for all $i$ becuase $p$ is odd. Hence
\begin{equation*}
\Phi'(t)= \sum_{i \geq 1} a_i t^{-2i} \quad \textrm{with} \quad a_i := -\frac{1}{2}\binom{1/2}{i}(4 B^2)^i \sum_{j=1}^4 \eta_j m_j^{-2i+1}.
\end{equation*}
Here we have used that $\sum_j \eta_j m_j =0$ to get rid of the constant term. Since $B = \Omega^d \varpi^{\gamma} b$, and $\Omega^2 = \varpi$, we see that $v(4 B^2)= d+2 \gamma$. Hence, by inspection, we have $v(a_1)= d+ 2 \gamma +v(m_1 + m_2)+ v_{\mathbf{m}}$. Since $W_b(a(t) \kappa)$ is only supported on $\lvert t \rvert=1$, we can assume that $v(m_1 + m_2)=0$, as explained in Section \ref{first_estimates_section}. Hence we have $v(a_1)=d+ 2 \gamma + v_{\mathbf{m}}$. Differentiating $\Phi '(t)$ $k$ times gives
\begin{equation*}
\begin{split}
  \frac{1}{(k+1)!} \Phi^{(k+1)}(t) & = \sum_{i \geq 1} a_i \frac{(-2i)(-2i-1)\cdots(-2i-k+1)}{(k+1)!}t^{-2i-k}\\
                                  & =(-1)^k a_1 t^{-2-k} + \sum_{i \geq 2} \frac{(-1)^k}{2i-1}\binom{2i+k-1}{k+1} a_i t^{-2i -k}
\end{split} 
\end{equation*}
so the proof is complete if we can show that
\begin{equation*}
v(a_1)< v \left( \frac{1}{2i-1} \binom{2i+k-1}{k+1} a_i \right)
\end{equation*}
for all $i \geq 2$. The binomial coefficient is an integer so the right-hand side is at least $v(a_i)-v(2 i-1)$. Since $p$ is odd, $\binom{1/2}{i}$ has non-negative valuation so we have $v(a_i) \geq (d+2 \gamma)i + v(\sum_{j=1}^4 \eta_j m_j^{-2i+1} ) $. We have
\begin{equation*}
(d + 2 \gamma )i- v(2i-1) = (2 \gamma(i-1) - v(2 i -1))+ 2 \gamma + d i > 2 \gamma +d 
\end{equation*}
since $v(2i-1)\leq \log_p(2i-1)< \log(2i-1) \leq 2(i-1)$ (using $p \geq 3$). Hence it is enough to show that
\begin{equation*}
v\left(\sum_{j=1}^4 \eta_j m_j^{-2i+1}\right) \geq v_{\mathbf{m}},
\end{equation*}
but since $v(m_1 + m_2)=0$, $v_{\mathbf{m}}$ equals $v(m_1^{-1} + m_2^{-1} - m_3^{-1} - m_4^{-1})$ so this follows from Lemma \ref{elementary}. 
\end{proof}

\noindent
We can now prove Case (1) of Proposition \ref{final_estimates}. We have
\begin{equation*}
\int_{t \in t_0 + \mathfrak{p}_F} \psi(\varpi^{-\gamma - n_2} \Phi(t))\, d^\times t = \psi(\varpi^{- \gamma - n_2} \Phi(0)) \int_{z \in \mathfrak{p}_F} \psi\left(\varpi^{- \gamma - n_2} \sum_{k \geq 1} \frac{1}{k}b_k z^ k\right)\, d z
\end{equation*}
where $b_k := \frac{1}{k!} \Phi^{(k)}(t_0)$. By the above lemma, $d+2 \gamma + v_{\mathbf{m}}= v(b_1) = v(b_k)$ for all $k \geq 1$ so by Lemma \ref{non_linear_substitution}, the $z$-integral vanishes when $- \gamma - n_2 +d + 2 \gamma +v_{\mathbf{m}}<-1$, i.e. for $v_{\mathbf{m}}<n_2 - \gamma -1-d$. By part (1) of Lemma \ref{first_estimates}, with $v_0 = n_2 - \gamma -1-d$, we deduce that $\mathcal{N}_a(\kappa W_b) \prec p^{\gamma}$ as desired.

\subsection{Case (2)}
In this case, $\kappa = \left(
\begin{smallmatrix}
1&0\\
1&1 \\
\end{smallmatrix}
\right)$, and the goal is to estimate $\mathcal{N}_a(\kappa W_b^s)$ for $0<s<\frac{n}{6}$. We recall that the only possibility for $\pi_p$ is $\chi \boxplus \chi^{-1}$ for some unitary character $\chi : F^\times \rightarrow S^1$. In other words, the quadratic space $E$ is split and equals $F \times F$. Moreover, the conductor exponent $n$ must be even. Using Equation \eqref{balanced_to_unbalanced1} and Equation \eqref{case3.2} in Theorem \ref{whittaker_formulas}, we find that
\begin{equation*}
\kappa W_b^s(t) = \mathbf{1}_{\lvert t \rvert=q^{-s}}q^{-\frac{s}{2}} \sum_{\pm}C_{\pm}\xi^{-1}(x_{\pm} \pm b ) \psi((x_{\pm}+t) \varpi^{-n/2})
\end{equation*}
where $C_{\pm} \in S^1$ are constants, and $x_{\pm}$ denote the solution the equation $x^2 + t x - b^2=0$ satisfying $x_{\pm} \equiv \mp b_0 \pmod{\mathfrak{p}_F}$ where $b=(b_0 ,-b_0)$ for some $b_0 \in \mathcal{O}_F^\times$ c.f. Lemma \ref{logarithm_lemma}. Let $\kappa W_b^{s, \pm}$ denote the two terms in the above sum. Then $\mathcal{N}_a(\kappa W_b^s) \ll \mathcal{N}_a(\kappa W_b^{s,+})+\mathcal{N}_a(\kappa W_b^{s,-})$, and it is not hard to see that Case (ii) of Lemma \ref{first_estimates} holds with $\kappa W_b^{s,\pm}$ in place of $\kappa W_b^s$. Hence we have to decide for which tuples $\mathbf{m} =(m_1 , m_2 , m_3 , m_4)$ with $m_1 + m_2 = m_3 + m_4$, does the integral
\begin{equation*}
I_p(\mathbf{m}) = \int_{t \in \mathcal{O}_F^\times} W_{\pm}(m_1 t) W_{\pm}(m_2 t) \overline{W_{\pm}(m_3 t) W_{\pm}(m_4 t)}\,d^\times t
\end{equation*}
vanish where $W_{\pm}(t):= \kappa W_b^{s,\pm}(t)$. Since we integrate over $t \in \mathcal{O}^\times$, the support property of $\kappa W_b^{s}$ forces $v(m_j)=s$ for all $j=1,2,3,4$. As explained in Section \ref{first_estimates_section}, we can further assume that $v(m_1 + m_2)=s$. As in the previous case, we start by writing down a formula for $W_{\pm}$ only in terms of the additive character $\psi$. $W_{\pm}(t)$ is only supported on $\mathfrak{p}_F^s \setminus \mathfrak{p}_F^{s+1}$, and we partition this range into cosets of $\mathfrak{p}_F^{s+1}$. Fix $t_0 \in \mathfrak{p}_F^{s}\setminus \mathfrak{p}_F^{s+1}$.

\begin{lemma}
For $t \in t_0 + \mathfrak{p}_F^{s+1}$, we have
\begin{equation*}
W_{\pm}(t) = \mathbf{1}_{\lvert t \rvert= q^{-s}} A_{\pm} \psi \left( \left(-b \log_E \left(d_{\pm} \frac{R_{\pm}(t) \pm b}{R_{\pm}(t) \mp b} \right)+ R_{\pm }(t)\right)\varpi^{-n/2} \right)
\end{equation*}
where $A_{\pm} \in \mathbb{C}$ is a non-zero constant that depends on $t_0$, $d_{\pm} \in E^\times$ is a constant, also depending on $t_0$, such that the input to the logarithm lies in $1+\mathfrak{p}_E$, and $R_{\pm }(t)$ is the solution to $R^2 +t R -b^2 =0$ satisfying $R_{\pm}(t)\equiv \mp b_0 \pmod{\mathfrak{p}_E}$. 
\end{lemma}

\begin{proof}
We have
\begin{equation*}
W_{\pm}(t) = \mathbf{1}_{\lvert t \rvert = q^{-s}} q^{-\frac{s}{2}}C_{\pm} \xi^{-1}(R_{\pm}(t) \pm b) \psi((R_{\pm}(t)+t)\varpi^{-n/2}) 
\end{equation*}
where $R_{\pm}(t)$ is the unique solution to $R^2 + t R - b^2 =0$ satisfying $R_{\pm}(t) \equiv \mp b_0 \pmod{\mathfrak{p}_F}$. In fact one sees that $t \in t_0 + \mathfrak{p}^{s+1}$ implies that $\varpi^{-s}(R_{\pm}\pm b_0)$ is an element of $\mathcal{O}_F^\times$ with a fixed value modulo $\mathfrak{p}_F$. It follows that, in $E = F \times F$, the elements $(\varpi^{-s},1)(R_+(t)+b)$ and $(1, \varpi^{-s})(R_{-}(t) -b)$ lie in $\mathcal{O}_E^\times$ and have a fixed value modulo $\mathfrak{p}_E$. Hence we can find $c_{\pm} \in E^\times$ such that $c_{\pm}(R_{\pm}(t) \pm b) \equiv 1 \pmod{\mathfrak{p}_E}$ for all $t \in t_0 + \mathfrak{p}^{s+1}$. Arguing as in the previous section, we find that
\begin{equation*}
\xi^{-1}(R_{\pm}(t) \pm b)= \xi(c_{\pm}) \psi \left(-b \varpi^{-n/2} \log_E  \left(\frac{c_{\pm}}{\sigma(c_{\pm})} \frac{R_{\pm}(t)\pm b}{R_{\pm}(t)\mp b} \right)\right)
\end{equation*}
where $\sigma$ is the non-trivial element of $\mathrm{Aut}(E/F)$. By defining the constant $A_{\pm}$ as $q^{-\frac{s}{2}} C_{\pm } \xi(c_{\pm}) \psi(b \varpi^{-n/2})$ and the constant $d_{\pm}$ as $ c_{ \pm}/ \sigma(c_{\pm})$, we arrive at the desired expression. 
\end{proof}

\noindent
If we now fix a tuple $\mathbf{m} =(m_1 , m_2 , m_3 , m_4)$ with $m_1 + m_2 = m_3 + m_4$, and $v(m_j)=s$ for $j=1,2,3,4$, it follows that we can write $I_p(\mathbf{m})$ as a finite linear combination of integrals of the form
\begin{equation*}
\int_{z \in \mathfrak{p}_F} \psi(\varpi^{-n/2} \Phi(z))\, d z 
\end{equation*}
where 
\begin{equation*}
\Phi(z):= \sum_{j=1}^4 \eta_j H(m_j(t_0 +z)) \quad \textrm{and} \quad H(t):= -b \log_E \left( d_{\pm} \frac{R_{\pm}(t) \pm b}{R_{\pm}(t)\mp b} \right)+R_{\pm}(t)  
\end{equation*}
for $\lvert t \rvert=q^{-s}$. Here $t_0$ is a fixed coset representative for $1 + \mathfrak{p}_F$ in $\mathcal{O}_F^\times$, and we have written $t \in t_0 + \mathfrak{p}_F$ as $t_0 + z$ for $z \in \mathfrak{p}_F$. This will be convenient in the calculations to follow. The constant $d_{\pm}$ depends on $\mathbf{m}$ and $t_0$. By a calculation identical to one in the previous section, we find that
\begin{equation*}
\Phi'(z) = -\frac{1}{2} \sum_{j=1}^4 \eta_j\frac{1}{t_0 + z}\sqrt{(m_j(t_0 +z))^2 + 4 b^2}
\end{equation*}
where the square root is chosen to be congruent to $\mp 2b_0 $ modulo $\mathfrak{p}_F$. We now prove:

\begin{lemma}
We have $\Phi'(z) = \pm \sum_{ k \geq 0} a_k z^k $ where $v(a_0) = v_{\mathbf{m}}$, and $v(a_k ) \geq v(a_0)$ for all $k \geq 1$. 
\end{lemma}

\begin{proof}
We start from the expression for $\Phi'(z)$ and Taylor-expand the square root:
\begin{equation*}
\Phi'(z) = \pm \frac{1}{2}\sum_{j=1}^4 \eta_j \frac{1}{2(t+z)} \sum_{i \geq 0} \binom{1/2}{i}(m_j(t + z))^{2i}(2b_0)^{1-2i} = \pm \sum_{k \geq 0} a_k z^k
\end{equation*}
where we have expanded $(t+z)^{2i-1}$ to obtain
\begin{equation*}
a_k = \frac{1}{2} \sum_{i \geq 1}(2b_0)^{1-2i} t^{2i-1-k} \binom{1/2}{i} \binom{2i-1}{k} \sum_{j=1}^4 \eta_j m_j^{2i}.
\end{equation*}
Note that this expression is convergent since $v(m_j)=s>0$. We now claim that $v(\sum_{j=1}^4 \eta_j m_j^2)  < v(\sum_{j=1}^4 \eta_j m_j^{2i})$ for all $i>1$. Write $m_j = p^s m_j '$ where $p \nmid m_j'$. Since $m_1 + m_2 = m_3 + m_4 $, we have
\begin{equation*}
m_1^2 + m_2^2 - m_3^2 - m_4^2 = -2p^{2s}(m_1' m_2 ' - m_3 ' m_4 ' ),
\end{equation*}
and since $v(m_1 ' + m_2 ')=0$, we have, by Lemma \ref{elementary},
\begin{equation*}
\begin{split}
  v(m_1 ' m_2 ' - m_3 ' m_4 ') &= v((m_1')^{-1} +(m_2 ')^{-1}-(m_3' )^{-1} -(m_4 ')^{-1})\\
  &\leq v((m_1 ')^{2i}+(m_2 ')^{2i}-(m_3 ')^{2i}-(m_4')^{2i})
\end{split}
\end{equation*}
for all $i$ so when $i>1$
\begin{equation*}
  v\left(\sum_{j=1}^4 \eta_j m_j^2\right)  \leq 2s + v \left(  \sum_{j=1}^4 \eta_j(m_j ')^{2i}\right)< 2si+v \left( \sum_{j=1}^4 \eta_j(m_j '
    )^{2i} \right)= v \left( \sum_{j=1}^4 \eta_jm_j^{2i} \right) 
\end{equation*}
as desired. The first term in the expansion of $a_0$ is $\frac{1}{8}b^{-1} t \sum_{j=1}^4 \eta_j m_j^{2}$, and, by the above, all subsequent terms have strictly larger valuation so $v(a_0)= v(\sum_j^4 \eta_j m_j^2 )=v_{\mathbf{m}}$. We also see that $v(a_k) \geq v_{\mathbf{m}}$ for all $k \geq 1$ so the proof is complete.    
\end{proof} 

\noindent
To complete the proof of Case (2) in Proposition \ref{final_estimates}, we use the above lemma to conclude that
\begin{equation*}
\Phi(z) = \Phi(0) \pm \sum_{k \geq 1} \frac{1}{k}a_{k-1} z^k
\end{equation*} 
with $v(a_0)= v_{\mathbf{m}}$ and $v(a_k)\geq v(a_0)$ for all $k \geq 0$. Hence, by Lemma \ref{non_linear_substitution}, the integral
\begin{equation*}
\int_{z \in \mathfrak{p}_F} \psi(\varpi^{-n/2} \Phi(z)) d z
\end{equation*}
vanishes when $-n/2 + v_{\mathbf{m}}<-1$, i.e. $v_{\mathbf{m}} < n/2 -1$. By part (2) of  Lemma \ref{first_estimates} with $v_0 =n/2 -1$, we have $\mathcal{N}_a(\kappa W_b^{s, \pm}) \prec 1+p^{n/2 +s-(n/2-1)}(p^{-\frac{s}{2}})^4 \ll 1$ where we have used that $\lVert \kappa W_b^{s, \pm}\rVert _\infty \ll p^{-\frac{s}{2}}$. This is what we wanted.

\subsection{Case (3)}
In this case $\kappa = \left(
\begin{smallmatrix}
1&0\\
1&1 \\
\end{smallmatrix}
\right)$, and for $W(t) := \kappa W_b^{0,u}(t)$ with $0 \leq u <  \frac{n}{4}$, and $\mathbf{m} =(m_1 , m_2 , m_3 , m_4)$ with $m_1 + m_2 = m_3 + m_4$, we consider the integral
\begin{equation*}
I_p(\mathbf{m}) = \int_{t \in \mathcal{O_F}^\times} W(m_1 t) W(m_2 t) \overline{W(m_3 t) W(m_4 t)}\, d^\times t.
\end{equation*}
By Theorm \ref{whittaker_formulas}, $E/F$ must be split or an unramified quadratic extension, and in particular $n$ is even. Moreover, we define $\Delta:=1+4 b^2 t^{-2}$ where $b = b_\xi$ is as in Lemma \ref{logarithm_lemma}. As in the previous two cases, we would like to express $W(t)$ only in terms of the additive character $\psi$. As before, we must fix the residue class of $t$, and, this time, we fix $t_0 \in \mathcal{O}_F^\times /(1+ \mathfrak{p}_F^{u+1})$. 

\begin{lemma}
Let $t_0 \in \mathcal{O}_F^\times /(1+ \mathfrak{p}_F^{u+1})$. Then $W(t)$ vanishes identically for $t \in t_0 + \mathfrak{p}_F^{u+1}$ unless $t_0^2 + 4 b^2$ is a square of valuation $u$ in which case
\begin{equation*}
W(t) = q^{\frac{u}{4}}\sum_{\pm}C_{\pm} \psi \left( \left( -b \log_E \left( d_{\pm}\frac{R_{\pm }(t)+b}{R_{\pm}(t)-b} \right)+R_{\pm}(t) +t \right)\varpi^{-n/2}  \right)
\end{equation*}
where $C_{\pm}$ and $d_{\pm}$ are constants that depend on $t_0$, and $R_{+}(t)$ and $R_{-}(t)$ denote the two solutions to the equation $R^2 +t R-b^2 =0$. 
\end{lemma}

\begin{proof}
By Equation \eqref{case3.3} and Equation \eqref{balanced_to_unbalanced1}, we have
\begin{equation*}
W(t) = \mathbf{1}_{\lvert t \rvert=1, \Delta \in \mathcal{O}_F^2, v(\Delta)=u } q^{\frac{u}{4}} \sum_{\pm}  T_{\pm}(t) \xi^{-1}(R_{\pm}(t)+b) \psi((R_{\pm}(t)+t)\varpi^{-n/2}) 
\end{equation*}
where $T_{\pm}(t) \in S^1$ depends at most on $t$ modulo $\mathfrak{p}_F$, and $R_{\pm}(t) \in \mathcal{O}_F$ denote the two solutions to the equation $R^2 + t R - b^2=0$. Since the residue class of $t$  modulo $\mathfrak{p}_F$ is fixed, $T_{\pm}(t)$ is constant. For $t \in \mathcal{O}_F^\times$, the condition that $1+4 b^2 t^{-2}$ is a square of valuation $u$ is determined by the residue class of $t$ modulo $\mathfrak{p}_F^{u+1}$ so this proves the first part of the lemma. Suppose now that $t_0^2 + 4 b^2 = \varpi^u \alpha_0^2$ for some $\alpha_0 \in \mathcal{O}_F^\times$. We can write $t \in t_0 + \mathfrak{p}_F^{u+1}$ as $t = t_0 + z \varpi^{u}$ with $z \in \mathfrak{p}_F$, and then $t^2 + 4 b^2 = \varpi^u(\alpha_0^2 + (2 t_0z + \varpi^u z^2) )$. By solving the quadratic equation for $R_{\pm}$, we have
\begin{equation*}
R_{\pm}(t) = \frac{-t \pm \varpi^{\frac{u}{2}}\sqrt{\alpha_0^2 +(2 t_0 z +  \varpi^u z^2)}}{2},
\end{equation*}
and since $\alpha_0^2  +(2 t_0 z + z \varpi^u z^2) \equiv \alpha_0^2 \pmod{\mathfrak{p}_F}$, it follows that the values of $R_{\pm}(t)$ modulo $\mathfrak{p}_F^{u/2+1}$ only depend on $t_0$. Hence there are constants $c_{\pm} \in E^\times$ such that $c_{\pm}(R_{\pm}(t)+b) \equiv 1 \pmod{\mathfrak{p}_E}$. Arguing as in the previous cases, we have
\begin{equation*}
\xi^{-1}(R_{\pm}(t)+b)= \xi(c_{\pm}) \psi \left( -b \varpi^{-n/2} \log_E \left( \frac{c_{\pm}}{\sigma(c_{\pm})} \frac{R_{\pm}(t)+b}{R_{\pm}(t)-b} \right) \right)
\end{equation*}
where $\sigma$ is the non-trivial element of $\mathrm{Aut}(E/F)$. 
\end{proof}

\noindent
We now fix $\mathbf{m} =(m_1 , m_2 , m_3 , m_4)$ with $m_1 + m_2 = m_3 + m_4$ and consider the integral $I_p(\mathbf{m})$ defined above. For the integrand to be supported, we need that $(m_j t)^2 + 4 b^2$ is a square of valuation $u$ for all $j=1,2,3,4$ which implies that $m_1^2 \equiv m_2^2 \equiv m_3^2 \equiv m_4^2 \pmod{p^u}$. Since $W(t)$ is only supported on $\lvert t \rvert=1$, we can assume that $m_1 \equiv m_2 \equiv m_3 \equiv m_4 \equiv a \pmod{p}$ for some fixed $a \in \mathbb{F}_p^\times$. Hence $I_p(\mathbf{m})$ vanishes identically unless $m_1 \equiv m_2 \equiv m_3 \equiv m_4 \pmod{p^u}$.\\

\noindent
Motivated by the above lemma, we partition the range of integration into cosets $t_0 + \mathfrak{p}_F^{u+1}$ of $1+\mathfrak{p}_F^{u+1} $ in $\mathcal{O}_F^\times$. We only need to consider cosets $t_0 + \mathfrak{p}_F^{u+1}$ such that $(m_j t_0)^2 + 4 b^2 = \varpi^u \alpha_j^2$ for all $j=1,2,3,4$ and some $\alpha_1 , \alpha_2 , \alpha_3 , \alpha_4 \in \mathcal{O}_F^\times$ (depending on $t_0$) because otherwise the integrand vanishes. Writing $W(m_jt)$ as in the lemma above, $W(m_j t)$ consists of two terms corresponding the two solutions $R_{\pm}(m_j t)$ to $R^2 + m_j t R - b^2 =0$. Since $\mathbf{m}$ is fixed, it follows from the proof of the above lemma, that if we also fix $t_0$, then the values of $R_{\pm}(m_j t)$ modulo $\mathfrak{p}_F^{u/2+1}$ only depends on $\mathbf{m}$ and $t_0$. If we expand the product $W(m_1 t)W(m_2 t) \overline{W(m_3 t)W(m_4 t)} $ into sixteen terms, it follows that $I_p(\mathbf{m})$ is a finite linear combination of integrals of the form
\begin{equation*}
\int_{z \in \mathfrak{p}} \psi(\varpi^{-n/2} \Phi(z))\, d z
\end{equation*}
where
\begin{equation*}
\Phi(z) = \sum_{j=1}^4 \eta_j \psi \left(\left(-b \log_E \left( d \frac{R_j(z)+b}{R_j(z)-b} \right)+ R_j(z)\right) \varpi^{-n/2} \right).
\end{equation*}
Here $R_j(z)$ denotes the solution to the equation $R^2 +(t_0 + z \varpi^{u}) m_j R- b^2 =0$ having a fixed value modulo $\mathfrak{p}_F^{u/2+1}$, and $t_0$ is some fixed element of $\mathcal{O}_F^\times /(1+ \mathfrak{p}_F^{u+1})$ such that $(m_j t_0)^2 + 4 b^2 = \varpi^u \alpha_j^2$ for some $\alpha_1 , \alpha_2, \alpha_3 , \alpha_4 \in \mathcal{O}_F^\times$. The number $d$ denotes a non-zero constant that depends on $\mathbf{m}$ and $t_0$. A familiar calculation gives
\begin{equation*}
\Phi '(z) = -\frac{1}{2} \varpi^u \sum_{j=1}^4 \eta_j \frac{1}{t_0 +z \varpi^u} \sqrt{((t_0 + z \varpi^u)m_j)^2 + 4b^2}
\end{equation*}
where the square roots are chosen to give $R_j(z)$ the right value modulo $\mathfrak{p}_F^{u/2+1}$.

We now write $\Phi(z)$ as a power series in $z$. In order to give a convergent expansion, we use that $(m_j t_0)^2 + 4 b^2 = \varpi^u \alpha_j^2$ so that $((t_0 + z \varpi^u) m_j)^2 + 4b^2$ equals $\varpi^u(\alpha_j^2 + m_j^2(2t_0z +\varpi^u z^2))$. Hence 
\begin{equation*}
\sqrt{((t_0 + z \varpi^u)m_j) + 4 b^2} = \varpi^{\frac{u}{2}} \sum_{l \geq 0} \binom{1/2}{l} m_j^{2l}(2t_0 z + \varpi^{u} z^2)^l \alpha_j^{1-2l}
\end{equation*}
where the sign of $\alpha_j$ is chosen according to the value of $R_j(z)$ modulo $\mathfrak{p}^{u/2+1}$. This series converges since $v(z)>0$. Therefore, $\Phi '(z)$ has the expansion
\begin{equation*}
\Phi '(z) = -\varpi^{\frac{3u}{2}} \sum_{l \geq 0} \binom{1/2}{l} \frac{(2t_0 z + \varpi^u z^2)^l}{2(t_0 + \varpi^{u}z)} \sum_{j=1}^4 \eta_j m_j^{2l} \alpha_j^{1-2l}.
\end{equation*}
We would now like to rewrite this in the form $\sum_{k \geq 0} a_k z^k$ for some coefficients $a_k$. Let
\begin{equation*}
-\binom{1/2}{l}\frac{(2t_0z+\varpi^u z^2)^l}{2(t_0 +\varpi^{u}z)}=\sum_{g \geq l} A_{l,g}z^g
\end{equation*}
for some $A_{l,g} \in \mathcal{O}$. We would also like to express $m_j^{2l}$ in terms of $\alpha_j$. Using $(m_j t_0)^2 + 4 b^2= \varpi^u \alpha_j^2$, we find that
\begin{equation*}
m_j^{2l} = t_0 ^{-2l}(\varpi^u \alpha_j^2 -4 b^2)^l = \sum_{h = 0}^l B_{l,h} \alpha_j^{2h}
\end{equation*}
where $B_{l,h}=t_0^{-2l} \binom{l}{h} \varpi^{uh}(-4 b^2)^{l-h} \in \mathcal{O} $. Inserting and rearranging gives
\begin{equation*}
\begin{split}
  \Phi'(z) &= \varpi^{\frac{3u}{2}} \sum_{l \geq 0} \left[ \sum_{g \geq l} A_{l,g} z^g \right] \left[ \sum_{j=1}^4 \eta_j \sum_{h=0}^l B_{l, h} \alpha_j^{1-2(l-h)} \right]\\
           &= \varpi^{\frac{3u}{2}} \sum_{k \geq 0} \left[ \sum_{l \leq k} A_{l,k} \sum_{j=1}^4 \eta_j \sum_{h=0}^l B_{l,h} \alpha_j^{1-2(l-h)} \right] z^k \\
  &= \varpi^{\frac{3u}{2}} \sum_{k \geq 0} \left[ \sum_{h=0}^k \left[\sum_{l=h}^k A_{l,k}B_{l,l-h}\right] \sum_{j=1}^4 \eta_j \alpha_j^{1-2h}\right] z^k. 
\end{split}
\end{equation*}
We conclude that
\begin{equation*}
\Phi'(z) = \varpi^{\frac{3u}{2}} \sum_{k \geq 0} a_k z^k \quad \textrm{with} \quad a_{k}=\sum_{h=0}^k b_{k,h} \sum_{j=1}^4 \eta_j \alpha_j^{1-2h}
\end{equation*}
for certain $b_{k,h} \in \mathcal{O}$ which can be explicitly computed. As in the preceeding two cases, our aim is to prove $v(a_0) \leq v(a_k)$ for all $k$ and express $v(a_0)$ in terms of $v_{\mathbf{m}}= v(m_1 m_2 - m_3 m_4)$. Unlike in the other cases, we only get an upper bound for $v(a_0)$ in terms of $v_{\mathbf{m}}$ and $u$, but this turns out to be sufficient. By tracing through the above calculations, we find that $a_ 0 = b_{0,0} \sum_{j=1}^4 \eta_j \alpha_j$ where
\begin{equation*}
b_{0,0}=A_{0,0} B_{0,0} =-\frac{1}{2t_0} \in \mathcal{O}^\times 
\end{equation*}
so $v(a_0)= v(\alpha_1 + \alpha_2 - \alpha_3 - \alpha_4)$. By Lemma \ref{elementary} and the expression for $a_k$, the inequalities $v(a_0) \leq v(a_k)$ would all follow, if we can show that
\begin{equation*}
v(\alpha_1 + \alpha_2 - \alpha_3 - \alpha_4) \leq v(\alpha_1^{-1}+ \alpha_2^{-1}-\alpha_3^{-1}-\alpha_4^{-1}).
\end{equation*}
We now have to work a bit harder than in the other cases because there are no obvious relations between $\alpha_1 , \alpha_2 , \alpha_3 , \alpha_4$. The only relations, we can use, are the equalities $(m_j t_0)^2 + 4 b^2=\varpi^u \alpha_j^2$ for $j=1,2,3,4$, and $m_1 + m_2 = m_3 + m_4$. The following lemma use this to relate valuations of expressions in $\alpha_1 ,\alpha_2 , \alpha_3 , \alpha_4$ to $v_{\mathbf{m}}$.

\begin{lemma}
We have the following inequalities:
\begin{enumerate}
\item
  \begin{equation*}
    \begin{split}
 2 v_{\mathbf{m}}& \geq v(\alpha_1 + \alpha_2 -\alpha_3 - \alpha_4)+v(\alpha_1 + \alpha_2 + \alpha_3 + \alpha_4)  \\  &+ \min \left\{ v_{\mathbf{m}}, u+v(\alpha_1 \alpha_2 - \alpha_3 \alpha_4) \right\}+3u;
\end{split}
\end{equation*}
\item
\begin{equation*}
v(\alpha_1 + \alpha_2 - \alpha_3 - \alpha_4) \leq v(\alpha_1^{-1} + \alpha_2^{-1} - \alpha_3^{-1} - \alpha_4^{-1});
\end{equation*}
\item
  \begin{equation*}
 v(\alpha_1 + \alpha_2 - \alpha_3 - \alpha_4) \leq v_{\mathbf{m}} - 2u;
\end{equation*}
\end{enumerate}  
\end{lemma}

\begin{proof}
(i) For the first part, let $v_1 := v(\alpha_1 + \alpha_2 - \alpha_3 - \alpha_4)$. The idea of the proof is to start from $\alpha_1 + \alpha_2 \equiv \alpha_3 + \alpha_4 \pmod{\mathfrak{p}_F^{v_1}}$, and then square and rearrange repeatedly until both sides are expressed only in terms of $m_1 , m_2 , m_3 , m_4$. Throughout the proof, we use the obvious fact that if $a \equiv b \pmod{\mathfrak{p}_F^v}$ then $a^2 \equiv  b^2 \pmod{\mathfrak{p}_F^{v+v(a + b)}}$. From $\alpha_1 + \alpha_2 \equiv \alpha_3 + \alpha_4 \pmod{\mathfrak{p}_F ^{v_1}}$ we get
\begin{equation*}
\varpi^{u}(\alpha_1 + \alpha_2)^2 \equiv \varpi^u(\alpha_3 + \alpha_4)^2 \pmod{\mathfrak{p}_F^{v_1 + v_2 +u}}
\end{equation*}
where $v_{2}:= v(\alpha_1 + \alpha_2 + \alpha_3 +\alpha_4)$. We expand both sides and use $\varpi^u \alpha_j^2 =(m_j t_0)^2 +  4 b^2$ and $m_1 + m_2 = m_3 + m_4$ to get
\begin{equation*}
t_0^2(m_1 m_2 - m_3 m_4) \equiv \varpi^u(\alpha_1 \alpha_2 - \alpha_3 \alpha_4) \pmod{\mathfrak{p}_F^{v_1 + v_2 +u}}.
\end{equation*}
Squaring both sides gives
\begin{equation*}
t_0^4(m_1 m_2 - m_3 m_4)^2 \equiv \varpi^{2u}(\alpha_1 \alpha_2 - \alpha_3 \alpha_4 )^2 \pmod{\mathfrak{p}_F^{v_1 + v_2 +v_3 + u}}
\end{equation*}
where $v_3 := v(t_0^2(m_1 m_2 - m_3 m_4)+\varpi^u(\alpha_1 \alpha_2 - \alpha_3 \alpha_4))$. We note that $v_3 \geq u$ since $(m_j t_0)^2 + 4 b^2 = \varpi^u \alpha_j^2$ and $m_j \equiv a \pmod{p}$ for all $j=1,2,3,4$ forces $m_1 \equiv m_2 \equiv m_3 \equiv m_4 \pmod{p^u}$. Expanding the right-hand side, and using $\varpi^{u} \alpha_j^2=(m_j t_0)^2 + 4 b^2 $, we find that, modulo $\mathfrak{p}_F^{v_1 + v_2 + v_3 + u}$, the following congruence holds
\begin{equation}\label{eq:congruence}
\begin{split}
  \varpi^{2 u} \alpha_1 \alpha_2 \alpha_3 \alpha_4 \equiv & -\frac{1}{2} [ t_0^4(m_1 m_2-m_3 m_4 )^2 -((m_1 t_0)^2 +4 b^2)((m_2 t_0)^2 +4 b^2)\\
  & -((m_3 t_0)^2 + 4 b^2)(( m_4 t_0)^2 + 4 b^2) ]  
\end{split}
\end{equation}
The right-hand side equals
\begin{equation*}
 s_4(m_1 , m_2 , m_3 , m_4) t_0 ^4+2 b^2 s_1(m_1^2 , m_2^2 , m_3^2 , m_4^2) t_0^2+16 b^4
\end{equation*}
where
\begin{equation*}
s_i(x_1 , x_2 , x_3 , x_4) = \sum_{1 \leq j_1 <\cdots<j_i \leq 4} x_{j_1}\cdots x_{j_i}
\end{equation*}
is the $i$\textsuperscript{th} elementary symmetric polynomial in the variables $x_1 , x_2 , x_3 , x_4$. Since
\begin{equation*}
(\varpi^{2u} \alpha_1 \alpha_2 \alpha_3 \alpha_4)^2=(\varpi^u \alpha_1^2)(\varpi^u \alpha_2^2)(\varpi^u \alpha_3^2)(\varpi^u \alpha_4^2),
\end{equation*}
we can eliminate the remaining $\alpha_1 , \alpha_2 , \alpha_3 , \alpha_4$ by squaring both sides of the above congruence \eqref{eq:congruence}. We note that both sides have valuation at least $2u$. Indeed, it is clear the left-hand side $\varpi^{2u}\alpha_1 \alpha_2 \alpha_3 \alpha_4$ has valuation $2u$, and since $v_1 + v_2 + v_3 + u \geq 2u$ (since we noted earlier that $v_3 \geq u$), the right-hand side has valuation at least $2u$. Hence
\begin{multline*}
\prod_{j=1}^4 \left[(m_j t_0)^2 + 4 b^2 \right] \equiv \left[ s_4(m_1 , m_2 , m_3 , m_4)t_0^4 + 2 b^2 s_1(m_1^2 , m_2^2 , m_3^2 , m_4^2) t_0^2 + 16 b^4 \right]^2\\ \pmod{\mathfrak{p}^{v_1 + v_2 + v_3 + 3u}}.
\end{multline*}
It is easy to see that the left-hand side is
\begin{equation*}
\begin{split}
  & s_4(m_1^2 , m_2^2 , m_3^2 , m_4^2)t_0^8\\
+ & 4 b^2 s_3(m_1^2 , m_2^2 , m_3^2 , m_4^2) t_0^6 \\
+ & 16 b^4 s_2(m_1^2 , m_2^2 , m_3^2 , m_4^2) t_0^4 \\
+ & 64 b^6 s_1(m_1^2 , m_2^2 , m_3^2 , m_4^2) t_0^2 \\
+ & 256 b^8 ,
\end{split}
\end{equation*}
and by a direct calculation, the right-hand side equals
\begin{equation*}
\begin{split}
  & s_4(m_1^2 , m_2^2 , m_3^2 , m_4^2 ) t_0^8\\
+ & 4 b^2 s_1(m_1^2 , m_2^2 , m_3^2 , m_4^2) s_4(m_1 , m_2 ,m_3 , m_4) t_0^6\\
  + &  4 b^4 \left[ 8 s_4(m_1 , m_2 , m_3 , m_4)+s_1(m_1^2 , m_2^2 , m_3^2, m_4^2)^2\right] t_0^4\\
  + &  64 b^6 s_1(m_1^2 , m_2^2 , m_3^2 , m_4^2)t_0^2\\
  + & 256 b^8.
\end{split}
\end{equation*}
The difference between these two expressions is
\begin{equation}\label{expression}
\begin{split}
  & 4 b^2 \left[ s_3(m_1^2 , m_2^2 , m_3^2, m_4^2 )-s_1(m_1^2 , m_2^2, m_3^2, m_4^2) s_4(m_1 , m_2 , m_3 , m_4) \right] t_0^6\\
+ & 4 b^4 \left[ 4 s_2(m_1^2 , m_2^2 , m_3^2 , m_4^2)-8 s_4(m_1 , m_2 , m_3 , m_4)- s_1(m_1^2 , m_2^2 , m_3^2 , m_4^2)^2 \right] t_0^4.
\end{split}
\end{equation}
To simplify this further, we use that $m_1 + m_2 = m_3 + m_4 $. Let $y_1 =  m_1 m_2$, $y_2 = m_3 m_4$ and $y_3 = m_1 + m_2 = m_3 + m_4$. Then
\begin{equation*}
\begin{split}
  s_3(m_1^2 , m_2^2 , m_3^2 , m_4^2) & = y_1^2 y_3^2 + y_2^2 y_3^2 - 2y_1^2 y_2 -2 y_1 y_2^2\\
  s_1(m_1^2 , m_2^2 , m_3^2 , m_4^2) & = 2 (y_3^2 - y_1 - y_2)\\
  s_4(m_1 , m_2 , m_3 , m_4)        & = y_1 y_2\\
  s_2(m_1^2 , m_2^2 , m_3^2 , m_4^2) & = y_1^2 + y_2^2 +4 y_1 y_2 + y_3^4 - 2 y_1 y_3^2 - 2 y_2 y_3^2.
\end{split}
\end{equation*}
Here we use that $m_1^2 + m_2^2 = y_3^2 - 2 y_1$ and $m_3^2 + m_4^2 = y_3^2 -2 y_2$. The first term in \eqref{expression} is therefore $4 b^2 t_0^6$ times
\begin{equation*}
\begin{split}
  & y_1^2 y_3^2 + y_2^2 y_3^2 - 2 y_1^2 y_2 -2 y_1 y_2^2-2(y_3^2 -y_1 - y_2) y_1 y_2 \\
= &(y_1 - y_2)^2 y_3^2 \\
= &(m_1 m_2 - m_3 m_4)^2(m_1 + m_2)^2,
\end{split}
\end{equation*}
and the second term in \eqref{expression} is $4 b^2 t_0^4$ times
\begin{equation*}
4(y_1^2 + y_2^2 + 4 y_1 y_2 + y_3^4 - 2 y_1 y_3^2 -2 y_2 y_3^2)-8 y_1 y_2 -(2(y_3^2 -y_1 -y_2))^2 =0.
\end{equation*}
Since $v(m_1 + m_2)=0$, we conclude that $2 v_{\mathbf{m}} \geq v_1 + v_2 + v_3 + 3u$. The proof is complete once we note that
\begin{equation*}
v_3 = v(t_0^2(m_1 m_2 - m_3 m_4)+ \varpi^u(\alpha_1 \alpha_2 - \alpha_3 \alpha_4 )) \geq \min \left\{ v_{\mathbf{m}},u+ v(\alpha_1 \alpha_2 - \alpha_3 \alpha_4) \right\}.
\end{equation*}\\
\noindent
(ii) For the second part, assume for the sake of contradiction that $v(\alpha_1 + \alpha_2 - \alpha_3 - \alpha_4)> v(\alpha_1^{-1} + \alpha_2^{-1} - \alpha_3^{-1} - \alpha_4^{-1})$. Since
\begin{equation*}
\alpha_1^{-1} + \alpha_2^{-1} - \alpha_3^{-1} - \alpha_4^{-1} = \frac{(\alpha_1 + \alpha_2)(\alpha_3 \alpha_4 - \alpha_1 \alpha_2 )}{\alpha_1 \alpha_2 \alpha_3 \alpha_4}+ \frac{\alpha_1 + \alpha_2 - \alpha_3 - \alpha_4}{ \alpha_3 \alpha_4},
\end{equation*}
and $\alpha_1 ,\alpha_2 , \alpha_3 , \alpha_4 \in \mathcal{O}^\times$, we deduce that
\begin{equation*}
v(\alpha_1^{-1} + \alpha_2^{-1} - \alpha_3^{-1} - \alpha_4^{-1})= v(\alpha_1 + \alpha_2)+ v(\alpha_1 \alpha_2 - \alpha_3 \alpha_4).
\end{equation*}
Hence $v(\alpha_1 + \alpha_2), v(\alpha_1 \alpha_2 - \alpha_3 \alpha_4)< v(\alpha_1 + \alpha_2 - \alpha_3 - \alpha_4)$. In particular, $v(\alpha_1 + \alpha_2 + \alpha_3 + \alpha_4)= v(\alpha_1 + \alpha_2)$. From the proof of part (i), we have
\begin{equation*}
t_0^2(m_1 m_2 - m_3 m_4) \equiv \varpi^{u}(\alpha_1 \alpha_2 - \alpha_3 \alpha_4) \pmod{\mathfrak{p}_F^{v_1 +v_2 +u}}
\end{equation*}
where $v_1 = v(\alpha_1 + \alpha_2 - \alpha_3 - \alpha_4)$, and $v_2 = v(\alpha_1 + \alpha_2 + \alpha_3 + \alpha_4)= v(\alpha_1 +\alpha_2)$. Since $v_1 + v_2 + u \geq v_1 +u > u + v(\alpha_1 \alpha_2 - \alpha_3 \alpha_4)$, we deduce that $v_{\mathbf{m}}=u + v(\alpha_1 \alpha_2 - \alpha_3 \alpha_4)$. By part (i), we therefore have
\begin{equation*}
2(u + v(\alpha_1 \alpha_2 - \alpha_3 \alpha_4)) \geq v(\alpha_1 + \alpha_2 - \alpha_3 - \alpha_4)+ v(\alpha_1 + \alpha_2 ) + u + v(\alpha_1 \alpha_2 - \alpha_3 \alpha_4)+3u.
\end{equation*}
Hence, using $v(\alpha_1 + \alpha_2 - \alpha_3 - \alpha_4)>v(\alpha_1 \alpha_2 - \alpha_3 \alpha_4)$, we find that
\begin{equation*}
  0\geq v(\alpha_1 + \alpha_2 - \alpha_3 - \alpha_4)+v(\alpha_1 + \alpha_2)-v(\alpha_1 \alpha_2 - \alpha_3 \alpha_4)+2u>v(\alpha_1 + \alpha_2)+2u
\end{equation*}
which is clearly a contradiction.\\

\noindent
(iii) For the last part, we use the inequality from part (i) considering the two possibilites for the minimum of the right-hand side separately. If $v_{\mathbf{m}} \leq  u + v(\alpha_1 \alpha_2 - \alpha_3 \alpha_4)$, we get
\begin{equation*}
v(\alpha_1 + \alpha_2 - \alpha_3 -\alpha_4) \leq v_{\mathbf{m}}-3u - v(\alpha_1 + \alpha_2+ \alpha_3 +\alpha_4 ) \leq v_{\mathbf{m} }-2u.
\end{equation*}
If $v_{\mathbf{m}}>u + v(\alpha_1 \alpha_2 - \alpha_3 \alpha_4)$, we instead use the congruence
\begin{equation*}
t_0^2(m_1 m_2 - m_3 m_4) \equiv \varpi^u( \alpha_1 \alpha_2 - \alpha_3 \alpha_4) \pmod{\mathfrak{p}_F^{v_1 +v_2 + u}}
\end{equation*}
where $v_1 = v(\alpha_1 + \alpha_2 - \alpha_3 -\alpha_4)$, and $v_2 =v(\alpha_1 + \alpha_2 + \alpha_3 + \alpha_4)$. By our assumption, the right-hand side has strictly smaller valuation than the left-hand side. This forces $u + v(\alpha_1 \alpha_2 - \alpha_3 \alpha_4) \geq v_1 + v_2 +u$ so, using part (i), we have
\begin{equation*}
2 v_{\mathbf{m}} \geq 2(v(\alpha_1 + \alpha_2 - \alpha_3 - \alpha_4)+ v(\alpha_1 + \alpha_2 + \alpha_3 + \alpha_4))+4u \geq 2 v(\alpha_1 + \alpha_2 -\alpha_3 - \alpha_4)+4u.
\end{equation*}
This completes the proof.
\end{proof}

\noindent
We can now prove part (iii) of Proposition \ref{final_estimates} which is the only thing left to do. We have $\Phi(z) = \Phi(0) -\varpi^{\frac{3u}{2}} \sum_{k \geq 1} \frac{1}{k} a_{k-1} z^k$ where $v(a_0) \leq v(a_k)$ for all $k$, and $v(a_0) \leq v_{\mathbf{m}}-2u$. Hence $\int_{z \in \mathfrak{p}_F} \psi(\varpi^{- n/2} \Phi(z))d z$ vanishes for $-n_2 + \frac{3u}{2}+v_{\mathbf{m}}-2u < -1$, or equivalently, $v_{\mathbf{m}}<\frac{n}{2} + \frac{u}{2}-1$. By Lemma \ref{first_estimates} with $v_0 = \frac{n}{2} + \frac{u}{2}-1$, we have
\begin{equation*}
\mathcal{N}_a(\kappa W_b^{0,u}) \prec 1+ p^{n/2-u-\max \left\{ v_0 ,2u \right\}}\lVert \kappa W_b^{0,u} \rVert _\infty^4.
\end{equation*}
Since $u< \frac{n}{4}$, $v_0 \geq 2u$. Moreover, we see from Table \ref{heuristic_balanced} that  $\lVert \kappa W_b^{0,u} \rVert _\infty \ll p^{\frac{u}{4}}$ so we get the upper bound
\begin{equation*}
\mathcal{N}_a(\kappa W_b^{0,u}) \prec 1+ p^{n_2 -u-n_2 -\frac{u}{2}+1+u} \ll 1
\end{equation*}
as desired. The proof of the $L^4$-norm bound is complete.

\section*{Acknowledgement}
  I am grateful to my supervisor Paul Nelson for suggesting this project and for many helpful discussions. I would also like to thank Edgar Assing who read an early draft of this paper and gave many useful comments. I also thank the following people for their comments: Anshul Adve, Valentin Blomer, Jack Buttcane, Trajan Hammonds, Yueke Hu, Peter Humphries, Haseo Ki, Rizwanur Khan, Simon Marshall and Abhishek Saha. I am also grateful to the anonymous referee for a careful reading of my manuscript that led to numerous clarifications and corrections. 

\subsection*{Statement on the use of AI}
When the first version of this paper was written, Gemini 3 was used to make minor language improvements. In the revision process, ChatGPT 6 has been used for proofreading my revisions to Section \ref{cubic_phase} and the proof of Theorem \ref{whittaker_formulas} which also led to some simplifications in the proofs of Lemma \ref{easy_cubic_integral}, Proposition \ref{airy} and Theorem \ref{whittaker_formulas}, Subcase 3.2.2 (especially the choice of basis $e_1 =(1,1)$ and $e_2 =(1,-1)$ in the case $ E = F \times F$). Moreover, the conditions in Lemma \ref{easy_cubic_integral} and Lemma \ref{lem:valuation-lemma-3} have been optimised using suggestions from ChatGPT 6 so that in Theorem \ref{whittaker_formulas}, we only need $n \geq 16 v(3)+4$ for $p=3$, instead of the condition $n \geq 20 v(3)+10$ that I originally found while revising the paper. 
  
\bibliographystyle{amsalpha}
\bibliography{refs2} 

@phdthesis{assing_thesis,
title = {Local {A}nalysis of {W}hittaker {N}ew {V}ectors and {G}lobal {A}pplications},
author = {Assing, E.},
year = {2019},
month = {January},
note  = {Available at \url{https://www.math.uni-bonn.de/people/assing/Thesis/Edgar-thesis.pdf}},
 school = {University of Bristol},
  type  = {Ph{D} thesis}}

@book{Bump_1997,
 place={Cambridge}, 
series={Cambridge Studies in Advanced Mathematics}, 
title={Automorphic {F}orms and {R}epresentations}, 
publisher={Cambridge University Press}, 
author={Bump, D.},
 year={1997},
 collection={Cambridge Studies in Advanced Mathematics}}

@article{Ki2023L4norms,
    author  = {Ki, H.},
    title   = {{$L^4$-norms and sign changes of {M}aass forms}},
    journal = {Ann. Math.},
    volume = {204},
    number = {2},
    year = {2026},
    pages = {741-775}
}

@article{Ghosh_2013,
   title={Nodal {D}omains of {M}aass {F}orms {I}},
   volume={23},
   ISSN={1420-8970},
   url={http://dx.doi.org/10.1007/s00039-013-0237-4},
   DOI={10.1007/s00039-013-0237-4},
   number={5},
   journal={Geom. Funct. Anal.},
   publisher={Springer Science and Business Media LLC},
   author={Ghosh, A. and Reznikov, A. and Sarnak, P.},
   year={2013},
   month={jul}, pages={1515–1568} }

@misc{Nelson-notes,
title ={Notes on {H}aseo {K}i's ${L}^4$-norm bound},
author = {Nelson, P. D.},
note = {\url{https://ultronozm.github.io/math/20230331T100654--haseo-ki-l4-bound.html}},
year = {2023}}

@book{goldfeld_hundley,
title={Automorphic {R}epresentations and {L}-Functions for the {G}eneral {L}inear {G}roup: Volume 1},
author={Goldfeld, D. and Hundley, J.},
ISBN={9781139500135},
series={Cambridge Studies in Advanced Mathematics},
year={2011},
publisher={Cambridge University Press},
edition={Special {ICM} {S}eoul edition}}

@book{silverman,
title={The {A}rithmetic of {E}lliptic {C}urves},
author = {Silverman, J. H.},
series = {Graduate Texts in Mathematics},
DOI = {https://doi.org/10.1007/978-0-387-09494-6},
publisher = {Springer New York, NY},
edition = {1},
ISBN = {978-1-4419-1858-1},
year = {1986}
}

@article {assing_on_the_size,
    AUTHOR = {Assing, E.},
     TITLE = {On the size of {$p$}-adic {W}hittaker functions},
   JOURNAL = {Trans. Amer. Math. Soc.},
  FJOURNAL = {Transactions of the American Mathematical Society},
    VOLUME = {372},
      YEAR = {2019},
    NUMBER = {8},
     PAGES = {5287--5340},
      ISSN = {0002-9947,1088-6850},
   MRCLASS = {11F70 (11L40 11S80)},
  MRNUMBER = {4014277},
MRREVIEWER = {Ivan\ Mati\'c},
       DOI = {10.1090/tran/7685},
       URL = {https://doi.org/10.1090/tran/7685},
}

@article{PMIHES_2010__111__171_0,
     author = {Michel, P. and Venkatesh, A.},
     title = {The subconvexity problem for {GL2}},
     journal = {Publ. Math. IHÉS},
     pages = {171--271},
     publisher = {Springer-Verlag},
     volume = {111},
     year = {2010},
     doi = {10.1007/s10240-010-0025-8},
     mrnumber = {2653249},
     language = {en},
     url = {https://www.numdam.org/articles/10.1007/s10240-010-0025-8/}
}

@article{humphries-khan,
author = {Humphries, P. and Khan, R.},
title = {${L}^p$-norm bounds for automorphic forms via spectral reciprocity},
journal = {Proc. London Math. Soc.},
volume = {130},
number = {6},
pages = {e70061},
doi = {https://doi.org/10.1112/plms.70061},
url = {https://londmathsoc.onlinelibrary.wiley.com/doi/abs/10.1112/plms.70061},
eprint = {https://londmathsoc.onlinelibrary.wiley.com/doi/pdf/10.1112/plms.70061},
year = {2025}
}

@article{luo,
    author = {Luo, W.},
    title = {L4-Norms of the {D}ihedral {M}aass {F}orms},
    journal = {Int. Math. Res. Not.},
    volume = {2014},
    number = {8},
    pages = {2294-2304},
    year = {2013},
    month = {01},
     issn = {1073-7928},
    doi = {10.1093/imrn/rns298},
    url = {https://doi.org/10.1093/imrn/rns298},
    eprint = {https://academic.oup.com/imrn/article-pdf/2014/8/2294/18916735/rns298.pdf},
}

@article{blomer,
author = {Blomer, V.},
title = {On the 4-norm of an automorphic form},
journal = {J. Eur. Math. Soc. },
volume = {15},
number = {5},
pages = {1825–1852},
year = {2013},
doi = {10.4171/JEMS/405}
}

@article{buttcane_khan,
author = {Buttcane, J. and Khan, R.},
title = {${L}^4$-norms of {H}ecke newforms of large level},
journal = {Math. Ann.},
volume = {362},
pages = {699-715},
year = {2025},
doi = {https://doi.org/10.1007/s00208-014-1142-3}
}

@article{hu_nelson_saha,
author = {Hu, Y. and Nelson, P. D. and Saha, A.},
title = {Some analytic aspects of automorphic forms on $\mathrm{GL}_2$ of minimal type},
journal = {Comment. Math. Helv.},
volume = {94},
pages = {767-801},
year = {2019},
doi = {https://doi.org/10.4171/CMH/473}
}

@article{Tunnell1978,
author = {Tunnell, J. B.},
journal = {Invent. Math.},
pages = {179-200},
title = {On the {L}ocal {L}anglands {C}onjecture for {GL(2)}.},
url = {http://eudml.org/doc/142559},
volume = {46},
year = {1978}
}

@book{Godement1970,
author = {Godement, R.},
title = {Notes on {J}acquet-{L}anglands' theory},
year = {1970},
publisher = {The Institute for Advanced Study}
}

@book{bushnell_henniart,
author = {Bushnell, C. J. and Henniart, G.},
title = {The {L}ocal {L}anglands {C}onjecture for {GL}(2)},
series = {Grundlehren der mathematischen Wissenschaften},
publisher = {Springer Berlin, Heidelberg},
edition = {1},
year = {2006},
doi = {https://doi.org/10.1007/3-540-31511-X}
}

@article{schmidt,
author = {Schmidt, R.},
title = {Some remarks on local newforms for {GL(2)}},
journal = {J. Ramanujan Math. Soc.},
volume = {17},
year = {2002},
pages = {115-147}
}

@article{Watson2008Rankin,
    author  = {Watson, T. C.},
    title   = {Rankin {T}riple {P}roducts and {Q}uantum {C}haos},
    journal = {arXiv preprint arXiv:0810.0425},
    archivePrefix = {arXiv},
    eprint  = {0810.0425},
    primaryClass = {math.NT},
    year    = {2008}
}

@article{saha,
author = {Saha, A.},
title = {Hybrid sup-norm bounds for {M}aass newforms of powerful level},
volume = {11},
journal = {Algebra Number Theory},
number = {5},
publisher = {MSP},
pages = {1009 -- 1045},
year = {2017},
doi = {10.2140/ant.2017.11.1009},
URL = {https://doi.org/10.2140/ant.2017.11.1009}
}

@article{assing_toma_2024,
      title={The orbit method in number theory through the sup-norm problem for $\mathrm{GL}(2)$}, 
      author={Assing, E. and Toma, R.},
      year={2024},
      journal = {arXiv preprint arXiv:2404.16480},
      eprint={2404.16480},
      archivePrefix={arXiv},
      primaryClass={math.NT},
      url={https://arxiv.org/abs/2404.16480},
      note = {To appear in Essent. Number Theory}
}

@article{Buttcane_Khan_2017,
title={On the fourth moment of {H}ecke–{M}aass forms and the random wave conjecture},
volume={153},
DOI={10.1112/S0010437X17007199},
number={7},
journal={Compos. Math.},
author={Buttcane, J. and Khan, R.},
year={2017},
pages={1479–1511}}

@article{saha_2024,
title={The {M}anin constant and the modular degree},
author={\v{C}esnavi\v{c}ius, K. and Neururer, M. and Saha, A.},
journal = {J. Eur. Math. Soc. },
volume={26},
number={2},
pages={573-637},
year={2024},
DOI={10.4171/JEMS/1367}}

@article{marshall,
author = {Marshall, S.},
title = {Local bounds for ${L}^p$ norms of {M}aass forms in the level aspect},
volume = {10},
journal = {Algebra Number Theory},
number = {4},
publisher = {MSP},
pages = {803 -- 812},
year = {2016},
doi = {10.2140/ant.2016.10.803},
URL = {https://doi.org/10.2140/ant.2016.10.803}
}

@article{Hu_Saha_2020,
title={Sup-norms of eigenfunctions in the level aspect for compact arithmetic surfaces, {II}: newforms and subconvexity},
volume={156},
DOI={10.1112/S0010437X20007460},
number={11},
journal={Compos. Math.},
author={Hu, Y. and Saha, A.},
year={2020},
pages={2368–2398}}

@article{Iwaniec_Sarnak_1995,
 ISSN = {0003486X, 19398980},
 URL = {http://www.jstor.org/stable/2118522},
 author = {Iwaniec, H. and Sarnak, P.},
 journal = {Ann. Math.},
 number = {2},
 pages = {301--320},
 title = {${L}^{\infty}$ {N}orms of {E}igenfunctions of {A}rithmetic {S}urfaces},
 urldate = {2026-01-16},
 volume = {141},
 year = {1995}
}

@article{BlomerHolowinsky2010,
  author    = {Blomer, V. and Holowinsky, R.},
  title     = {Bounding sup-norms of cusp forms of large level},
  journal   = {Invent. math.},
  year      = {2010},
  volume    = {179},
  number    = {3},
  pages     = {645--681},
  doi       = {10.1007/s00222-009-0228-0},
  publisher = {Springer}
}

@article{LesterMatomakiRadziwill2018,
  author  = {Lester, S. and Matom{\"a}ki, K. and Radziwi{\l}{\l}, M.},
  title   = {Small scale distribution of zeros and mass of modular forms},
  journal = {J. Eur. Math. Soc.},
  year    = {2018},
  volume  = {20},
  number  = {7},
  pages   = {1595--1627},
  doi     = {10.4171/JEMS/796}
}

@article{BlomerCorbett2021,
  author    = {Blomer, V. and Corbett, A.},
  title     = {A symplectic restriction problem},
  journal   = {Math. Ann.},
  year      = {2021},
  volume    = {381},
  number    = {1},
  pages     = {589--623},
  doi       = {10.1007/s00208-020-02113-5},
  publisher = {Springer}
}

@article{Nelson2018,
  author  = {Nelson, P. D.},
  title   = {Microlocal lifts and quantum unique ergodicity on $\mathrm{GL}_2(\mathbb{Q}_p)$},
  journal = {Algebra Number Theory},
  year    = {2018},
  volume  = {12},
  number  = {9},
  pages   = {2033--2064},
  doi     = {10.2140/ant.2018.12.2033}
}

@article{humphries2018,
 author = {Humphries, P.},
 title = {Equidistribution in shrinking sets and ${L}^4$-norm bounds for automorphic forms},
 journal = {Math. Ann.},
 year = {2018},
 volume = {371},
 pages = {1497-1543},
 doi = {10.1007/s00208-018-1677-9}
}

@phdthesis{spinu2003,
title = {The ${L}^4$-norm of {E}isenstein series},
author = {Spinu, F.},
year = {2003},
school = {Princeton University},
type  = {PhD thesis}}

@article{hoffstein-lockhart,
title = {Coefficients of {M}aass forms and the {S}iegel {Z}ero},
author = {Hoffstein, J. and Lockhart, P.},
journal = {Ann. Math.},
year = {1994},
volume = {140},
number = {1},
pages = {161-176}}

@article{kim03,
title = {Functoriality for the {E}xterior {S}quare of ${GL}_4$ and the {S}ymmetric {F}ourth of ${GL}_2$},
author = {Kim, H. H.},
journal = {J. Am. Math. Soc.},
volume = {16},
number = {1},
pages = {139-183},
notes = {With Appendix 1 by Dinakar Ramakrishnan
          and Appendix 2 by Henry H. Kim and Peter Sarnak},
year = {2003}}

@article{iwaniec90,
title = {Small eigenvalues for {L}aplacian for $\Gamma_0({N})$},
author = {Iwaniec, H.},
journal = {Acta Arith.},
volume = {56},
number = {1},
pages = {65-82},
year = {1990}}

\end{document}